\documentclass[10pt,english,letter, reqno]{amsart}

\usepackage[T1]{fontenc}

\usepackage{textgreek}
\def\api{\ifmmode\pi\else\textpi\fi}

\usepackage{amssymb,url,xspace,amsthm,adjustbox}
\usepackage[verbose,colorlinks,backref]{hyperref}
\usepackage[alphabetic]{amsrefs}
\usepackage{amsmath}

\usepackage{longtable}
\usepackage{multicol}
\usepackage{mathtools}
\usepackage{graphicx}

\usepackage{float}

\author{Clifton Cunningham}
\address{University of Calgary}
\email{clifton@automorphic.ca}
\thanks{Clifton Cunningham gratefully acknowledges the support of NSERC Discovery Grant RGPIN-2020-05220}

\author[A. Fiori]{Andrew Fiori}
\address{University of Lethbridge}
\email{andrew.fiori@uleth.ca}
\thanks{Andrew Fiori thanks and acknowledges both the University of Lethbridge for their financial support as well as the support of NSERC Discovery Grant RGPIN-2020-05316.}

\author[Q. Zhang]{Qing Zhang}
\address{Korea Advanced Institute of Science and Technology}
\email{qingzhang0@gmail.com}
\thanks{Qing Zhang would like to thank the support NSFC grant 11801577, a postdoc fellowship from the Pacific Institute for the Mathematical Sciences (PIMS)}

\usepackage{datetime}
\usepackage[T1]{fontenc}
\usepackage[margin=3.5cm]{geometry}
\usepackage{chngcntr}

\usepackage{amssymb}
\usepackage{mathrsfs,stmaryrd}
\usepackage{yfonts, bbm}
\usepackage{enumitem}

\usepackage{hyperref}
\hypersetup{
  colorlinks   = true, 
  urlcolor     = blue, 
  linkcolor    = blue, 
  citecolor   = purple 
}

\usepackage{tikz}
\usetikzlibrary{shapes,arrows,calc,matrix}
\usepackage{tikz-cd}
\usepackage{todonotes}
\usepackage{color}

\usepackage{amsmath}
\usepackage{lipsum}
\usepackage{setspace}

\theoremstyle{plain}
      \newtheorem{theorem}{Theorem}[section]
      \newtheorem{proposition}[theorem]{Proposition}
      \newtheorem{lemma}[theorem]{Lemma}
      \newtheorem{corollary}[theorem]{Corollary}

      \theoremstyle{definition}
      \newtheorem{definition}{Definition}

      \newtheorem{remark}[theorem]{Remark}

      \newtheorem{conjecture}{Conjecture}
      \newtheorem*{conjecture*}{Conjecture}

\counterwithin{table}{subsection}

\renewcommand{\o}[1]{\buildrel _{\circ} \over {#1}}

\newcommand{\ZZ}{{\mathbb{Z}}}

\newcommand{\CC}{{\mathbb{C}}}

\newcommand{\QQ}{{\mathbb{Q}}}

\newcommand{\RR}{{\mathbb{R}}}

\DeclareMathOperator{\GL}{GL}
\DeclareMathOperator{\Sym}{Sym}

\DeclareMathOperator{\SL}{SL}

\DeclareMathOperator{\PGL}{PGL}
\DeclareMathOperator{\Sp}{Sp}

\DeclareMathOperator{\GSpin}{GSpin}

\DeclareMathOperator{\SO}{SO}

\DeclareMathOperator{\GSO}{GSO}

\newcommand{\Lgroup}[1]{{\hskip-2 pt \,^L\hskip-1pt{#1}}}
\newcommand{\dualgroup}[1]{{\widehat{#1}}}

\DeclareMathOperator{\Lift}{Lift}

\newcommand{\Frob}{{\operatorname{Fr}}}
\DeclareMathOperator{\ord}{ord}

\DeclareMathOperator{\id}{id}

\DeclareMathOperator{\trace}{trace}

\DeclareMathOperator{\Ad}{Ad}

\newcommand{\abs}[1]{{\vert #1 \vert}}
\newcommand{\ceq}{{\, :=\, }}
\newcommand{\tq}{{\ \vert\ }}
\newcommand{\iso}{{\ \cong\ }}

\DeclareMathOperator{\diag}{diag}

\newcommand{\Perv}{\mathsf{Per}}

\newcommand{\Loc}{\mathsf{Loc}}
\newcommand{\Rep}{\mathsf{Rep}}

\newcommand{\Ev}{\operatorname{\mathsf{E}\hskip-1pt\mathsf{v}}}

\newcommand{\NEvs}{\operatorname{\mathsf{N}\hskip-1pt\mathsf{E}\hskip-1pt\mathsf{v}\hskip-1pt\mathsf{s}}}

\newcommand{\IC}{{\mathcal{I\hskip-1pt C}}}

\newcommand{\Ft}{\operatorname{\mathsf{F\hskip-1pt t}}}

\makeatletter
\newcommand{\labitem}[2]{
\def\@itemlabel{\textbf{#1}}
\item
\def\@currentlabel{#1}\label{#2}}
\makeatother

\newcommand{\1}{{\mathbbm{1}}}

\newcommand{\Ind}{\text{Ind}}

\newcommand{\C}{\mathbb{C}}

\newcommand{\R}{\mathbb{R}}

\newcommand{\Aut}{\text{Aut}}

\newcommand{\res}{{\operatorname{res}}}

\newcommand{\Lie}{\operatorname{Lie}}

 \newcommand{\BC}{{\mathbb {C}}}

 \newcommand{\CO}{{\mathcal {O}}}

\newcommand{\ABV}{{\mbox{\raisebox{1pt}{\scalebox{0.5}{$\mathrm{ABV}$}}}}}

\newcommand{\Fr}{{\rm {Fr}}}

     \newcommand{\ra}{\rightarrow}

\newcommand{\da}{{\hat{\alpha}}}
\newcommand{\db}{{\hat{\beta}}}
\newcommand{\wh}{\widehat}
\newcommand{\St}{\mathrm{St}}

\newcommand{\wpair}[1]{\left\{{#1}\right\}}
 
\newcommand{\wt}{\widetilde} 
\newcommand{\bpm}{\begin{pmatrix}}
\newcommand{\epm}{\end{pmatrix}}

\theoremstyle{plain}

\begin{document}

\setcounter{tocdepth}{2}

\title{Toward the endoscopic classification of unipotent representations of $p$-adic $G_2$}

\date{\today}                                           

\maketitle

\begin{abstract}
We begin this paper by reviewing the Langlands correspondence for unipotent representations of the exceptional group of type $G_2$ over a $p$-adic field $F$ and present it in an explicit form. 
Then we compute all ABV-packets, as defined in \cite{CFMMX} following ideas from Vogan's 1993 paper {\it The local Langlands Conjecture}, and prove that these packets satisfy properties derived from the expectation that they are generalized A-packets. 
We attach distributions to ABV-packets for $G_2$ and its endoscopic groups and study a geometric endoscopic transfer of these distributions. 
This paper builds on earlier work by the same authors.
\end{abstract}

\tableofcontents

\setcounter{section}{-1}

\section{Introduction}

\subsection{Background and motivation}

Local Arthur packets are finite sets $\Pi_\psi(G(F))$ of irreducible admissible representations of a $p$-adic $G(F)$,  
characterized by endoscopic character identities for symplectic and quasi-split orthogonal groups $G$ (and their inner twists over $F$) in \cite{Arthur:book}.
Arthur packets generalize tempered $L$-packets and they remedy certain drawbacks of $L$-packets associated with stable distributions.
The description of Arthur packets in \cite{Arthur:book}, relies heavily on the trace formula.
These notions have been extended to some other groups, including unitary groups in \cite{Mok:Unitary},
 but not yet to exceptional groups; however, some candidates for local Arthur packets for $G_2(F)$ were proposed in \cite{GGJ}, \cite{GG} and \cite{Gan-Savin}.
In this paper we use the microlocal geometry of the moduli space of unramified Langlands parameters to calculate ABV-packets, as defined in \cite{CFMMX} building on ideas from \cite{Vogan:Langlands}, for all unipotent representations of $G_2(F)$ and show that these packets are compatible with endoscopy in a very precise sense.

Vogan's work \cite{Vogan:Langlands} shows that there are advantages to studying the representation theory of $G(F)$ together with that of its pure inner forms $G^\delta(F)$ simultaneously; in this way one passes from $L$-packets $\Pi_\phi(G(F))$ to "pure" $L$-packets $\Pi^\text{pure}_\phi(G)$ and from Arthur packets $\Pi_\psi(G(F))$ to "pure"  Arthur packets $\Pi^\text{pure}_\psi(G)$.  
As explained in \cite{CFMMX}, when Arthur's work is viewed from this perspective it determines a function $\Pi^\text{pure}_\psi(G)\to \Rep(A_\psi)$ where $A_\psi$ is the component group of $Z_{\dualgroup{G}}(\psi)$.

A geometric construction of Arthur packets for any $p$-adic group is proposed in \cite{CFMMX} building on ideas introduced by \cite{Vogan:Langlands} and also by an adaptation of ideas from \cite{ABV}. In fact, this geometric construction is more general: for any Langlands parameter $\phi$ for $G(F)$, \cite{CFMMX} constructs a group $A_\phi^{\ABV}$, and proposes a packet $\Pi_{\phi}^{\ABV}(G)$ of irreducible representations of $G(F)$ and its pure inner forms, called the ABV-packet of $\phi$, and a natural map $\Pi_\phi^{\ABV}(G) \to \Rep(A_\phi^{\ABV})$. If $\phi$ happens to be of Arthur type, then $A_\phi^{\ABV}=A_\psi$ for the associated Arthur parameter $\psi$.
The main conjecture of \cite{CFMMX} says that when $\phi$ is of Arthur type then $\Pi_\phi^{\ABV}(G)=\Pi^\text{pure}_\psi(G)$ and the map $\Pi_\phi^{\ABV}(G)\to \Rep(A_\phi^{\ABV})$ should be the same as $\Pi^\text{pure}_\psi(G)\to \Rep(A_\psi)$ when $G$ is a classical group, where $\psi$ is the associated Arthur parameter. The construction in \cite{CFMMX} proposes a concrete way to compute Arthur packets for classical groups, as illustrated by the numerous examples in \cite{CFMMX}, and also provides a method to generalize Arthur packets to exceptional groups and non-Arthur type local Langlands parameters.

\subsection{Fundamental properties of ABV-packets and packet coefficients}
To give some details of our results, we recall that, in \cite{Lusztig:Classification1} Lusztig provided a version of the local Langlands correspondence for unipotent representations and showed that the unipotent representations are classified by unramified parameters. Using that, in this paper we compute the ABV-packets defined in \cite{CFMMX},  for each unramified local Langlands parameter $\phi$ of $G_2(F)$, and thus we obtain a decomposition
\[
\Pi(G_2(F))_\text{unip}
=
\mathop{\bigcup}\limits_{\phi \in \Phi(G_2(F))_\text{unr}} \Pi^\ABV_{\phi}(G_2(F)),
\]
where $\Pi(G_2(F))_\text{unip}$ is the set of isomorphism classes of irreducible unipotent representations of $G_2(F)$. For every ABV-packet $\Pi^\ABV_{\phi}(G_2(F))$, the general, canonical, geometric constructions in \cite{CFMMX} also give a group $A^\ABV_\phi$ and a function
\[
\begin{array}{rcl}
\Pi^\ABV_\phi(G_2(F)) &\to& \Rep(A^\ABV_\phi) \\
\pi &\mapsto& \langle \ , \pi\rangle
\end{array}
\]
and thus a function $\langle \ , \ \rangle :  A^\ABV_\phi \times \Pi^\ABV_\phi(G_2(F)) \to {\bar\QQ}$; we refer to this function as the {\it ABV-packet coefficients} for $\phi$.
%

In this paper we argue that ABV-packet coefficients provide a canonical way to extend the notion of Arthur packets to all unipotent representations of the exceptional $p$-adic group $G_2$, including those that are not of Arthur type.
For instance, in Theorem~\ref{thm:coefficients} we show that the ABV-packet coefficients extend the local Langlands correspondence for unipotent representations of $G_2(F)$:
for all unramified Langlands parameters $\phi : W'_F \to \Lgroup{G_2}$, the following diagram commutes
\[
\begin{tikzcd}
\Pi^\ABV_\phi(G_2(F)) \arrow{rr}{} && \widehat{A^\ABV_\phi} \\
\Pi_\phi(G_2(F)) \arrow[>->]{u} \arrow{rr}{{}}[swap]{\mathrm{LLC}} && \arrow{u} \widehat{A_\phi}
\end{tikzcd}
\]
where $\widehat{A^\ABV_\phi}$ (resp. $\widehat{A_\phi}$) denotes the set of characters of irreducible representations of $A^\ABV_\phi$ (resp. $A_\phi$).
Here the local Langlands correspondence is normalized such that if $\pi\in \Pi_\phi(G_2(F))$ is generic or spherical then the corresponding representation of $A_\phi$ is trivial.
We give other fundamental properties of ABV-packet coefficients in Theorem~\ref{thm:coefficients}.
For instance, Part \ref{G2:Aubert} of Theorem~\ref{thm:coefficients} shows that for every unramified Langlands parameter $\phi$ there is an unramified Langlands parameter ${\hat \phi}$, with the same infinitesimal parameter as ${\hat \phi}$, such that the Aubert involution defines a bijection 
\[
\Pi^\ABV_{\phi}(G_2(F)) \to \Pi^\ABV_{\hat \phi}(G_2(F)).
\]

\subsection{Distributions attached to ABV-packets}

We study distributions defined by ABV-packet coefficients.
To do this we introduce an algebraic group $\mathcal{S}^\ABV_\phi$ contained in $Z_{\dualgroup{G}}(\phi)$ such that $A^\ABV_\phi = \pi_0(\mathcal{S}^\ABV_\phi)$ and $\mathcal{S}^\ABV_\phi = Z_{\dualgroup{G_2}}(\psi)$ if $\phi$ is of Arthur type $\psi$.
For $s\in \mathcal{S}^\ABV_\phi$, we set
\begin{equation}\label{eqn:introTheta}
\Theta^{}_{\phi,s} \ceq \sum_{\pi \in\Pi^\ABV_{\phi_\psi}(G_2(F))} (-1)^{\dim(\phi)-\dim(\pi)} \langle s,\pi\rangle\ \Theta_{\pi},
\end{equation}
as in Definition~\ref{def:Thetaphis}, where we identify $s$ with its image under $\mathcal{S}^\ABV_\phi \to A^\ABV_\phi$ and where $\Theta_{\pi}$ is the Harish-Chandra distribution character determined by the admissible representation $\pi$.
The terms $\dim(\phi)$ and $\dim(\pi)$ are defined in Section~\ref{ssec:VC}.
We show that if $\phi$ is of Arthur type $\psi$ then
$
(-1)^{\dim(\phi)-\dim(\pi)}\langle s,\pi\rangle= \langle s_\psi\, s,\pi\rangle ,
$
where $s_\psi$ is the image of $\psi(1,-1)\in Z_{\dualgroup{G}_2}(\psi)$; 
in this case, Equation~\eqref{eqn:introTheta} takes the more familiar form
\[
\Theta^{}_{\psi,s} =  \sum_{\pi \in\Pi^\ABV_{\phi_\psi}(G_2(F))}  \langle s_\psi s,\pi\rangle\ \Theta_{\pi}.
\]

In Theorem~\ref{thm:stable} we prove that if $\pi$ is an irreducible unipotent representation of $G_2(F)$ then $\Theta_\pi$ can be expressed in terms of the distributions $\Theta^{}_{\phi,s}$, letting $\phi$ range over unramified $L$-parameters with the same infinitesimal parameter as $\pi$ and letting $s$ range over representatives for the components of $\mathcal{S}^\ABV_\phi$.
This is an important part of the endoscopic classification promised in the title of this paper. 
We also show that the span of the distributions $\Theta^{}_{\phi,s}$, as $s$ ranges over $\mathcal{S}^\ABV_\phi$, coincides with the span of the distributions $\Theta_\pi$, as $\pi$ ranges over $\Pi^\ABV_\phi$, if and only if $\Pi^\ABV_\phi(G_2(F)) \to \widehat{A^\ABV_\phi}$ is bijective. 
In this theorem we also show that if $\Theta^{}_{\phi}\ceq \Theta^{}_{\phi,1}$ is stable when $\phi$ is elliptic (see Section~\ref{ssec:stable}), then $\Theta^{}_{\phi}$ is stable for all $\phi$. 

\subsection{Lifting stable distributions from endoscopic groups}

We show the ABV-packet coefficients $\langle \ , \ \rangle $ are compatible with the theory of endoscopy in the following sense.
Let $(G,s,\xi)$ be an endoscopic triple for $G_2$.
Unlike the group $G_2$, the endoscopic groups $G$ may admit pure inner forms. 
Recall that a pure inner form for $G$ is a cocycle $\delta \in Z^1(F,G)$ and that $\delta$ determines an inner form $G^\delta$ of $G$. 
We calculate the ABV-packets and ABV-packet coefficients for the endoscopic group $G$.
We generalize the definition of the distributions above Equation~\eqref{eqn:introTheta} from $G_2$ to the pure inner forms of endoscopic groups $G$ for $G_2$.
%
In Definition~\ref{def:lifting} we introduce a "lifting" of stable distributions attached to ABV-packets for $G(F)$ to invariant distributions on $G_2(F)$ and prove Theorem~\ref{thm:lifting}:
if $\phi : W'_F\to \Lgroup{G}$ is an unramified Langlands parameter that is $\xi$-conormal, in the sense of Definition~\ref{def:xi-conormal}, then
\begin{equation}\label{eqn:introlifting}
\Lift^{}_{(G,s,\xi)} \Theta^{G}_{\phi} = \Theta^{}_{\xi\circ \phi,s},
\end{equation}
for all $s\in \mathcal{S}^\ABV_\phi$, for $\mathcal{S}^\ABV_\phi$ given in Definition~\ref{def:SABV}.
We show that if $\phi$ is of Arthur type then it is always $\xi$-conormal.
Since $G(F)$ is also an endoscopic group for the inner form $G^\delta(F)$ we also define "lifting" from stable distribution attached to ABV-packets for $G(F)$ to invariant distributions on $G^\delta(F)$ and, in Theorem~\ref{thm:lifting}, we  show that if $\phi : W'_F\to \Lgroup{G}$ is an unramified Langlands parameter that is relevant to $G^\delta(F)$, in the sense of \cite{Borel:Automorphic}, then the lift of $ \Theta^{G}_{\phi}$ is $e(\delta)\Theta^{G^\delta}_{\phi}$, where $e(\delta)$ is the Kottwitz sign.
Conjecture~\ref{conjecture:transfer} predicts that $\Theta^{}_{\xi\circ \phi,s}$ is the Langlands-Shelstad transfer of the stable distribution $\Theta^{G}_{\phi}$.


All this comes together in Theorem~\ref{thm:EC}, in which we show that ABV-packet coefficients are uniquely characterized by the properties given in this paper. 
This result also presents the endoscopic classification of unipotent representations of $G_2(F)$ promised in the title of this paper, stated here only for tempered representations to simplify the exposition. 
Let $\pi$ be a tempered unipotent representation of $G_2(F)$ with Langlands parameter $\phi$.
Every $s \in \mathcal{S}_\psi$ determines an endoscopic triple $(G,s,\xi)$ and a factorization $\phi = \xi\circ\phi^s$, for a Langlands parameter $\phi^s$ for $G$. 
In this case, Theorem~\ref{thm:EC} gives
\begin{equation}\label{eqn:introEC}
\Theta_{\pi} 
= \sum_{(G,s,\xi)} (-1)^{\dim(\phi^s)-\dim(\pi)} \frac{\overline{\langle s,\pi\rangle}}{\abs{Z_{A_\phi}(s)}} \ \Lift^{}_{(G,s,\xi)} \Theta^{G}_{\phi^s},
\end{equation}
where the sum is taken over equivalence classes of endoscopic triples $(G,s,\xi)$ with $s \in \mathcal{S}^\ABV_\phi$ and where we identify $s$ with its image under $\mathcal{S}^\ABV_\phi \to A^\ABV_\phi$ in the calculation of $Z_{A_\phi}(s)$.
This result is generalized from tempered unipotent representations to a wider class of unipotent representations in Theorem~\ref{thm:EC}.

\subsection{Next steps}

Before we can declare that the ABV-packets of this paper are indeed generalized Arthur packets for $G_2(F)$, the following two issues must be addressed:
{\it (i)} we must show that $\Theta^{G}_{\phi}$ is a stable distribution for every unramified Langlands $\phi$ of $G(F)$ of Arthur type, where $G(F)$ is an endoscopic group for $G_2(F)$; and
{\it (ii)} we must show that, if $(G,s,\xi)$ is an endoscopic triple for $G_2(F)$, then
\[
\Theta^{G}_{\phi}(f^s) = \Theta^{}_{\xi\circ \phi,s}(f)
\]
where $f^s\in \mathcal{H}(G(F))$ is the Langlads-Shelstad transfer of $f\in \mathcal{H}(G_2(F))$; in other words, we must show that $\Theta^{}_{\xi\circ \phi,s}$ is the endoscopic lift of $\Theta^{G}_{\phi}$.
In fact, it is a conjecture of Vogan that {\it (i)} is true for all Langlands parameters, not just those of Arthur type. This is elementary for unramified parameters except when $G=G_2$ and, as we show in Theorem~\ref{thm:stable}, in that case it is sufficient to show that $\Theta_{\phi}$ is stable for the four elliptic, tempered, unipotent $L$-packets in Table~\ref{table:elliptic}; see also Remark~\ref{rem:stability}. We are working on a proof of this statement.
We will also prove a generalization of {\it (ii)} that asks only that $\phi$ is $s$-conormal, as in Definition~\ref{def:s-conormal}, again generalizing from the case of Langlands parameters of Arthur type.

\subsection{Geometry and arithmetic}

ABV-packet coefficients are defined by a microlocal analysis of the moduli space of unramified Langlands parameters, following  \cite{CFMMX}, which builds on \cite{Vogan:Langlands}.
As in \cite{CFMMX} and \cite{CFZ:cubics}, we define the \emph{infinitesimal parameter} of a Langlands parameter $\phi : W'_F \to \Lgroup{G_2}$ is $\lambda_\phi : W_F \to \Lgroup{G_2}$ to be $\lambda_\phi(w)=\phi(w,\diag(|w|^{1/2},|w|^{-1/2}))$. 
There is a natural moduli space for Langlands parameters with the same infinitesimal parameter as $\phi$, introduced by Vogan and used extensively in \cite{CFMMX} and \cite{CFZ:cubics}.
This moduli space is denoted by $V_{\lambda_\phi}$ in this paper and it naturally carries an action of the reductive group $H_{\lambda_\phi} \ceq Z_{\dualgroup{G_2}}(\lambda_\phi)$; it is, in particular, a prehomogeneous vector space.
In this paper we review the bijection between simple objects in the category $\Perv_{H_{\lambda_\phi}}(V_{\lambda_\phi})$ equivariant perverse sheaves on $V_{\lambda_\phi}$ and irreducible admissible representations of $G_2(F)$ for which the Langlands parameter has infinitesimal parameter $\lambda_\phi$.

We show that the Aubert involution of admissible representations appears as the Fourier transform of equivariant perverse sheaves on the other side of this bijection.
We also show how to recognize when a unipotent representation is tempered, Arthur type, unitary, generic, spherical or supercuspidal purely in terms of the geometry of the corresponding simple equivariant perverse sheaf.
To explain this, we say that $\phi$ is \emph{open} (resp. \emph{closed}) if the $Z_{\dualgroup{G_2}}(\lambda_\phi)$-orbit of $\phi$ is open (resp. closed) in $V_{\lambda_\phi}$.
If $\pi$ is tempered then its Langlands parameter is open but the converse is not true; in this sense, the notion of open parameters generalizes the notion of parameters bounded upon restriction to $W_F$.
We remark that a parameter $\phi : W'_F \to \Lgroup{G_2}$ is closed if and only if it is trivial on the unipotent part of $\SL_2(\CC)$ appearing in $W'_F \ceq W_F \times \SL_2(\CC)$.
Returning to the claim that ABV-packet coefficients provide an extension of the Langlands correspondence and Arthur packets, we use these notions to prove fundamental properties of ABV-packet coefficients in Theorem~\ref{thm:coefficients}; for example, in Part~\ref{G2:open} of Theorem~\ref{thm:coefficients} we show that if $\phi: W'_F \to \Lgroup{G_2}$ is open or closed then $\Pi^\ABV_\phi(G_2(F)) \to \widehat{A^\ABV_\phi}$ is bijective. In \cite{CFZ} we found examples when $\Pi^\ABV_\phi(G_2(F)) \to \widehat{A^\ABV_\phi}$ is not bijective. 
In Part~\ref{G2:spherical} of Theorem~\ref{thm:coefficients} we show that a representation $\pi  \in \Pi_{\phi}(G_2(F))$ is spherical if and only if $\phi$ is closed and $\langle \ , \pi \rangle $ is the trivial representation of $A^\ABV_\phi$, and in this case, the only ABV-packet that contains $\pi$ is $\Pi^\ABV_{\phi}(G_2(F))$.

\subsection{Notation}\label{ssec:notation}

Throughout this paper $F$ denotes a $F$-adic field, $\varphi$ a fixed uniformizer in the ring $\mathcal{O}_F$ of integers of $F$ and $q$ denotes the cardinality of the residue field of $F$; we place no lower bound on $q$.

We will use "polar notation" $\chi =\mu\nu^{a}$ for complex characters $\chi$ of $F^\times$, where $\nu : F^\times \to \CC^\times$ is the unramified character defined by $\nu(\varpi) = q$ and where $\mu : F^\times \to \CC^\times$ is unitary and $a\in \RR$. When subscripts are necessary, this notation becomes $\chi_i =\mu_i\nu^{a_i}$. We will use the notation $\theta_n$ for a fixed, unitary character of $F^\times$ of order $n$.

Similarly, we use the "$q$-polar notation" $z  = u q^a$ for complex numbers $z\in \CC^\times$ where $u$ is unitary and $a\in \RR$.

By abuse of notation, we also denote a fixed primitive complex $n$-th root-of-unity by $\theta_n$ and write $\vartheta_n$ for the corresponding character of $\langle \theta_n \rangle \ceq \{ 1, \theta_n, \ldots, \theta_n^{n-1}\}$, so $\vartheta_n(\theta_n) = \theta_n$.
When the order $n$ is understood by context, we will write $\vartheta$ for this character.

We write $\varepsilon$ for the sign character of $S_3$ and $\varrho$ for the reflection representation of $S_3$.

For a split group $G$ over $F$ with a fixed Borel $B$ and torus $T$ over $F$, we set 
$I^{G}(\sigma)\ceq\Ind_{B(F)}^{G(F)}(\sigma)$, where $\sigma$ is a representation of $T(F)$, inflated to $B(F)$; we write $J^{G}(\sigma)$ for the Langlands quotient $I^{G}(\sigma)$ when this induced representation is a standard module.
More generally, if $P$ is a parabolic subgroup of $G$ with Levi subgroup $M$, we set
$I_{P}^{G}(\sigma)\ceq\Ind_{P(F)}^{G(F)}(\sigma)$; we write $J_{P}^{G}(\sigma)$ for the Langlands quotient of $I_{P}^{G}(\sigma)$ when this induced representation is a standard module.

For an algebraic group $G$ over $F$, denote by $1_G$ the trivial representation of $G(F)$ and $\St_G$ the Steinberg representation of $G(F)$.  

For representations of $G_2(F)$, we follow the notations from \cite{Muic} except that the Steinberg representation of $\GL_2(F)$ is denoted by $\St_{\GL_2}$ instead of $\delta(1)$ as in \cite{Muic}. 

There are 4 cuspidal unipotent representations of $G_2(\mathbb{F}_q)$: they appear as $G_2[1]$, $G_2[-1]$, $G_2[\theta_3]$ and $G_2[\theta^2_3]$ in \cite{Carter}*{p.460}. For a cuspidal unipotent representation $\sigma$ of $G_2(\mathbb{F}_q)$, we use the shorthand $I_0(\sigma)$ for $\mathrm{cInd}_{G_2(\CO_F)}^{G_2(F)}(\sigma)$.
These supercuspidal representations also appear in \cite{CO}, where they are denoted by $v_6 =I_0(G_2[1])$, $v_7=I_0(G_2[-1])$, $v_8 = I_0(G_2[\theta_3])$ and $v_9 = I_0(G_2[\theta^2_3])$.
The last two of these also appear in \cite{Savin}*{Section 4}, where they are denoted by $\pi'[\nu]$ and $\pi'[\nu^2]$.

Throughout this paper, we use the notation $\1_X$ for the constant sheaf on $X$, or simply $\1$ if the support $X$ is understood.

\subsection{Acknowledgements}

The authors gratefully acknowledge and thank Goran Mui\'{c} for explaining his work on admissible representations of $G_2(F)$ and Maartin Solleveld for helpful comments on the first draft this paper.
The authors would also like to thank all the members of the \href{http://automorphic.ca}{Voganish Project}, especially Ahmed Moussaoui, Jerrod Smith and Geoff Vooys.



\section{Geometry of the moduli space of Langlands parameters for $G_2$}\label{sec:PHV}

\subsection{Roots for $G_2$ and $\dualgroup{G_2}$}\label{ssec:notations}

Let $T$ be a fixed maximal split torus of $G_2$; the corresponding set of roots are denoted by $R\ceq R(G_2,T)$. We write $W$ for the Weyl group for this root system $R$.
Let ${\gamma_1},{\gamma_2} : T \rightarrow F^\times$ be a choice simple roots of $G_2$ with ${\gamma_1}$ short and ${\gamma_2}$ long so that the positive roots $R$ are
\[
{\gamma_1}, {\gamma_2}, {\gamma_1}+{\gamma_2}, 2{\gamma_1}+{\gamma_2}, 3{\gamma_1}+{\gamma_2}, 3{\gamma_1}+2{\gamma_2}.
\]
The coroots are denoted  $R^\vee\ceq R(G_2,T)^\vee$, with ${\gamma_1^\vee},{\gamma_2^\vee} : F^\times \to T$ the coroots of ${\gamma_1}$ and ${\gamma_2}$.
We fix the isomorphism $T \rightarrow F^\times \times F^\times$ by 
 \[ t \mapsto ((2{\gamma_1}+{\gamma_2})(t),({\gamma_1}+{\gamma_2})(t)). \]
  We use the notation $m: F^\times \times F^\times \to T$ for the inverse of this isomorphism.
We have
 \[ {\gamma_1^\vee}(a)=m(a,a^{-1}),\qquad\text{\ and\ }\qquad {\gamma_2^\vee}(a)=m(1,a) . \]

We denote by $\dualgroup{G_2}$ the dual group of $G_2$ over $\mathbb{C}$ and $\dualgroup{T}$ the dual torus and let ${\hat R} \ceq R(\dualgroup{G_2},\dualgroup{T})$ be the roots of $\dualgroup{G_2}$ with the usual identification of coroots of $G_2$ with roots of $\dualgroup{G_2}$. We again write $W$ for the Weyl group for the root system ${\hat R}$.
Denote by ${\hat\gamma}_1\in {\hat R}$ (resp. ${\hat\gamma}_2\in {\hat R}$) the image of ${\gamma_1^\vee}$ (resp. ${\gamma_2^\vee}$) under this identification. 
Then $\wh G_2$ is a complex reductive group of type $G_2$, with simple roots ${{\hat\gamma}_1},{{\hat\gamma}_2}$.
where ${{\hat\gamma}_1}$ is a long root of $\wh G_2$ and ${{\hat\gamma}_2}$ is a short root of $\wh G_2$. 
%
As above, we fix an isomorphism $\dualgroup{T}\rightarrow \mathbb{C}^\times\times\mathbb{C}^\times $ by 
 \[ t\mapsto (({\hat\gamma_1}+2{\hat\gamma_2})(t),({\hat\gamma_1}+{\hat\gamma_2})(t)) \]
and we write $\wh m : \mathbb{C}^\times\times\mathbb{C}^\times \to \dualgroup{T}$ for the inverse of this isomorphism.
Note that we have 
\begin{equation}\label{eqn:rootactionontorus}
{{\hat\gamma}_1} (\wh m(x,y))=x^{-1}y^2, \qquad \mathrm{\ and\ }\qquad  {{\hat\gamma}_2}(\wh m(x,y))=xy^{-1}.
\end{equation}
We will denote by ${{\hat\gamma}_1}^\vee, {{\hat\gamma}_2}^\vee :\mathbb{C}^\times \rightarrow \dualgroup{T}$ the coroots of ${{\hat\gamma}_1}$ and ${{\hat\gamma}_2}$. 
We have again that
\[ {{\hat\gamma}_1}^\vee(a)={\wh m}(1,a),\quad {{\hat\gamma}_2}^\vee(a)={\wh m}(a,a^{-1}) . \] 

Observe that under the identification $R(G_2,T)= R(\dualgroup{G_2},\dualgroup{T})^\vee$, we have 
${\gamma_1} = {{\hat\gamma}_1}^\vee$,
${\gamma_2} = {{\hat\gamma}_2}^\vee$,
$({{\hat\gamma}_1} + {{\hat\gamma}_2})^\vee = 3{{\hat\gamma}_1}^\vee + {{\hat\gamma}_2}^\vee$, 
$({{\hat\gamma}_1} + 2{{\hat\gamma}_2})^\vee = 3{{\hat\gamma}_1}^\vee + 2{{\hat\gamma}_2}^\vee$, 
$({{\hat\gamma}_1} + 3{{\hat\gamma}_2})^\vee = {{\hat\gamma}_1}^\vee + {{\hat\gamma}_2}^\vee$ and 
$(2{{\hat\gamma}_1} + 3{{\hat\gamma}_2})^\vee = 2{{\hat\gamma}_1}^\vee + {{\hat\gamma}_2}^\vee$.

\subsection{Moduli spaces of Langlands parameters}\label{ssec:infcases}


Set $W'_F \ceq W_F \times \SL_2(\CC)$. 
Given a Langlands parameter $\phi:W_F'\to \Lgroup{G_2}$, recall that its {\it infinitesimal parameter} $\lambda_\phi:W_F\to \Lgroup{G_2}$ is defined by $\lambda_\phi(w)=\phi(w,\diag(|w|^{1/2},|w|^{-1/2}))$. 

For a given infinitesimal parameter $\lambda : W_F \to \Lgroup{G_2}$, there are only a finite number of $\dualgroup{G_2}$-conjugacy classes of Langlands parameters $\phi$ with $\lambda=\lambda_\phi$. 
As explained in \cite{CFMMX} more generally, each unramified infinitesimal parameter $\lambda: W_F\to \Lgroup{G_2}$ determines a prehomogeneous vector space
\[
V_{\lambda}=\{x\in \Lie\dualgroup{G_2} \tq \Ad(\lambda(w))x =q^{|w|} x, \forall w\in W_F\},
\]
equipped with an action of the group 
\[
H_\lambda \ceq\{g\in \dualgroup{G_2} \tq \lambda(w)g\lambda(w)^{-1}=g, \forall w\in W_F\}.
\] 
acting via the adjoint action of $\dualgroup{G_2}$ on $\Lie\dualgroup{G_2}$.
Using this action we can assume that $\lambda(\Frob)\in \Lgroup{T} $ and define
\[
{\hat R}_\lambda \ceq \{ {\hat\gamma}\in R(\dualgroup{G},\dualgroup{T}) \tq \Lie U_{\hat\gamma} \subseteq V_\lambda \},
\]
where $U_{\hat\gamma}\subset \dualgroup{G_2}$ is the root subgroup for ${\hat\gamma}$.
Then $V_\lambda$ is a moduli space for all Langlands parameters $\phi$ with $\lambda_\phi = \lambda$ and $X_\lambda \ceq \dualgroup{G_2}\times_{H_\lambda} V_\lambda$ is a moduli space for all Langlands parameters $\phi$ for which $\lambda_\phi$ is $\dualgroup{G_2}$-conjugate to $\lambda$. Note that $\dualgroup{G_2}$ acts naturally on $X_\lambda$ while $H_\lambda$ acts on $V_\lambda$.

These notions lead to the following classification of unramified infinitesimal parameters $\lambda$ for $G_2(F)$.

\begin{proposition}\label{prop:infcases}
For every unramified $\lambda:W_F\to \Lgroup{G_2}$, the set ${\hat R}_\lambda$ is $W$-conjugate to one of the following subsets of ${\hat R}$, also pictured in Table~\ref{table:infcases}:
\begin{multicols}{2}
\begin{enumerate}

\labitem{0}{infcase:0}
${\hat R}_\lambda =\emptyset$;

\labitem{1}{infcase:1-short}
${\hat R}_\lambda =\{\hat\gamma_1+2\hat\gamma_2\}$;

\labitem{2}{infcase:2-long}
${\hat R}_\lambda =\{2\hat\gamma_1+3\hat\gamma_2\}$;

\labitem{3}{infcase:3}
${\hat R}_\lambda =\{\hat\gamma_1,2\hat\gamma_1+3\hat\gamma_2 \}$;

\labitem{4}{infcase:4-D2}
${\hat R}_\lambda=\{\hat\gamma_1,\hat\gamma_1+2\hat\gamma_2 \}$;

\labitem{5}{infcase:5}
${\hat R}_\lambda =\{\hat\gamma_2,\hat\gamma_1+\hat\gamma_2\}$; 

\labitem{6}{infcase:6-A2}
${\hat R}_\lambda=\{\hat\gamma_1, \hat\gamma_1+3\hat\gamma_2\}$;

\labitem{7}{infcase:7-reg}
${\hat R}_\lambda=\{\hat\gamma_1,\hat\gamma_2 \}$;

\labitem{8}{infcase:8-sub}
${\hat R}_\lambda=\{ \hat\gamma_1, \hat\gamma_1+\hat\gamma_2, \hat\gamma_1+2\hat\gamma_2, \hat\gamma_1+3\hat\gamma_2\}$.

\end{enumerate}
\end{multicols}
Each case above determines a moduli space $V_\lambda$ of unramified Langlands parameters with infinitesimal parameter $\lambda$.
\end{proposition}

\begin{table}[tp]
\caption{The classification of unramified infinitesimal parameters $\lambda$ for $G_2(F)$ in terms of the subsets ${\hat R}_\lambda$ appearing in Proposition~\ref{prop:infcases}. Black root vectors with a halo appear in ${\hat R}_\lambda$ and so the corresponding root spaces form $V_{\lambda}$. Vectors in black without a halo determine root groups appearing in the reductive group $H_\lambda$ acting on $V_\lambda$ }
\label{table:infcases}
\begin{center}

\includegraphics[height=4.6cm,width=4cm]{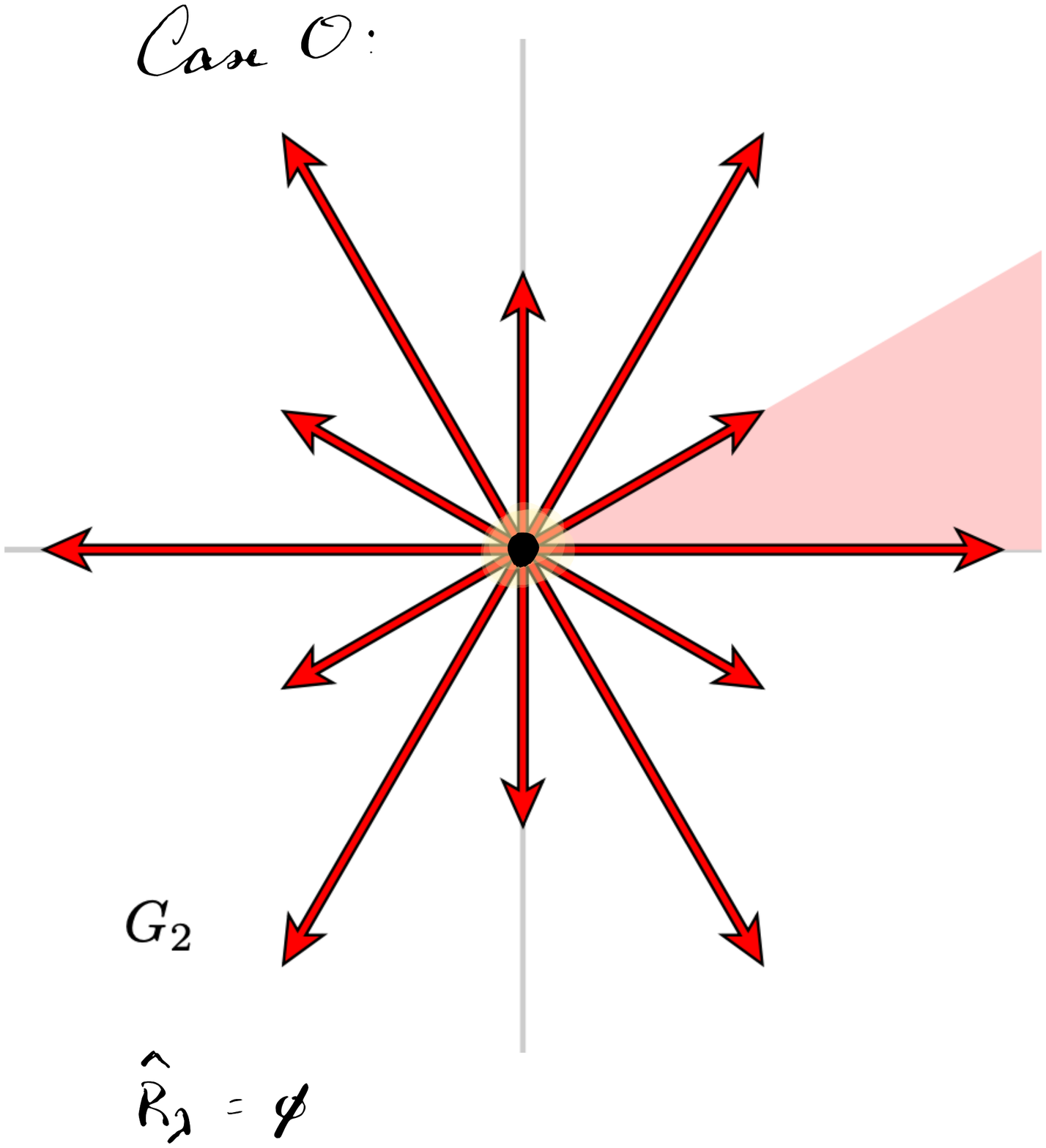}
\includegraphics[height=4.6cm,width=4cm]{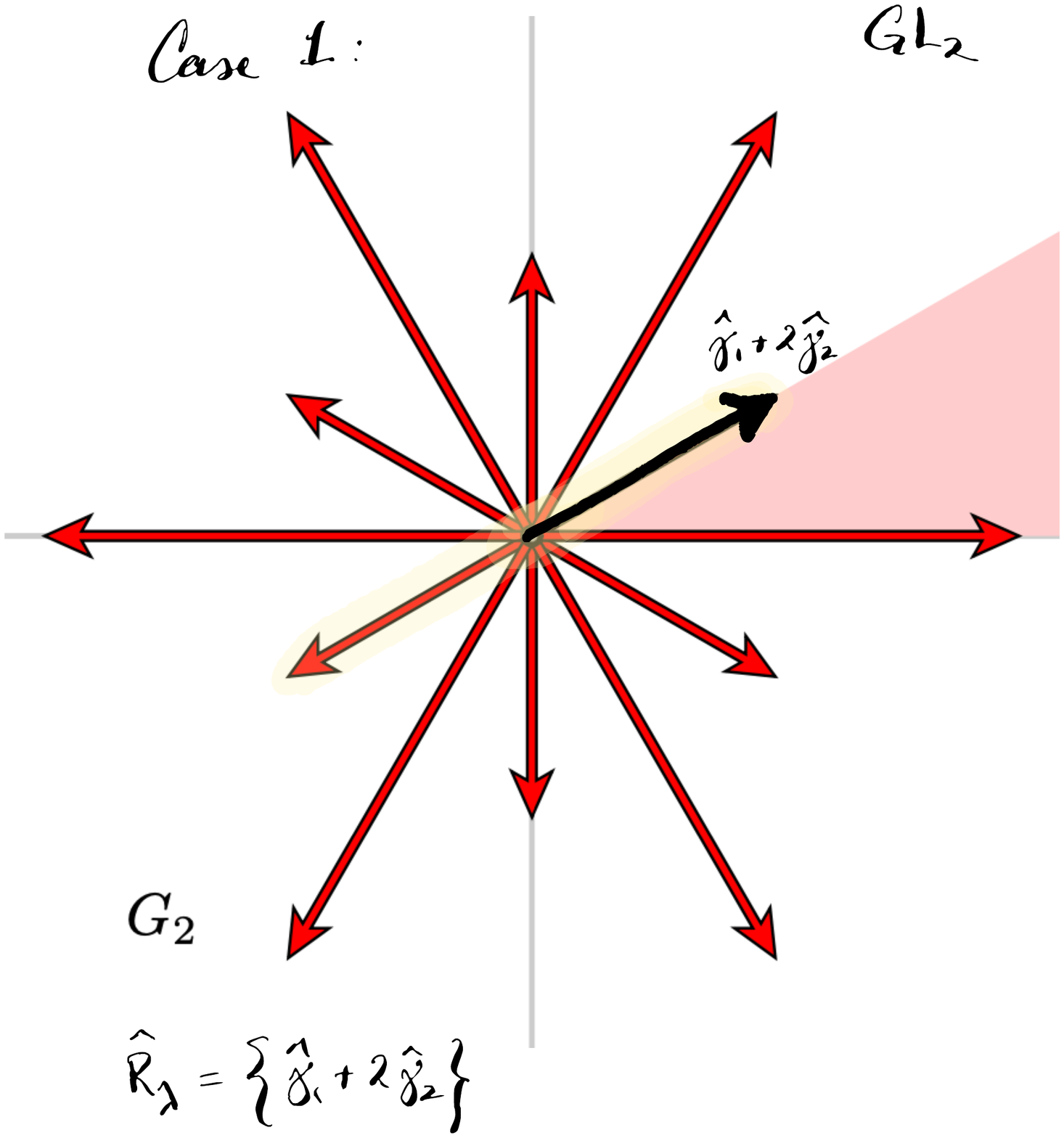}
\includegraphics[height=4.6cm,width=4cm]{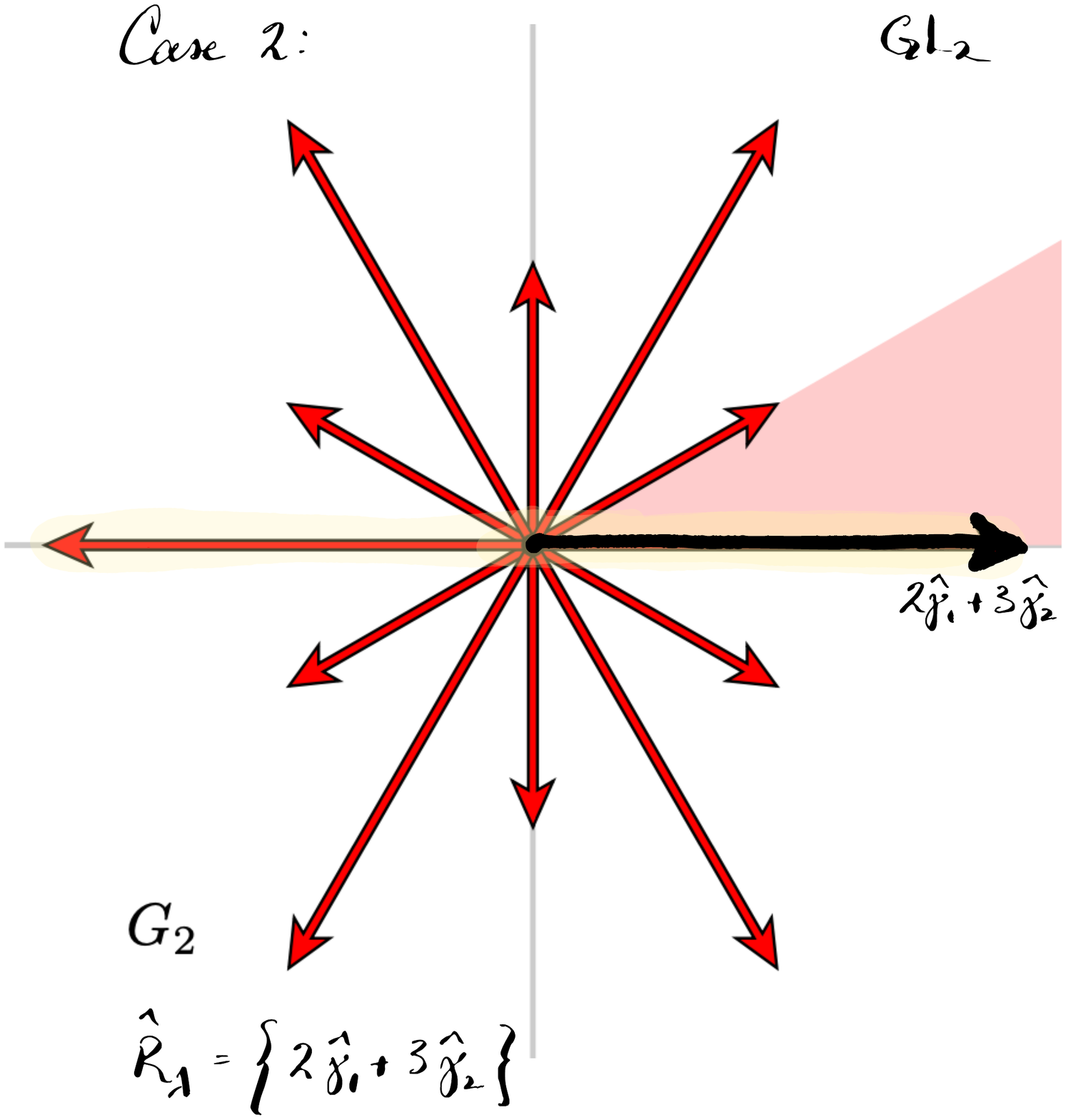}

\includegraphics[height=4.6cm,width=4cm]{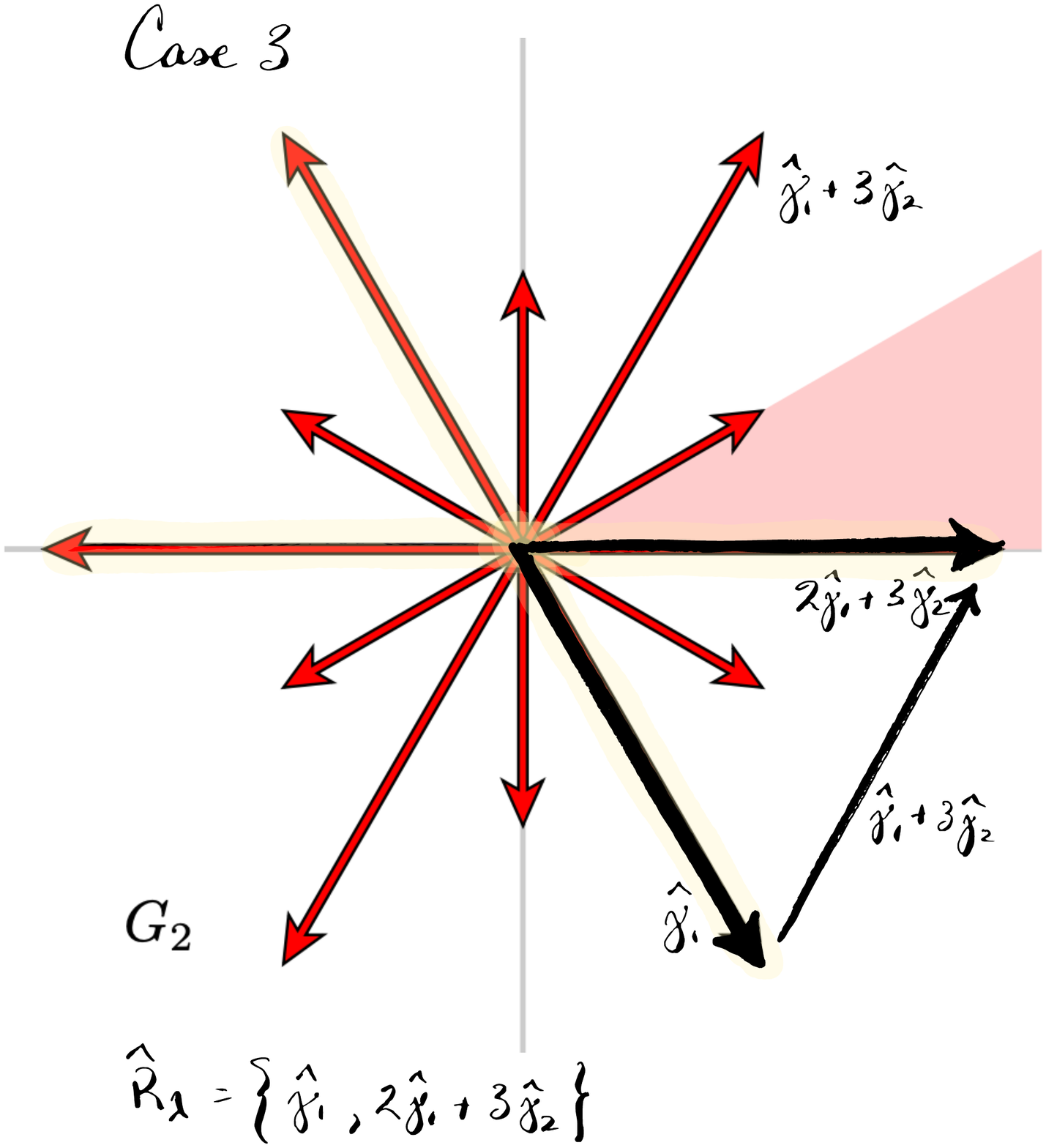}
\includegraphics[height=4.6cm,width=4cm]{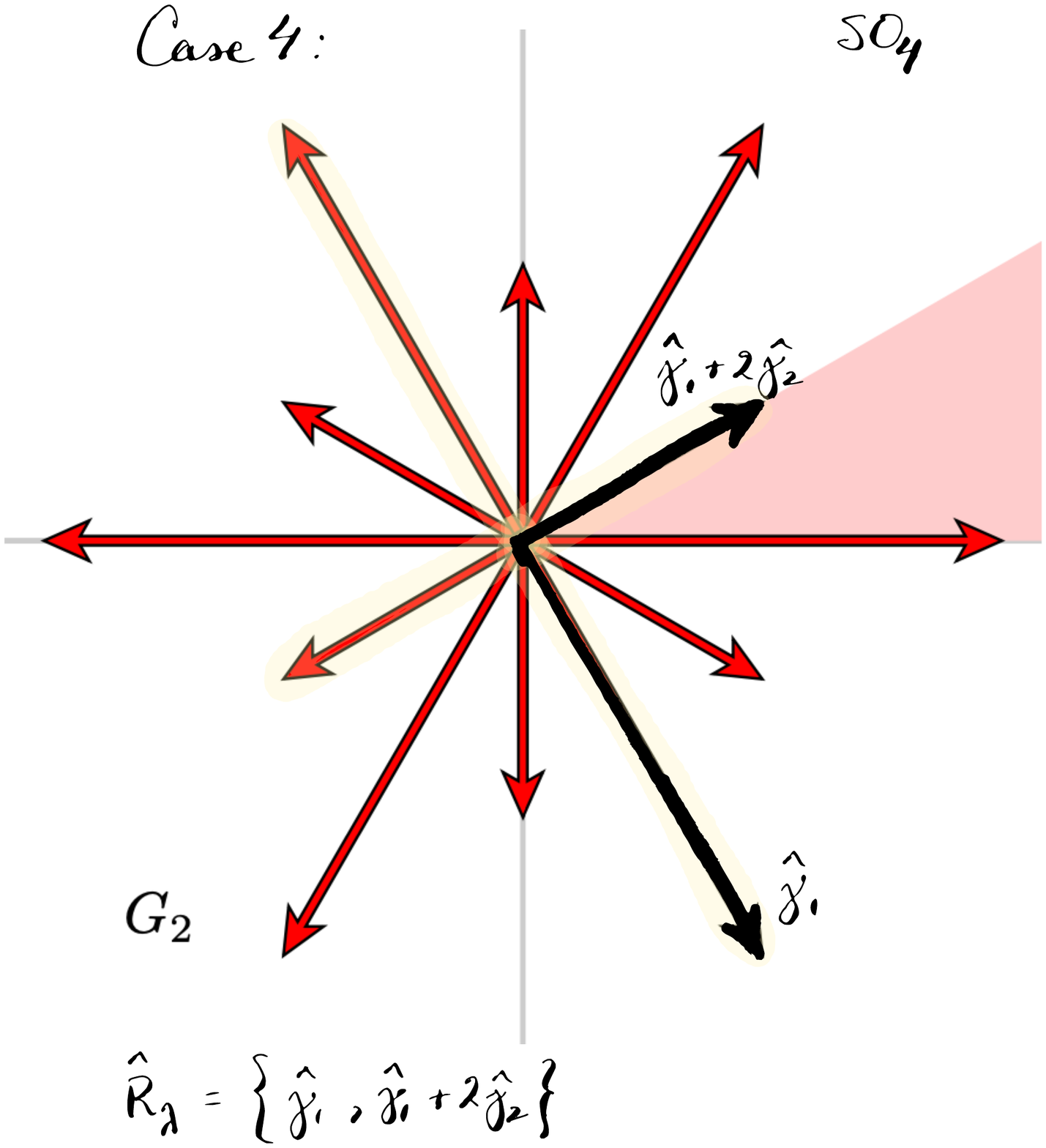}
\includegraphics[height=4.6cm,width=4cm]{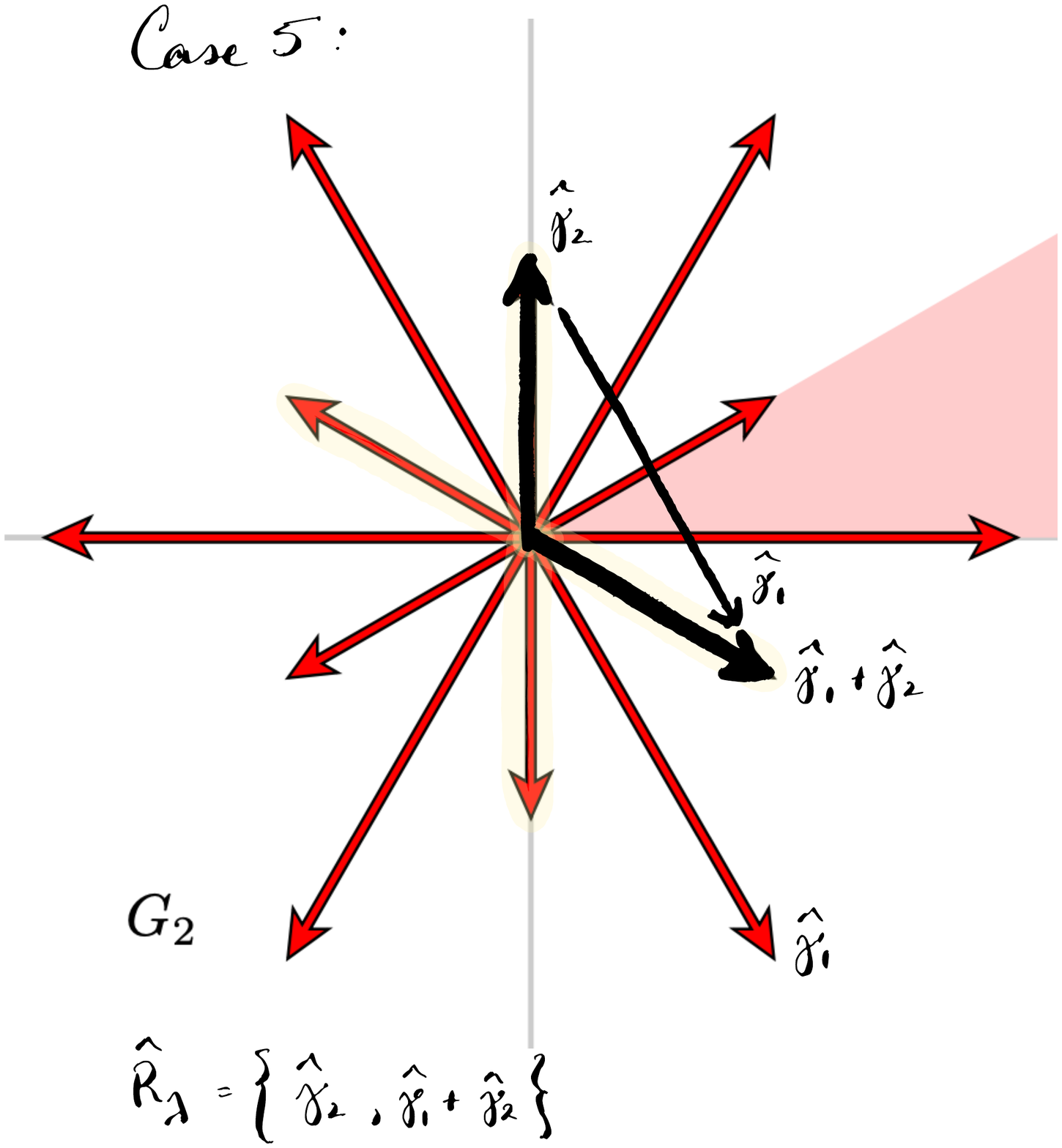}

\includegraphics[height=4.6cm,width=4cm]{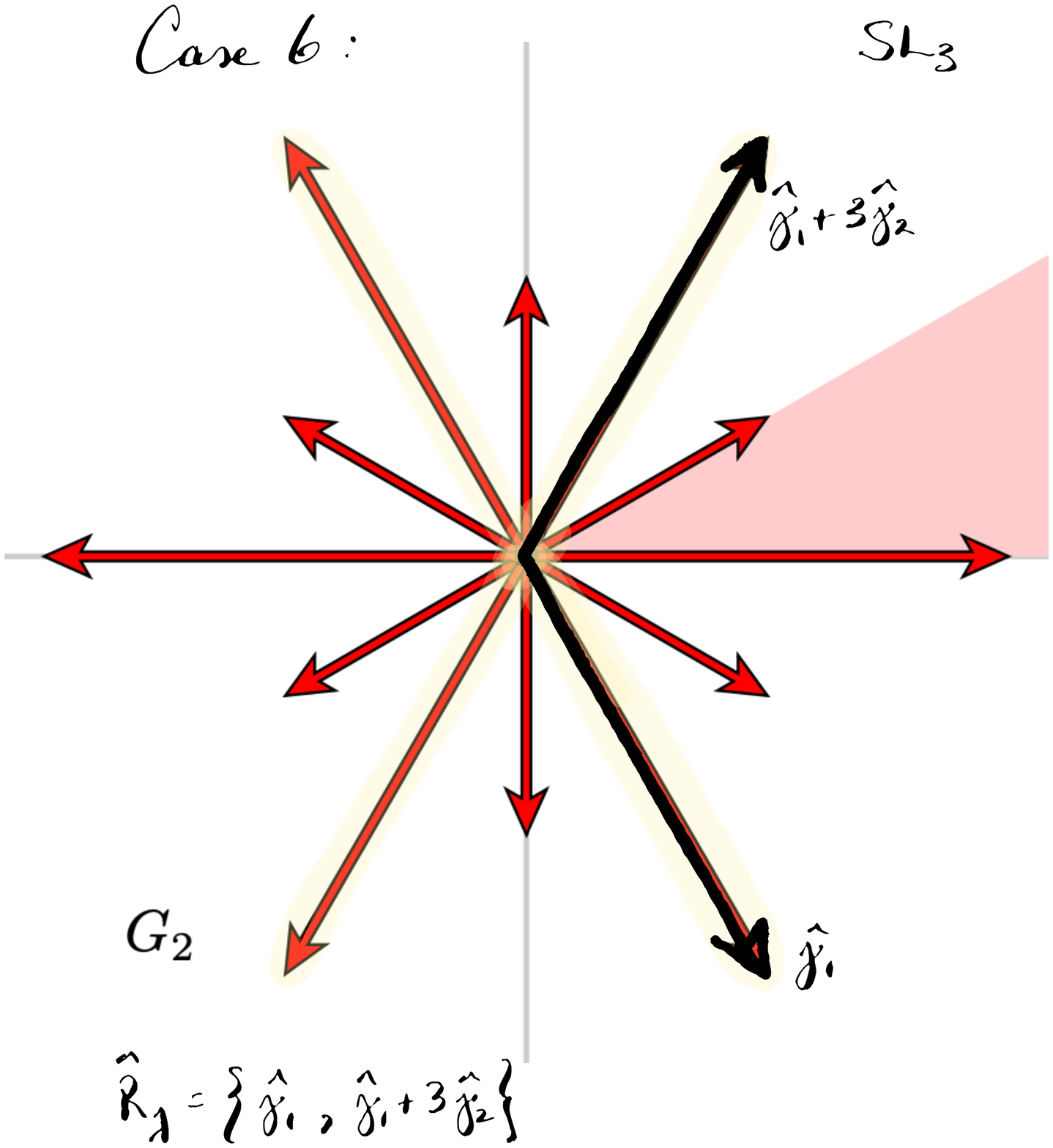}
\includegraphics[height=4.6cm,width=4cm]{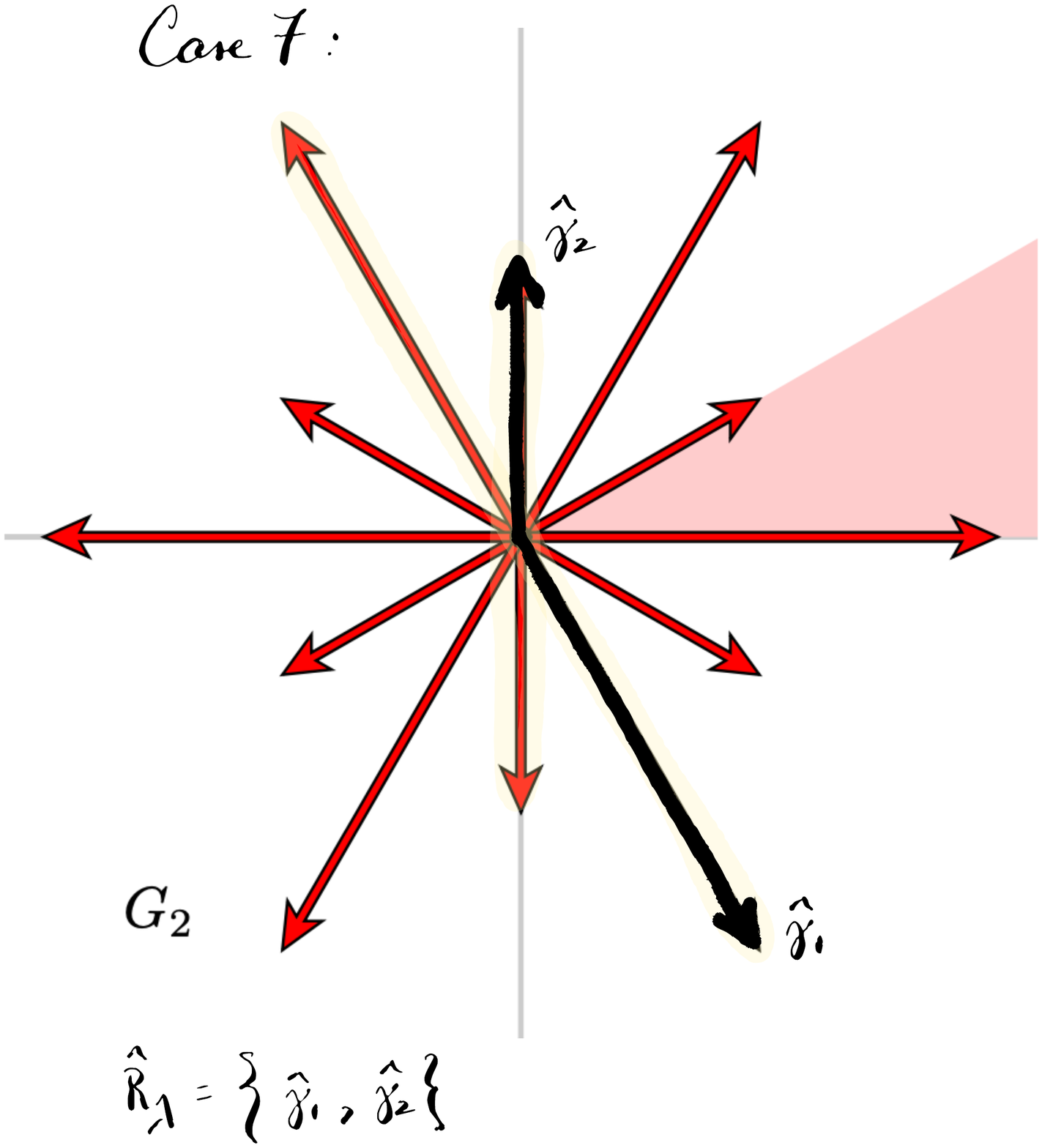}
\includegraphics[height=4.6cm,width=4cm]{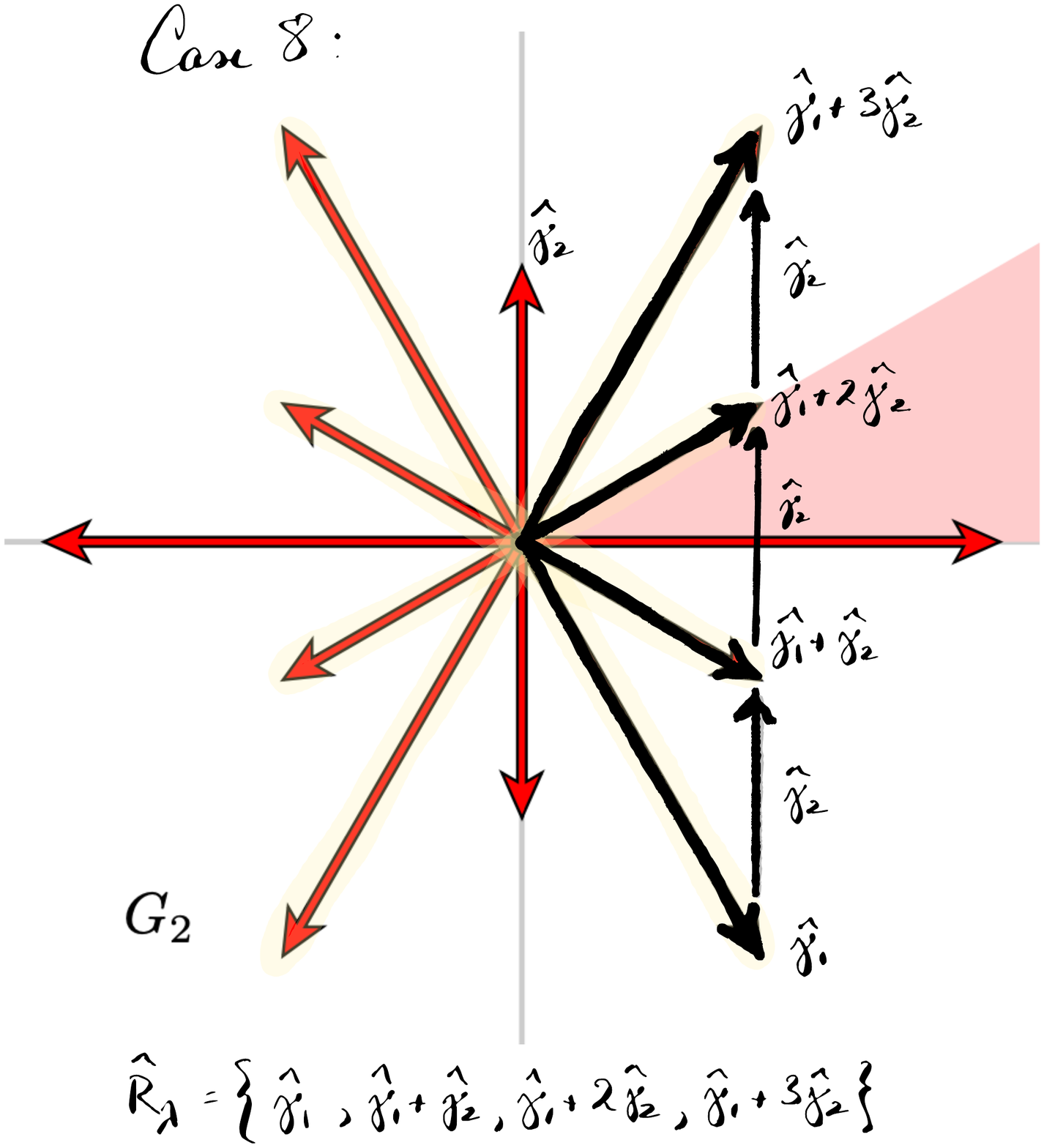}

\end{center}
\end{table}%

\begin{proof}
The possibilities for $\lambda$ are classified by the associated prehomogeneous vector spaces $V_\lambda$ as they appear in $\Lie\dualgroup{G_2}$ up to $\dualgroup{G}_2$-conjugacy.

Case~\ref{infcase:0} describes all $\lambda$ for which $V_\lambda=0$.
These are the $\lambda$ for which ${\hat\gamma}(\lambda(\Frob)) \ne q$ for every root ${\hat\gamma}\in {\hat R}$.
In this case the group $H_{\lambda}$ is $\dualgroup{G_2}$ or one of the subgroups $\SL_3(\C)$, $\SO_4(\C)$,  $\GL_{2}^{{\hat\gamma}_1}(\CC)$, $\GL_{2}^{{\hat\gamma}_2}(\CC)$ or $\dualgroup{T}$, where $U_{\hat\gamma} \subset \GL_{2}^{\hat\gamma}$.

Cases~\ref{infcase:1-short} and \ref{infcase:2-long} describe all $\lambda$ for which $\dim V_\lambda=1$; in these cases, $V_\lambda = \mathrm{Span}\{X_{\hat\gamma}\}$ for some root ${\hat \gamma}$ of $\dualgroup{G}_2$ and $H_\lambda$ is either $\dualgroup{T}$ or $\GL_2^{\hat\gamma'}(\CC)$ for a root ${\hat\gamma'}$ perpendicular to ${\hat\gamma}$.
Up to $\dualgroup{G}_2$-conjugacy, there are only two cases: either ${\hat\gamma}$ is short or long; in Case~\ref{infcase:1-short} we consider the short root ${\hat \gamma} = {{\hat\gamma}_1}+2{{\hat\gamma}_2}$, to which ${{\hat\gamma}_1}$ is perpendicular, and in Case~\ref{infcase:2-long} we consider the long root ${\hat\gamma} = 2{{\hat\gamma}_1}+3{{\hat\gamma}_2}$, to which ${{\hat\gamma}_2}$ is perpendicular.

Cases~\ref{infcase:3}, \ref{infcase:4-D2}, \ref{infcase:5}, \ref{infcase:6-A2}  and \ref{infcase:7-reg} describe all $\lambda$ for which $\dim V_\lambda=2$; in these cases, $V_\lambda = \mathrm{Span}\{X_{\widehat\gamma}, X_{\widehat\gamma'}\}$ for roots ${\widehat \gamma}$ and $X_{\hat\gamma'}$ of $\dualgroup{G}_2$ and $H_\lambda$ is either $\dualgroup{T}$ or $\GL_2^{\hat\gamma_0}$ for a root ${\hat\gamma_0}$ such that ${\widehat\gamma'} = {\widehat\gamma} + {\widehat\gamma_0}$.

If $\dim V_\lambda >2$ there is only one possibility for $\lambda$ and this is Case~\ref{infcase:8-sub}, treated in \cite{CFZ:cubics}.
When $\dim V_\lambda >2$ this forces $\dim V_\lambda=4$ and that 
\[
V_\lambda = \mathrm{Span}\{X_{\widehat\gamma}, X_{{\widehat\gamma} + {\widehat\gamma_0}}, X_{{\widehat\gamma} + 2{\widehat\gamma_0}}, X_{{\widehat\gamma} + 3{\widehat\gamma_0}} \}
\]
for roots $\hat\gamma$ and ${\widehat\gamma_0}$.
The group that acts on $V_\lambda$ in this case is $H_\lambda = \GL_2^{\widehat\gamma_0}$.
There are exactly six such $V_\lambda$, all $\dualgroup{G}_2$-conjugate to $\mathrm{Span}\{X_{{{\hat\gamma}_1}}, X_{{{\hat\gamma}_1}+{{\hat\gamma}_2}}, X_{{{\hat\gamma}_1}+2{{\hat\gamma}_2}}, X_{{{\hat\gamma}_1}+3{{\hat\gamma}_2}} \}$, in which case $H_\lambda = \GL_2^{{{\hat\gamma}_2}}$.
\end{proof}

\subsection{Prehomogeneous vector spaces}

This classification of infinitesimal parameters $\lambda$ for $G_2$ by the associated prehomogeneous vector spaces $H_\lambda\times V_\lambda \to V_\lambda$ given in Proposition~\ref{prop:infcases} may be simplified further by examining the categories $\Perv_{H_\lambda}(V_\lambda)$ that arise as $\lambda$ goes through the cases appearing there.


\begin{proposition}\label{prop:PHV}
If $\lambda : W_F \to \Lgroup{G_2}$ is an unramified infinitesimal parameter for $G_2(F)$ then the category $\Perv_{H_\lambda}(V_\lambda)$ is equivalent to $\Perv_{H}(V)$ where $H\times V\to V$ is one of the following five prehomogeneous vector spaces.

\begin{enumerate}
\labitem{P0}{geocase-0}
$V=0$ and $H$ is an algebraic group, not necessarily connected (trivial action);
\labitem{P1}{geocase-1}
$V=\mathbb{A}^1$ with $H= \GL_1$-action given by $t\cdot x= tx$ (scalar multiplication);
\labitem{P2}{geocase-2}
$V=\mathbb{A}^2$ with $H= \GL_2$-action $h\cdot x= det(h)^n hx$, for non-negative integer $n$ (twisted matrix multiplication);
\labitem{P3}{geocase-toric}
$V=\mathbb{A}^2$ with $H= \GL_1^2$-action $(t_1,t_2)\cdot(x_1,x_2) = (t_1x_1,t_1t_2^nx_2)$, for positive integer $n$ (toric variety);
\labitem{P4}{geocase-sub}
$V=\mathbb{A}^4$ with $H= \GL_2$-action $h\cdot x=(\det^{-1}\otimes\Sym^3)(h)(x) = \det(h)^{-1}\Sym^3(h)x$ (normalized Symmetric cube).
\end{enumerate}
\end{proposition}

\begin{proof}
We begin the proof with a general observation.
Suppose $H$ acts on $V$ and let $u : H'\to H$ be an epimorphism of algebraic groups with connected kernel.
Let $H'$ act on $V$ through $u$. Then $\Perv_{H}(V)$ and $\Perv_{H'}(V)$ are equivalent.
To see this, recall that an $H$-equivariant perverse sheaf on $V$ is a perverse sheaf $\mathcal{F}$ equipped with an isomorphism $\mu : m^*\mathcal{F} \to p^*\mathcal{F}$ satisfying a so-called cocycle condition. 
Consider the functor $\Perv_{H}(V) \to \Perv_{H'}(V)$ defined on objects by $(\mathcal{F},\mu) \mapsto (\mathcal{F},(u\times\id)^*\mu)$ and on arrows by $\varphi \mapsto \varphi$. 
This functor is essentially surjective. 
Since $u\times\id : H'\times V\to H\times V$ is smooth with connected fibres of equal dimension $\dim \ker u$ the functor $(u\times\id)^*$ is full and faithful by, for example, \cite{BBD}*{Proposition 4.2.5}. Thus, $\Perv_{H}(V)$ and $\Perv_{H'}(V)$ are equivalent.
We refer to this equivalence as an instance of base-change, below, and in this case we will say that the prehomogeneous vector spaces $H\times V\to V$ and $H'\times V\to V$ are equivalent.

The cases below all refer to Proposition~\ref{prop:infcases}.
\begin{enumerate}

\item[(Case \ref{infcase:0})]
The prehomogeneous vector space in Proposition~\ref{prop:infcases}, Case~\ref{infcase:0} is $V=0$ for a connected group acting on $0$, which is an instance of the prehomogeneous vector space \ref{geocase-0}. By the base-change argument above, this is equivalent to \ref{geocase-0} for trivial group $H$.

\item[(Case \ref{infcase:1-short})]
There are two group actions appearing in Proposition~\ref{prop:infcases}, Case~\ref{infcase:1-short}, either $H_{\lambda_{\ref{infcase:1-short}}}= \dualgroup{T}$ or $H_{\lambda_{\ref{infcase:1-short}}}= \GL_2$. 
In the former case, observe that ${\hat\gamma_1}+2{\hat\gamma_2} : \dualgroup{T} \to {\GL_1}$ is an epimorphism of algebraic groups with connected fibre so, by the paragraph above,  $\Perv_{H_{\lambda_{\ref{infcase:1-short}}}}(V_{\lambda_{\ref{infcase:1-short}}})$ is equivalent to $\Perv_{{\GL_1}}(\mathbb{A}^1)$ for the prehomogeneous vector space \ref{geocase-1}.
In the latter case, observe that $\det : \dualgroup{T} \to {\GL_1}$ is an epimorphism of algebraic groups with connected fibre so, by the base-change argument above  $\Perv_{H_{\lambda_{\ref{infcase:1-short}}}}(V_{\lambda_{\ref{infcase:1-short}}})$ is again equivalent to $\Perv_{{\GL_1}}(\mathbb{A}^1)$ for the prehomogeneous vector space  \ref{geocase-1}.

\item[(Case \ref{infcase:2-long})]
In Proposition~\ref{prop:infcases}, Case~\ref{infcase:2-long} there are two group actions to consider: either $H_{\lambda_{\ref{infcase:2-long}}}= \dualgroup{T}$ or $H_{\lambda_{\ref{infcase:2-long}}}= \GL_2$.
Here, $V_{\lambda_{\ref{infcase:2-long}}}= \mathbb{A}^1$. 
If $H_{\lambda_{\ref{infcase:2-long}}}= \dualgroup{T}$ then the action is $h.x= \det(h)x$. 
Using the base-change argument above we have $\Perv_{H_{\lambda_{\ref{infcase:2-long}}}}(V_{\lambda_{\ref{infcase:2-long}}}) = \Perv_{{\GL_1}}(\mathbb{A}^1)$ since $2{\hat\gamma_1}+3{\hat\gamma_2} : \dualgroup{T} \to {\GL_1}$ is an epimorphism of algebraic groups with connected fibre. If $H_{\lambda_{\ref{infcase:2-long}}}= \GL_2$ then the action is $h\cdot x = \det(h)x$.
Then, using the base-change argument above, $\Perv_{H_{\lambda_{\ref{infcase:2-long}}}}(V_{\lambda_{\ref{infcase:2-long}}}) = \Perv_{{\GL_1}}(\mathbb{A}^1)$ since $\det : \GL_2 \to {\GL_1}$ is an epimorphism of algebraic groups with connected fibre. 
Thus, $\Perv_{H_{\lambda_{\ref{infcase:2-long}}}}(V_{\lambda_{\ref{infcase:2-long}}})$ is equivalent $\Perv_{H}(V)$ for the prehomogeneous vector space \ref{geocase-1}.

\item[(Case \ref{infcase:3})]
The prehomogeneous vector space in Proposition~\ref{prop:infcases}, Case \ref{infcase:3} is $V_{\lambda_{\ref{infcase:3}}}= \mathbb{A}^2$ and $H_{\lambda_{\ref{infcase:3}}}= \GL_2$; this is equivalent to the prehomogeneous vector space \ref{geocase-2} for $n=1$.

\item[(Case \ref{infcase:4-D2})]
The prehomogeneous vector space in Proposition~\ref{prop:infcases}, Case \ref{infcase:4-D2} is $V_{\lambda_{\ref{infcase:4-D2}}}= \mathbb{A}^2$ and $H_{\lambda_{\ref{infcase:4-D2}}} = \dualgroup{T}$; this is equivalent to the prehomogeneous vector space \ref{geocase-toric} for $n=2$.

\item[(Case \ref{infcase:5})]
The prehomogeneous vector space in Proposition~\ref{prop:infcases}, Case \ref{infcase:5} is $V_{\lambda_{\ref{infcase:5}}}= \mathbb{A}^2$ and $H_{\lambda_{\ref{infcase:5}}}= \GL_2$; this is equivalent to the prehomogeneous vector space \ref{geocase-2} for $n=0$.

\item[(Case \ref{infcase:6-A2})]
The prehomogeneous vector space in Proposition~\ref{prop:infcases}, Case \ref{infcase:6-A2} is $V_{\lambda_{\ref{infcase:6-A2}}}= \mathbb{A}^2$ and $H_{\lambda_{\ref{infcase:6-A2}}} = \dualgroup{T}$; this is equivalent to the prehomogeneous vector space \ref{geocase-toric} for $n=3$.

\item[(Case \ref{infcase:7-reg})]
The prehomogeneous vector space in Proposition~\ref{prop:infcases}, Case \ref{infcase:7-reg} is $V_{\lambda_{\ref{infcase:7-reg}}}= \mathbb{A}^2$ and $H_{\lambda_{\ref{infcase:7-reg}}} = \dualgroup{T}$; this is equivalent to the prehomogeneous vector space \ref{geocase-toric} for $n=1$.

\item[(Case \ref{infcase:8-sub})]
The prehomogeneous vector space in Proposition~\ref{prop:infcases}, Case \ref{infcase:8-sub} is $V_{\lambda_{\ref{infcase:8-sub}}}= \mathbb{A}^4$ and $H_{\lambda_{\ref{infcase:8-sub}}} = \GL_2$; as explained in \cite{CFZ:cubics}*{Proposition~3.1}, this is equivalent to the prehomogeneous vector space \ref{geocase-sub}.
Category $\Perv_{\GL_2}(\det^{-1}\otimes\Sym^3)$ appeared in \cite{CFZ:cubics}.\qedhere
\end{enumerate}
\end{proof}

\subsection{The geometry of the moduli space of Langlands parameters}\label{ssec:geometry}

In this section we make a microlocal study of the five prehomogeneous vector spaces appearing in Proposition~\ref{prop:PHV}. In Section~\ref{ssec:LA-packets} we use this to determine the ABV-packets for $G_2(F)$.

Let $H\times V \to V$ be a prehomogeneous vector space.
Let $V^*$ be the dual vector space to $V$ and let $H\times V^* \to V^*$ be the adjoint action. Then $H\times H$ acts on $T^*(V) = V\times V^*$ by $(h,k)\cdot (x,y) \ceq (h\cdot x, k\cdot y)$ and $T^*(V)$ is also a prehomogeneous vector space. 
Now let $[ \, \ ] : V\times V^* \to \Lie H$ be the momentum map and consider the conormal variety  
\[
\Lambda \ceq \{ (x,y)\in T^*(V),\ [x,y]=0\}.
\]
This Lagrangian subspace $\Lambda \subseteq T^*(V)$ carries the diagonal $H$-action $h\cdot(x,y)\ceq (h\cdot x,h\cdot y)$. 
This is a bundle of vector spaces with group actions: for each $x\in V$, the fibre 
\[
\Lambda_x = \{ y\in V^*,\ [x,y]=0\}
\]
caries a natural action of $Z_{H}(x)$. 
It is not always the case that $\Lambda_x$ is a prehomogeneous vector space.

\begin{proposition}
If $V$ is one of the 5 prehomogeneous vector spaces appearing in Proposition~\ref{prop:PHV} then $\Lambda$ is a bundle of prehomogeneous vector spaces.
\end{proposition}

When the conormal variety $\Lambda$ to $V$ is a bundle of prehomogeneous vector spaces, we write $\Lambda_x^\text{sreg}$ for the open dense $Z_H(x)$-orbit in $\Lambda_x$ and $\Lambda^\text{reg}_{C}$ for the open dense $H$-orbit in $\Lambda_C$, for each $H$-orbit $C\subseteq V$. 

In this section we study the conormal varieties to the prehomogeneous vector spaces appearing in Proposition~\ref{prop:PHV}. In each case we:
\begin{itemize}[leftmargin=15pt]
\item find all $H$-orbits $C\subseteq V$ and all dual $H$-orbits $C^*\subseteq V^*$;
\item calculate each equivariant fundamental group $A_C \ceq \pi_0(Z_{H}(x))$, for $x\in C$;
\item enumerate the simple objects in $\Perv_{H}(V)$;
\item find the Fourier transform of each simple equivariant perverse sheaf on $V$;
\item show that the conormal variety $\Lambda$ is a bundle of prehomogeneous vector spaces;
\item calculate the equivariant fundamental group $A^\ABV_{C}$ of $\Lambda_{C}^\text{sreg}$, for each $H$-orbit $C$ in $V$;
\item calculate $\NEvs_{C_j} \IC(\mathcal{L}_{C_i}) \in \Loc_{H}(\Lambda_{C_j}^\text{sreg})$ for every simple $\IC(\mathcal{L}_{C_i})$ in $\Perv_{H}(V)$, where $\NEvs$ is defined in \cite{CFMMX}*{Section 7.10}.
\end{itemize}
These geometric calculations, and their applications to representation theory, follow the ideas presented in the examples treated in \cite{CFMMX}*{Part II} and in \cite{CFZ:cubics}. 

\begin{enumerate}
\item[(\ref{geocase-0})]
$\Perv_{H}(0)$ is the category of finite dimensional representations of $\pi_0(H)$. 
Then $\Lambda = 0$ and the functor $\NEvs$ is the identity. 

\item[(\ref{geocase-1})]
For the action of $\GL_1$ on $\mathbb{A}^1$ given by scalar multiplication there are two ${\GL_1}$-orbits: $C_{-} = \{0\}$ and its complement, $C_{1}$.
Then $\Perv_{{\GL_1}}(\mathbb{A}^1)$ has two simple objects, $\IC(\1_{C_0}) = \1_{C_0}^![0]$, the  extension-by-zero of $\1_{C_0}$ and $\IC(\1_{C_1}) = \1_{\mathbb{A}^1}[1]$. 
The conormal varieties are given as follows.
First observe that $\Lambda_{0} = T_{0}(\mathbb{A}^1) \iso \mathbb{A}^1$ and $\Lambda_{0}^\text{sreg} \iso \{ y\in \mathbb{A}^1 \tq y\ne 0\}$; note that this is a single orbit under the $\GL_1$-action. 
If $x\in C_1$ then $\Lambda_x = \{ (x,y) \tq xy =0 \} =\{ (x,0)\} $ which itself is a $\GL_1$-orbit.
In this way we see that $T^*_{C_0}(\mathbb{A}^1)^\text{sreg}$ and $T^*_{C_1}(\mathbb{A}^1)^\text{sreg}$ are both isomorphic, as $\GL_1$-spaces, to $C_1$ and consequently, equivariant local systems on these spaces are naturally identified with finite-dimensional vector spaces.
The functor $\NEvs : \Perv_{\GL_1}(\mathbb{A}^1) \to \Loc_{\GL_1}(T^*_{\GL_1}(\mathbb{A}^1)_\text{sreg})$ is given by the following table:
\[
\begin{array}{c| cc | c}
\Perv_{\GL_1}(\mathbb{A}^1) & \Loc_{\GL_1}(\Lambda_{C_0}^\text{sreg}) & \Loc_{\GL_1}(\Lambda_{C_1}^\text{sreg}) & \text{Fourier} \\
\hline\hline
\IC(\1_{C_0}) & \1_{\Lambda_{C_0}^\text{sreg}} & 0 & \IC(\1_{C_1}) \\
\IC(\1_{C_1}) & 0 & \1_{\Lambda_{C_1}^\text{sreg}} & \IC(\1_{C_0})
\end{array}
\]
\item[(\ref{geocase-2})]
For the twisted matrix multiplication action of $\GL_2$ on $\mathbb{A}^2$, again there are only two $H$-orbits, $C_0=\{(0,0)\}$ and the complement $C_1$ and only two simple objects in $\Perv_{\GL_2}(\mathbb{A}^2)$ are $\IC(\1_{C_0}) = \1_{C_0}^![0]$ and $\IC(\1_{C_1}) = \1_{\mathbb{A}^2}[2]$. 
Again, $T^*_{C_0}(\mathbb{A}^2)^\text{sreg}$ and $T^*_{C_1}(\mathbb{A}^2)^\text{sreg}$ are both isomorphic, as $\GL_2$-spaces, to ${\bar C_1}$, which is a single $\GL_2$-orbit and also has trivial equivariant fundamental group.
The functor $\NEvs : \Perv_{\GL_2}(\mathbb{A}^2) \to \Loc_{\GL_2}(T^*_{\GL_2}(\mathbb{A}^2)_\text{sreg})$ is given by
\[
\begin{array}{c| cc | c }
\Perv_{\GL_2}(\mathbb{A}^2) & \Loc_{\GL_2}(\Lambda_{C_0}^\text{sreg}) & \Loc_{\GL_2}(\Lambda_{C_1}^\text{sreg}) & \text{Fourier} \\
\hline\hline
\IC(\1_{C_0}) & \1_{\Lambda_{C_0}^\text{sreg}} & 0 & \IC(\1_{C_1})\\
\IC(\1_{C_1}) & 0 & \1_{\Lambda_{C_1}^\text{sreg}} &\IC(\1_{C_0}) 
\end{array}
\]
\item[(\ref{geocase-toric})]
Consider the action of $H= \GL_1^2$ on $V=\mathbb{A}^2$ by $(t_1,t_2).(x_1,x_2) = (t_1x_1,t_1t_2^nx_2)$ for positive integer $n$.
There are four $H$-orbits in this case: $C_0= \{(0,0)\}$, $C_1=\{ (x_1,0)\tq x_1\ne 0\}$, $C_2=\{ (0,x_2)\tq x_2\ne 0\}$ and the open orbit $C_3 = \{ (x_1,x_2)\tq x_1x_2\ne 0\}$.
The equivariant fundamental groups of these orbits are trivial with the exception of $C_3$, for which $A_{C_3}$ is the group of $n$-th roots-of-unity.
There are $4+ (n-1)$ simple objects in this category: $\IC(\1_{C_0}) = \1_{C_0}^!$, $\IC(\1_{C_1})= \1_{\bar C_1}^![1]$, $\IC(\1_{C_2}) = \1_{\bar C_2}^![1]$, $\IC(\1_{C_3}) = \1_{\mathbb{A}^2}[2]$ and $\IC(\vartheta_{C_3})$, where $\vartheta$ is a non-trivial character of the group of $n$th roots-of-unity. 
For $x\in C_0$ or $C_3$ we have $\Lambda_x \iso {\bar C_3}$ which has a unique open $\GL_1^2$-orbit.
For $x\in C_1$, $\Lambda_x = \{ ((x_1,0),(y_1,y_2)) \in T^*(V)\tq x_1y_1=0 \} = \{ ((x_1,0),(0,y_2)) \tq y_2\} \iso {\bar C_2}$ which also has a unique $\GL_1^2$-orbit; likewise for $x\in C_2$, $\Lambda_x \iso {\bar C_1}$.
It follows that $\Lambda_{C}$ has a unique open $H$-orbit, $\Lambda_{C}^\text{sreg}$, for every $H$-orbit $C\subset V$.
A simple calculation shows that the equivariant fundamental group of $\Lambda_{C}^\text{sreg}$ is the group $\langle \theta_n \rangle$ of $n$th roots-of-unity, for every $H$-orbit $C\subset V$.
The functor $\NEvs$ is given by the following table, where the last row is removed if $n=1$ and duplicated if $n>2$ for every non-trivial character $\vartheta$.
\[
\begin{array}{c | cccc | c}
\Perv_{ \GL_1^2}(\mathbb{A}^2) & \Loc_{ \GL_1^2}(\Lambda_{C_0}^\text{sreg}) & \Loc_{\GL_1^2}(\Lambda_{C_1}^\text{sreg}) & \Loc_{\GL_1^2}(\Lambda_{C_2}^\text{sreg}) & \Loc_{\GL_1^2}(\Lambda_{C_3}^\text{sreg}) & \text{Fourier } \\
\hline\hline
\IC(\1_{C_0}) & \1_{\Lambda_{C_0}^\text{sreg}} & 0 & 0 & 0 & \IC(\1_{C_3})\\
\IC(\1_{C_1}) & 0 & \1_{\Lambda_{C_1}^\text{sreg}} & 0 & 0 & \IC(\1_{C_2})\\
\IC(\1_{C_2}) & 0 & 0 & \1_{\Lambda_{C_2}^\text{sreg}} & 0 & \IC(\1_{C_1})\\
\IC(\1_{C_3}) & 0 & 0 & 0 & \1_{\Lambda_{C_3}^\text{sreg}} & \IC(\1_{C_0})\\
\hline
\IC(\vartheta_{C_3}) & \vartheta_{\Lambda_0^\text{sreg}} & \vartheta_{\Lambda_1^\text{sreg}} & \vartheta_{\Lambda_1^\text{sreg}} & \vartheta_{\Lambda_2^\text{sreg}} & \IC(\vartheta_{C_3})
\end{array}
\]

\item[(\ref{geocase-sub})]
The functor $\NEvs$ for $H= \GL_2$ acting on $V=\mathbb{A}^4$ by $h.x=\det^{-1}(h)\Sym^3(h)(x)$ (nomalized Symmetric cube) was calculated in \cite{CFZ:cubics}. We recall the result in the table below, in which $\vartheta \ceq \vartheta_2$ is the sign character of $\langle \theta_2\rangle$ and, as explained in Section~\ref{ssec:notation}, $\varepsilon$ is the sign character of $S_3$ and $\varrho$ is the character of the reflection representation of $S_3$.
\[
\begin{array}{ c | cccc | c }
\Perv_{\GL_2}(V) & \Loc_{\GL_2}(\Lambda_{C_0}^\mathrm{sreg}) &  \Loc_{\GL_2}(\Lambda_{C_1}^\mathrm{sreg}) &  \Loc_{\GL_2}(\Lambda_{C_2}^\mathrm{sreg}) &  \Loc_{\GL_2}(\Lambda_{C_3}^\mathrm{sreg}) & \text{Fourier} \\
\hline\hline
\IC(\1_{C_0}) & \1_{\Lambda_0^\mathrm{reg}} & 0 & 0 & 0 & \IC(\1_{C_3})\\
\IC(\1_{C_1})  & \varrho_{\Lambda_0^\mathrm{reg}} & \1_{\Lambda_1^\mathrm{reg}} & 0 & 0 & \IC(\varrho_{C_3}) \\
\IC(\1_{C_2}) & 0 & \vartheta_{\Lambda_1^\mathrm{reg}} & \1_{\Lambda_2^\mathrm{reg}} & 0 & \IC(\1_{C_2}) \\
\IC(\1_{C_3}) & 0 & 0 & 0 & \1_{\Lambda_3^\mathrm{reg}} & \IC(\1_{C_1}) \\
\IC(\varrho_{C_3}) &  0 & 0 &\vartheta_{\Lambda_2^\mathrm{reg}} & \varrho_{\Lambda_3^\mathrm{reg}} & \IC(\1_{C_0}) \\
\IC(\varepsilon_{C_3}) & \varepsilon_{\Lambda_0^\mathrm{reg}} & \1_{\Lambda_1^\mathrm{reg}}  & \vartheta_{\Lambda_2^\mathrm{reg}} & \varepsilon_{\Lambda_3^\mathrm{reg}} & \IC(\varepsilon_{C_3})
\end{array}
\]
\end{enumerate}

\section{Unipotent representations for $G_2$}\label{sec:unipG2}

In this section we use the classification of infinitesimal parameters for $G_2(F)$ in Proposition~\ref{prop:infcases}, together with the results of Section~\ref{ssec:geometry}, to enumerate all irreducible unipotent representations of $G_2(F)$ and to find their Langlands parameters. In this way we give an explicit form of the Langlands correspondence for unipotent representations of $G_2(F)$.  At the same time, we will find the ABV-packet coefficients and start assembling the proof of Theorem~\ref{thm:coefficients}.

\subsection{The Langlands correspondence}\label{ssec:LLC}


The Langlands correspondence for unipotent representations of $G_2(F)$ is a special case of \cite{Lusztig:Classification1}, though that paper uses a normalization of the correspondence that is different from the one we give here.
Instead, here we use a normalization of the local Langlands correspondence that satisfies the desiderata of \cite{GGP}*{\S 9}; in particular, it satisfies \cite{Gross-Prasad}*{Conjecture 2.6}, which specifies the Langlands parameter of generic representations for a particular infinitesimal parameter. 
Moreover, for a choice of hyperspecial group, our normalization also specifies the enhanced Langlands parameter of spherical representation for a particular infinitesimal parameter, {\it i.e.}, if $\pi\in \Pi_\phi(G_2(F))$ is spherical, then $\phi$ is trivial on the $\SL_2(\C)$ part and the corresponding representation of $\wh A_{\phi}$ is trivial. 
This choice of normalization is compatible with the requirement in the work of Arthur \cite{Arthur:book} and agrees with the local Langlands correspondence appearing in \cite{Solleveld}.
As in \cite{Waldspurger:SOimpair}*{p. 804}, this normalization is obtained by composing Lusztig's correspondence with the Aubert involution.

The Langlands correspondence for unipotent representations of $G_2(F)$ is given in Table~\ref{table:LLC}, also in Table~\ref{table:LA-packets}, making reference to the classification of unramified infinitesimal parameters in Proposition~\ref{prop:infcases} and the enumeration of unramified Langlands parameters for $G_2(F)$ given in Section~\ref{ssec:LA-packets}.
Both tables use notation from \cite{Muic} for irreducible admissible representations of $G_2(F)$, modified as explained in Section~\ref{ssec:notation}.
In Table~\ref{table:LLC}, characters of $A_\phi$ are listed in the column with $\phi$ at the top and the $L$-packet $\Pi_\phi(G_2(F))$ is found by gathering together the representations $\pi$ for which the character of $A_\phi$ is non-zero, thus defining the bijection
\[
\Pi_\phi(G_2(F)) \to \widehat{A_\phi},
\]
for all unramified Langlands parameters for $G_2(F)$.
In this way, Table~\ref{table:LLC} defines $\pi(\phi,\rho)$ for each enhanced Langlands parameter $(\phi,\rho)$, for the $\rho\in \widehat{A_\phi}$.

There are precisely three non-singleton $L$-packets of unipotent representations of $G_2(F)$: $\Pi_{\phi_{\ref{infcase:4-D2}d}}(G_2(F))$, $\Pi_{\phi_{\ref{infcase:6-A2}d}}(G_2(F))$ and $\Pi_{\phi_{\ref{infcase:8-sub}d}}(G_2(F))$, of which the first has order $2$ and the other two have order $3$. 
Note that the group $A_\phi$ is trivial except when $\phi$ is $\phi_{\ref{infcase:4-D2}d}$, $\phi_{\ref{infcase:6-A2}d}$ or $\phi_{\ref{infcase:8-sub}d}$, in which cases $A_\phi$ is $\langle \theta_2\rangle$, $\langle \theta_3\rangle$ and $S_3$, respectively. 
These three Langlands parameters are all of Arthur type and are distinguished by the fact that they have elliptic endoscopy, as explained in Section~\ref{sec:geoendo}.

\begin{table}[htp]
\caption{The local Langlands correspondence for unipotent representations of $G_2(F)$. See Section~\ref{ssec:LLC} for how to read this table.}
\label{table:LLC}
\begin{center}
\rotatebox{90}{
\resizebox{\textheight-46pt}{!}{
$
\begin{array}{r|c|cc|cc|cc|cccc|cc|cccc|cccc|cccc}
 & \multicolumn{ 1 }{ c | }{ \lambda_{\ref{infcase:0}} }  & \multicolumn{ 2 }{ c| }{ \lambda_{\ref{infcase:1-short}} } &  \multicolumn{ 2 }{ c| }{ \lambda_{\ref{infcase:2-long}} } &  \multicolumn{ 2 }{ c| }{ \lambda_{\ref{infcase:3}} } & \multicolumn{ 4 }{ c| }{ \lambda_{\ref{infcase:4-D2}} } & \multicolumn{ 2 }{ c| }{ \lambda_{\ref{infcase:5}} } & \multicolumn{ 4 }{ c| }{ \lambda_{\ref{infcase:6-A2}} } & \multicolumn{ 4 }{ c| }{ \lambda_{\ref{infcase:7-reg}} } & \multicolumn{ 4 }{ c }{ \lambda_{\ref{infcase:8-sub}} }\\
 &&&&&&&&&&&&&&&&&&&&&&&&&\PGL_3\\
&T&T&\GL_2&T&\GL_2&T&\GL_2&T&\GL_2&\GL_2&\SO_4&T&\GL_2&T&\GL_2&\GL_2&\PGL_3&T&\GL_2&\GL_2&G_2&T&\GL_2&\GL_2&\SO_4\\
& \phi_{\ref{infcase:0}} & \phi_{\ref{infcase:1-short}a} & \phi_{\ref{infcase:1-short}b} & \phi_{\ref{infcase:2-long}a} & \phi_{\ref{infcase:2-long}b} & \phi_{\ref{infcase:3}a} & \phi_{\ref{infcase:3}b} & \phi_{\ref{infcase:4-D2}a} & \phi_{\ref{infcase:4-D2}b} & \phi_{\ref{infcase:4-D2}c} & \phi_{\ref{infcase:4-D2}d} & \phi_{\ref{infcase:5}a} & \phi_{\ref{infcase:5}b} & \phi_{\ref{infcase:6-A2}a} & \phi_{\ref{infcase:6-A2}b} & \phi_{\ref{infcase:6-A2}c} & \phi_{\ref{infcase:6-A2}d} & \phi_{\ref{infcase:7-reg}a} & \phi_{\ref{infcase:7-reg}b} & \phi_{\ref{infcase:7-reg}c} & \phi_{\ref{infcase:7-reg}d} & \phi_{\ref{infcase:8-sub}a} & \phi_{\ref{infcase:8-sub}b} & \phi_{\ref{infcase:8-sub}c} & \phi_{\ref{infcase:8-sub}d} \\
\hline 
I(\mu_1\nu^{a_1}\otimes \mu_2\nu^{a_2}) 				&1&&&&&&&&&&&&&&&&&&&&&&&&\\
\hline
I_{\gamma_2}(a-1/2, \mu\circ \det) 					&&1&0&&&&&&&&&&&&&&&&&&&&&&\\
I_{\gamma_2}(a-1/2, \mu\otimes\St_{\GL_2}) 			&&0&1&&&&&&&&&&&&&&&&&&&&&&\\
\hline
I_{\gamma_1}(a-1/2, \mu\circ \det) 					&&&&1&0&&&&&&&&&&&&&&&&&&&&\\
I_{\gamma_1}(a-1/2, \mu\otimes\St_{\GL_2}) 			&&&&0&1&&&&&&&&&&&&&&&&&&&&\\
\hline
I_{{\gamma_1}}(1/6, \theta_3^n \circ \det)  			&&&&&&1&0&&&&&&&&&&&&&&&&&&\\
I_{{\gamma_1}}(1/6, \theta_3^n\otimes\St_{\GL_2}) 		&&&&&&0&1&&&&&&&&&&&&&&&&&&\\
\hline
J_{\gamma_2}(1,I^{\GL_2}(1\otimes\theta_2)) 			&&&&&&&&1&0&0&0&&&&&&&&&&&&&&\\
J_{{\gamma_1}}(1/2,\theta_2\otimes \St_{\GL_2}) 		&&&&&&&&0&1&0&0&&&&&&&&&&&&&&\\
J_{{\gamma_2}}(1/2,\theta_2\otimes \St_{\GL_2}) 		&&&&&&&&0&0&1&0&&&&&&&&&&&&&&\\
\pi(\theta_2) 									&&&&&&&&0&0&0&1&&&&&&&&&&&&&&\\
I_0(G_2[-1]) 									&&&&&&&&0&0&0&\vartheta_2&&&&&&&&&&&&&&\\
\hline
I_{\gamma_2}(3/2,1_{\GL_2}) 						&&&&&&&&&&&&1&0&&&&&&&&&&&&\\
I_{\gamma_2}(3/2,\St_{\GL_2})						&&&&&&&&&&&&0&1&&&&&&&&&&&&\\
\hline
J_{\gamma_2}(1,I^{\GL_2}(\theta_3\otimes\theta_3^{-1})) &&&&&&&&&&&&&&1&0&0&0&&&&&&&&\\
J_{{\gamma_1}}(1/2,\theta_3^{2}\otimes\St_{\GL_2}) 	&&&&&&&&&&&&&&0&1&0&0&&&&&&&&\\
J_{{\gamma_1}}(1/2,\theta_3^{}\otimes\St_{\GL_2})  		&&&&&&&&&&&&&&0&0&1&0&&&&&&&&\\
\pi(\theta_3) 									&&&&&&&&&&&&&&0&0&0&1&&&&&&&&\\
I_0(G_2[\theta_3]) 								&&&&&&&&&&&&&&0&0&0&\vartheta_3&&&&&&&&\\
I_0(G_2[\theta_3^{2}]) 							&&&&&&&&&&&&&&0&0&0&\vartheta_3^2&&&&&&&&\\
\hline
1_{G_2}  										&&&&&&&&&&&&&&&&&&1&0&0&0&&&&\\
J_{\gamma_1}(3/2,\St_{\GL_2}) 					&&&&&&&&&&&&&&&&&&0&1&0&0&&&&\\
J_{{\gamma_2}}(5/2,\St_{\GL_2}) 					&&&&&&&&&&&&&&&&&&0&0&1&0&&&&\\
 \St_{G_2} 									&&&&&&&&&&&&&&&&&&0&0&0&1&&&&\\
 \hline
J_{{\gamma_2}}(1,I^{\GL_2}(1\otimes 1)) 				&&&&&&&&&&&&&&&&&&&&&&1&0&0&0\\
J_{\gamma_1}(1/2,\St_{\GL_2}) 					&&&&&&&&&&&&&&&&&&&&&&0&1&0&0\\
J_{{\gamma_2}}(1/2,\St_{\GL_2}) 					&&&&&&&&&&&&&&&&&&&&&&0&0&1&0\\
\pi(1)' 										&&&&&&&&&&&&&&&&&&&&&&0&0&0&1\\
\pi(1) 										&&&&&&&&&&&&&&&&&&&&&&0&0&0&\varrho\\
I_0(G_2[1]) 									&&&&&&&&&&&&&&&&&&&&&&0&0&0&\varepsilon
\end{array}
$
}
}
\end{center}
\end{table}

Recall that we have classified all unramified infinitesimal parameters for $G_2(F)$ in Section~\ref{ssec:infcases}. Working through the 9 cases in Proposition~\ref{prop:infcases}, in Section~\ref{ssec:LA-packets} we find, in each case, all Langlands parameters $\phi$ with given infinitesimal parameter $\lambda$. Using results from \cite{CFMMX}, we compute the component group $A_\phi$. We find that these groups are trivial in all cases except the unique tempered Langlands parameter with infinitesimal parameters given by Cases~\ref{infcase:4-D2},  \ref{infcase:6-A2} and \ref{infcase:8-sub}. In all other cases, except these three, since $A_\phi$ is trivial, we use the Langlands classification to find the corresponding admissible representation and note that $\Pi_\phi(G_2(F)) \to \Rep(A_\phi)$ is trivial in these cases.
This is done case-by-case in Section~\ref{ssec:LA-packets}. 
It then remains to consider only the Langlands parameters with infinitesimal parameters given by Cases~\ref{infcase:4-D2},  \ref{infcase:6-A2} and \ref{infcase:8-sub}, specifically the three tempered ({\it i.e.}, bounded upon restriction to $W_F$) parameters denoted by $\phi_{\ref{infcase:4-D2}d}$, $\phi_{\ref{infcase:6-A2}d}$ and $\phi_{\ref{infcase:8-sub}d}$. These three are of Arthur type and are lifted from Arthur parameters of elliptic endoscopic groups, as explained in Section~\ref{sec:geoendo}. We found the local Langlands correspondence for parameters with infinitesimal parameter $\lambda_{\ref{infcase:8-sub}}$ in \cite{CFZ:cubics}*{Theorem 2.5}. The local Langlands correspondence for parameters with infinitesimal parameter $\lambda_{\ref{infcase:4-D2}}$ and $\lambda_{\ref{infcase:6-A2}}$  is obtained by similar arguments.

\subsection{Vogan's version of the local Langlands correspondence}\label{ssec:VC}

Let $\lambda$ be an unramifed infinitesimal parameter of $G_2(F)$ and let $\Pi_\lambda(G_2(F))$ be the set of (equivalence classes of) irreducible smooth representations of $G_2(F)$ with infinitesimal parameter $\lambda$. Associated with $\lambda$, we have the Vogan variety $V_\lambda$ and the group $H_\lambda$. If $\phi$ is an unramified Langlands parameter with infinitesimal parameter $\lambda$, let $C_\pi\subset V_\lambda$ be the $H_\lambda$-orbit corresponding to $\phi$. For use below, we define $\dim(\phi) \ceq \dim C_\phi$ and $\dim(\pi) \ceq \dim C_\phi$ when $\phi$ is the Langlands parameter for $\pi$.

Let $\Pi_\phi(G_2(F))$ be the $L$-packet of $\phi$, whose existence was proved in \cite{Lusztig:Classification1}. Since $A_\phi$ is both the component group of $\phi$ and the equivariant fundamental group of $C_\phi$, the Langlands correspondence, as presented in Section~\ref{ssec:LLC}, now determines a bijection 
\[
\mathcal{L} : \Pi_\phi(G_2(F))\to \Loc_{H_\lambda}(C_\phi)_{/\mathrm{iso}}^{\mathrm{simple}}
\]
such that $\mathcal{L}(\pi)$ corresponds to the representation $\langle\ ,\pi\rangle_\phi$ of $A_\phi$.
Letting $\phi$ run over all $L$-parameters with infinitesimal parameter $\lambda$, this defines a bijection
\[
\mathcal{P}: \Pi_\lambda(G_2(F))\to \Perv_{H_\lambda}(V_{\lambda})_{/\mathrm{iso}}^{\mathrm{simple}}
\]
such that $\mathcal{P}(\pi) = \IC(\mathcal{L}(\pi))$. We make use of this bijection in Section~\ref{ssec:LA-packets}.
In Case~\ref{infcase:8-sub}, this bijection was made explicit in \cite{CFZ:cubics}. 

\subsection{Kazhdan-Lusztig conjecture}

Using the case-by-case calculations of Section~\ref{ssec:LA-packets}, we have verified the Kazhdan-Lusztig conjecture, as stated in \cite{CFMMX}*{Section 10.3.3}, for $G_2(F)$; this extends the result of \cite{CFZ:cubics}*{Section 2.10}. 

\subsection{Aubert involution and the Fourier transform}
We have also confirmed that the Aubert involution, as calculated in \cite{Muic}, matches the Fourier transform of the corresponding simple perverse sheaves, using the bijection of Section~\ref{ssec:VC}, for all unipotent representations of $G_2(F)$, confirming the expectation of \cite{CFMMX}*{Section 10.3.4} and extending the result of \cite{CFZ:cubics}*{Section 2.13}.

\subsection{A generalization of the component group of an Arthur parameter}

We now introduce the algebraic group $\mathcal{S}^\ABV_\phi$ promised in the introduction.

\begin{corollary}\label{cor:prehomogeneous}
Let $\lambda : W_F \to \Lgroup{G_2}$ be an unramified infinitesimal parameter for $G_2(F)$. Let $V_\lambda$ be the moduli space of Langlands parameters with infinitesimal parameter $\lambda$. The conormal variety $\Lambda_\lambda$ is a bundle of prehomogeneous vector spaces over $V_\lambda$. 
\end{corollary}

\begin{proof}
Infinitesimal parameters $\lambda$ are classified by by the associated prehomogeneous vector space $V_\lambda$ in Proposition~\ref{prop:infcases} and these in turn are expressed in terms of the five prehomogeneous vector spaces in Proposition~\ref{prop:PHV}. In Section~\ref{ssec:geometry} we saw that the conormal variety of each of these five prehomogeneous vector spaces all have the property that they're bundles of prehomogeneous vector spaces. 
\end{proof}

\begin{definition}\label{def:SABV}
Notation as above.
Let $x$ be the point for $\phi$ in the moduli space $V$, where $\lambda$ is the infinitesimal parameter of $\phi$. 
Let $\Lambda_\lambda$ be the conormal variety and recall, by Corollary~\ref{cor:prehomogeneous}, that this is a bundle of prehomogeneous vector spaces above $V_\lambda$. 
Set $\mathcal{S}^\ABV_\phi \ceq  Z_{\dualgroup{G_2}}(x,y)$, where $x\in V$ is a point in the moduli space $V$ that maps to the Langlands parameter $\phi$ and for $y\in \Lambda_x^\text{sreg}$. This variety is independent of the choice of $y$.
Also set $A^\ABV_\phi \ceq \pi_0(\mathcal{S}^\ABV_\phi)$, equipped with the natural homomorphism $\mathcal{S}^\ABV_\phi\to A^\ABV_\phi$.
It follows immediately from \cite{CFMMX}*{Prop. 6.1.1}, $\phi$ is of Arthur type $\psi$, then $\mathcal{S}^\ABV_\phi = \mathcal{S}_\psi$.
\end{definition}

Definition~\ref{def:SABV} is needed to find the ABV-packet coefficients $\langle \ , \ \rangle: \mathcal{S}^\ABV_\phi\times \Pi^\ABV_\phi(G_2(F)) \to {\bar\QQ}$, which we do in the next section.

\subsection{Calculation of ABV-packet coefficients}\label{ssec:LA-packets}

In this section we calculate $\pi\mapsto \langle  , \pi \rangle $ for every unipotent representation $\pi$ of $G_2(F)$.
The results are summarized in Table~\ref{table:ABV}, in which characters of $A^\ABV_\phi$ are listed in the column with $\phi$ at the top, classified by infinitesimal parameters. We list the endoscopic groups for $\phi$ above $\phi$.
ABV-packets $\Pi^\ABV_\phi(G_2(F))$ are found by gathering together the representations $\pi$ for which the character of $A^\ABV_\phi$ is non-zero. 

\begin{table}[htp]
\caption{ABV-packets and ABV-packet coefficients for unipotent representations of $G_2(F)$. See Section~\ref{ssec:LA-packets} for how to read this table.}
\label{table:ABV}
\begin{center}
\rotatebox{90}{
\resizebox{\textheight-46pt}{!}{
$
\begin{array}{r|c|cc|cc|cc|cccc|cc|cccc|cccc|cccc}
 & \multicolumn{ 1 }{ c | }{ \lambda_{\ref{infcase:0}} }  & \multicolumn{ 2 }{ c| }{ \lambda_{\ref{infcase:1-short}} } &  \multicolumn{ 2 }{ c| }{ \lambda_{\ref{infcase:2-long}} } &  \multicolumn{ 2 }{ c| }{ \lambda_{\ref{infcase:3}} } & \multicolumn{ 4 }{ c| }{ \lambda_{\ref{infcase:4-D2}} } & \multicolumn{ 2 }{ c| }{ \lambda_{\ref{infcase:5}} } & \multicolumn{ 4 }{ c| }{ \lambda_{\ref{infcase:6-A2}} } & \multicolumn{ 4 }{ c| }{ \lambda_{\ref{infcase:7-reg}} } & \multicolumn{ 4 }{ c }{ \lambda_{\ref{infcase:8-sub}} }\\
&&&&&&&&&&&&&&&&&&&&&&&&&\PGL_3\\
&T&T&\GL_2&T&\GL_2&T&\GL_2&T&\GL_2&\GL_2&\SO_4&T&\GL_2&T&\GL_2&\GL_2&\PGL_3&T&\GL_2&\GL_2&G_2&T&\GL_2&\GL_2&\SO_4\\
& \phi_{\ref{infcase:0}} & \phi_{\ref{infcase:1-short}a} & \phi_{\ref{infcase:1-short}b} & \phi_{\ref{infcase:2-long}a} & \phi_{\ref{infcase:2-long}b} & \phi_{\ref{infcase:3}a} & \phi_{\ref{infcase:3}b} & \phi_{\ref{infcase:4-D2}a} & \phi_{\ref{infcase:4-D2}b} & \phi_{\ref{infcase:4-D2}c} & \phi_{\ref{infcase:4-D2}d} & \phi_{\ref{infcase:5}a} & \phi_{\ref{infcase:5}b} & \phi_{\ref{infcase:6-A2}a} & \phi_{\ref{infcase:6-A2}b} & \phi_{\ref{infcase:6-A2}c} & \phi_{\ref{infcase:6-A2}d} & \phi_{\ref{infcase:7-reg}a} & \phi_{\ref{infcase:7-reg}b} & \phi_{\ref{infcase:7-reg}c} & \phi_{\ref{infcase:7-reg}d} & \phi_{\ref{infcase:8-sub}a} & \phi_{\ref{infcase:8-sub}b} & \phi_{\ref{infcase:8-sub}c} & \phi_{\ref{infcase:8-sub}d} \\
\hline 
I(\mu_1\nu^{a_1}\otimes \mu_2\nu^{a_2}) 				&1&&&&&&&&&&&&&&&&&&&&&&&&\\
\hline
I_{\gamma_2}(a-1/2, \mu\circ \det) 					&&1&0&&&&&&&&&&&&&&&&&&&&&&\\
I_{\gamma_2}(a-1/2, \mu\otimes\St_{\GL_2}) 			&&0&1&&&&&&&&&&&&&&&&&&&&&&\\
\hline
I_{\gamma_1}(a-1/2, \mu\circ \det) 					&&&&1&0&&&&&&&&&&&&&&&&&&&&\\
I_{\gamma_1}(a-1/2, \mu\otimes\St_{\GL_2}) 			&&&&0&1&&&&&&&&&&&&&&&&&&&&\\
\hline
I_{{\gamma_1}}(1/6, \theta_3^n \circ \det)  			&&&&&&1&0&&&&&&&&&&&&&&&&&&\\
I_{{\gamma_1}}(1/6, \theta_3^n\otimes\St_{\GL_2}) 		&&&&&&0&1&&&&&&&&&&&&&&&&&&\\
\hline
J_{\gamma_2}(1,I^{\GL_2}(1\otimes\theta_2)) 			&&&&&&&&1&0&0&0&&&&&&&&&&&&&&\\
J_{{\gamma_1}}(1/2,\theta_2\otimes \St_{\GL_2}) 		&&&&&&&&0&1&0&0&&&&&&&&&&&&&&\\
J_{{\gamma_2}}(1/2,\theta_2\otimes \St_{\GL_2}) 		&&&&&&&&0&0&1&0&&&&&&&&&&&&&&\\
\pi(\theta_2) 									&&&&&&&&0&0&0&1&&&&&&&&&&&&&&\\
I_0(G_2[-1]) 									&&&&&&&&\vartheta_2&\vartheta_2&\vartheta_2&\vartheta_2&&&&&&&&&&&&&&\\
\hline
I_{\gamma_2}(3/2,1_{\GL_2}) 						&&&&&&&&&&&&1&0&&&&&&&&&&&&\\
I_{\gamma_2}(3/2,\St_{\GL_2})						&&&&&&&&&&&&0&1&&&&&&&&&&&&\\
\hline
J_{\gamma_2}(1,I^{\GL_2}(\theta_3\otimes\theta_3^{-1})) &&&&&&&&&&&&&&1&0&0&0&&&&&&&&\\
J_{{\gamma_1}}(1/2,\theta_3^{2}\otimes\St_{\GL_2}) 	&&&&&&&&&&&&&&0&1&0&0&&&&&&&&\\
J_{{\gamma_1}}(1/2,\theta_3^{}\otimes\St_{\GL_2})  		&&&&&&&&&&&&&&0&0&1&0&&&&&&&&\\
\pi(\theta_3) 									&&&&&&&&&&&&&&0&0&0&1&&&&&&&&\\
I_0(G_2[\theta_3]) 								&&&&&&&&&&&&&&\vartheta_3&\vartheta_3&\vartheta_3&\vartheta_3&&&&&&&&\\
I_0(G_2[\theta_3^{2}]) 							&&&&&&&&&&&&&&\vartheta_3^2&\vartheta_3^2&\vartheta_3^2&\vartheta_3^2&&&&&&&&\\
\hline
1_{G_2}  										&&&&&&&&&&&&&&&&&&1&0&0&0&&&&\\
J_{\gamma_1}(3/2,\St_{\GL_2}) 					&&&&&&&&&&&&&&&&&&0&1&0&0&&&&\\
J_{{\gamma_2}}(5/2,\St_{\GL_2}) 					&&&&&&&&&&&&&&&&&&0&0&1&0&&&&\\
 \St_{G_2} 									&&&&&&&&&&&&&&&&&&0&0&0&1&&&&\\
 \hline
J_{{\gamma_2}}(1,I^{\GL_2}(1\otimes 1)) 				&&&&&&&&&&&&&&&&&&&&&&1&0&0&0\\
J_{\gamma_1}(1/2,\St_{\GL_2}) 					&&&&&&&&&&&&&&&&&&&&&&\varrho&1&0&0\\
J_{{\gamma_2}}(1/2,\St_{\GL_2}) 					&&&&&&&&&&&&&&&&&&&&&&0&\vartheta_2&1&0\\
\pi(1)' 										&&&&&&&&&&&&&&&&&&&&&&0&0&0&1\\
\pi(1) 										&&&&&&&&&&&&&&&&&&&&&&0&0&\vartheta_2&\varrho\\
I_0(G_2[1]) 									&&&&&&&&&&&&&&&&&&&&&&\varepsilon&1&\vartheta_2&\varepsilon
\end{array}
$
}
}
\end{center}
\end{table}%

\begin{definition}
Let $\phi$ be an unramified Langlands parameter for $G_2(F)$.
Following \cite{CFMMX}, we set 
$$\Pi_\phi^{\ABV}(G_2(F)) \ceq\{\pi\in \Pi_\lambda(G_2/F)|~ \NEvs_{C_\phi}(\mathcal{P}(\pi))\ne 0\}.$$
For $s\in \mathcal{S}^\ABV_\phi$ and $\pi\in \Pi^\ABV_\phi(G_2(F))$, ABV-packet coefficients are given by
\begin{equation}\label{eqn:coefficients}
\langle s , \pi \rangle\ceq \trace_{s} \NEvs_{C_\phi} \mathcal{P}(\pi).
\end{equation}
In this definition we make implicit use of the homomorphism $\mathcal{S}^\ABV_\phi\to A^\ABV_\phi$.
\end{definition}

Let $\lambda:W_F\to \wh {G_2}$ be an unramified infinitesimal parameter for $G_2(F)$.
Then $\lambda$ is determined by $\lambda(\Fr)=\wh m(u_1q^{a_1},u_2q^{a_2})$ where $u_1$ and $u_2$ are unitary complex numbers and $a_1$ and $a_2$ are real numbers.
Then
\begin{align*}
\lambda(w)
&=\wh m(u_1^{\ord(w)}, u_2^{\ord(w)})\cdot \wh m(|w|^{a_1},|w|^{a_2})\\
&=(2{{\hat\gamma}_1}^\vee+{{\hat\gamma}_2}^\vee)(u_2^{\ord(w)}|w|^{a_2}) \cdot ({{\hat\gamma}_1}^\vee+{{\hat\gamma}_2}^\vee)((u_1u_2^{-1})^{\ord(w)}|w|^{a_2-a_1}).
\end{align*}
Let $\chi=\chi_1\otimes\chi_2:T\to \C^\times$ be the dual character of $\lambda : W_F \to \dualgroup{T}$. By the above form of $\lambda$, we know that $\chi_1=\mu_2\nu^{a_2}$ and $\chi_2=\mu_1\mu_2^{-1}\nu^{a_1-a_2}$, {\it i.e.}, $\chi$ is given by 
\begin{equation}\label{eqn:dualcharacter}
\chi
=\mu_2 \nu^{a_2}\otimes \mu_1\mu_2^{-1}\nu^{a_1-a_2},
\end{equation}
where $\mu_i$ is the unitary unramified character of $F^\times$ defined by $\mu_i(\varpi)=u_i$.

\begin{enumerate}

\item[(Case \ref{infcase:0})]

Suppose $\phi : W'_F \to \Lgroup{G_2}$ is a Langlands parameter with infinitesimal parameter $\lambda_\text{\ref{infcase:0}}$ corresponding to Case~\ref{infcase:0} from Proposition~\ref{prop:infcases}.
Then $\lambda_\text{\ref{infcase:0}}:W_F\to \wh{G_2}$ is an unramified infinitesimal parameter such that $\dim V_{\lambda_\text{\ref{infcase:0}}}=0$, so ${\hat\gamma}(\lambda_\text{\ref{infcase:0}}(\Frob)) \ne q$ for every root ${\hat\gamma}\in {\hat R}$; this is Case~\ref{infcase:0} from Proposition~\ref{prop:infcases}.
If we write $\lambda_\text{\ref{infcase:0}}(\Frob) = {\hat m}(u_1q^a_1,u_2q^{a_2})$ then this case is equivalent to the condition 
\[
\{ u_2^2u_1^{-1}q^{2a_2-a_1}, u_1u_2^{-1}q^{a_1-a_2}, u_2q^{a_2},
 u_1q^{a_1}, u_1^2u_2^{-1}q^{2a_1-a_2}, u_1u_2q^{a_1+a_2}\} \cap \{ q, q^{-1} \} = \emptyset.
\]
In this case $H_{\lambda_\text{\ref{infcase:0}}}$ is the dual of an endoscopic group for $\wh{G_2}$; so, besides $\wh{G_2}$ itself, $H_{\lambda_\text{\ref{infcase:0}}}$ is $\SL_3$ or $\SO_4$ or one of the Levi subgroups $\GL_{2}^{{{\hat\gamma}_1}}$, $\GL_2^{{{\hat\gamma}_2}}$ or $\dualgroup{T}$. 
For each such $\lambda_\text{\ref{infcase:0}}$, there is a unique $\phi$ with $\lambda_\phi=\lambda_\text{\ref{infcase:0}}$, {\it i.e.}, $\phi(w,x)=\lambda_\text{\ref{infcase:0}}(w)$.
In this case, there is a unique Langlands parameter with infinitesimal parameter $\lambda_{\ref{infcase:0}}$ which is defined by $\phi_{\ref{infcase:0}}(w,x)=\lambda_{\ref{infcase:0}}(w)$.
If $a_1=a_2=0$, then $\chi_1$ and $\chi_2$ are unitary. Notice that there is a unique order 2 unramified character of $F^\times$ and thus $\chi_1,\chi_2$ cannot be two distinct order 2 unramified characters. In this case, $I(\chi_1\otimes \chi_2)$ is irreducible by \cite{Keys}. 
If one of $a_1,a_2$ is nonzero, then at least one of $\chi_1,\chi_2 $ is non-unitary. One easily translates the condition $\hat\gamma(\lambda_{\ref{infcase:0}}(\Fr))\ne q^{\pm 1}, \forall \hat \gamma\in R(\wh{G_2},\dualgroup{T})$ to the conditions on $\chi_1,\chi_2$, which are exactly the conditions of \cite{Muic}*{Proposition 3.1}. Thus, the induced representation $I(\chi_1\otimes \chi_2)$ is also irreducible by \cite{Muic}*{Proposition 3.1}. 
In both cases, the unramified local Langlands correspondence says that $\pi(\phi_{\ref{infcase:0}})=I(\chi_1\otimes \chi_2)$ and the $L$-packet is a singleton.
Here, $\phi_{\ref{infcase:0}}(w,x) = \lambda_{\ref{infcase:0}}(w)$, given above.
There are no coronal representations to the admissible representation $I(\chi_1\otimes \chi_2)$; accordingly, this representation is its own ABV-packet in this case. 
This is a direct consequence of the simple observation that there is but one simple object in $\Perv_{H_{\lambda_{\ref{infcase:0}}}}(V_{\lambda_{\ref{infcase:0}}})$; see also \ref{geocase-0}. 

\item[(Case \ref{infcase:1-short})]

Suppose $\phi : W'_F \to \Lgroup{G_2}$ is a Langlands parameter with infinitesimal parameter $\lambda_\text{\ref{infcase:1-short}}$ corresponding to Case~\ref{infcase:1-short} from Proposition~\ref{prop:infcases}.
Then $V_{\lambda_{\ref{infcase:1-short}}}=\mathrm{Span}\{X_{{{\hat\gamma}_1}+2{{\hat\gamma}_2}}\}$ and $({{\hat\gamma}_1}+2{{\hat\gamma}_2})(\lambda_{\ref{infcase:1-short}}(\Frob)) = q$ and ${\hat\gamma}(\lambda_{\ref{infcase:1-short}}(\Frob)) \ne q$ for every ${\hat\gamma}\in {\hat R}-\{\wh \gamma_1+2\wh \gamma_2\}$; in other words, if we write $ \lambda_{\ref{infcase:1-short}}(\Frob)=(u_1q^{a_1},u_2q^{a_2})$, then $u_1u_2=1$, $a_1+a_2 =1$ and
\[
\{ u^{-3}q^{2-3a}, u^2q^{2a-1}, u^{-1}q^{1-a}, u q^a,  u^3 q^{3a-1}\} \cap \{q,q^{-1}\} = \emptyset,
\]
where $u\ceq u_1$ and $a\ceq a_1$ and $H_{\lambda_{\ref{infcase:1-short}}}=\dualgroup{T}$ if $u^2q^{2a-1}\ne 1$ and $H_{\lambda_{\ref{infcase:1-short}}}=\GL_{2}^{\hat\gamma_1}$ if $u^2q^{2a-1}=1$. 

The action of $H_{\lambda_{\ref{infcase:1-short}}}$ on $V_{\lambda_{\ref{infcase:1-short}}}$ is given by 
\[
\begin{array}{rcll}
\wh m(x,y). rX_{{{\hat\gamma}_1}+2{{\hat\gamma}_2}}&=&xr X_{{{\hat\gamma}_1}+2{{\hat\gamma}_2}}, & \textrm{ if } u^2q^{2a-1}\ne 1,\\
g. rX_{{{\hat\gamma}_1}+2{{\hat\gamma}_2}}&=&\det(g)rX_{{{\hat\gamma}_1}+2{{\hat\gamma}_2}}, & \textrm{ if } u^2q^{2a-1}= 1.
\end{array}
\]
This action has two orbits: the closed (zero) orbit and the open orbit. The corresponding Langlands parameters are explicitly given by 
\[
\begin{array}{rcll}
\phi_{\ref{infcase:1-short}a}(w,x)&=&\lambda_{\ref{infcase:1-short}}(w), &\text{closed};\\
\phi_{\ref{infcase:1-short}b}(w,x)&=&\wh m(1,u^{\ord(w)}|w|^{a-1/2}) \iota_{{{\hat\gamma}_1}+2{{\hat\gamma}_2}}(x), & \text{open}.
\end{array}
\]
The simple objects in $\Perv_{H}(V)$ are described in \ref{geocase-1} using the base-change argument in the proof of Proposition~\ref{prop:PHV}. 

By Equation~\eqref{eqn:dualcharacter}, the dual character of $\lambda_{\ref{infcase:1-short}}$ is $\chi=\mu \nu^a\otimes \mu^{-1}\nu^{1-a}$, where $\mu:F^\times \ra \C^1$  is the unramified unitary character determined by $\mu(\varpi)=u$. 
We consider the representation $I(\mu\nu^a\otimes \mu^{-1}\nu^{1-a}) $. In the Grothendieck group, we have 
\begin{align*}I(\mu\nu^a\otimes \mu^{-1}\nu^{1-a}) 
&=I(\mu \nu^{a-1}\otimes \nu) \\
&=I_{\gamma_2}(a-1/2, I^{\GL_2}(\mu \nu^{1/2}\otimes \mu \nu^{-1/2}))\\
&=I_{\gamma_2}(a-1/2, \mu\otimes\St_{\GL_2})+I_{\gamma_2}(a-1/2, \mu\circ \det).
\end{align*}
By \cite{Muic}*{Theorem 3.1 (ii)}, the two representations, $I_{\gamma_2}(a-1/2, \mu\otimes\St_{\GL_2}) $ and $I_{\gamma_2}(a-1/2, \mu\circ \det) $, are irreducible.  
For these representations, the local Langlands correspondence is given by 
\[
\begin{array}{rcl}
\Pi_{\phi_{\ref{infcase:1-short}a}}(G_2(F)) &=& \{ I_{\gamma_2}(a-1/2, \mu\circ \det)\},\\
\Pi_{\phi_{\ref{infcase:1-short}b}}(G_2(F)) &=& \{ I_{\gamma_2}(a-1/2, \mu\otimes\St_{\GL_2})\}.
\end{array}
\]

It follows that the equivariant perverse sheaf corresponding to $I_{\gamma_2}(a-1/2, \mu\circ \det)$ is $\IC(\1_{C_{\ref{infcase:1-short}a}})$, where $\1_{C_{\ref{infcase:1-short}a}}$ is the constant local system on $C_{\ref{infcase:1-short}a}=\{0\}$ in $V=\mathbb{A}^1$, and the equivariant perverse sheaf corresponding to $I_{\gamma_2}(a-1/2, \mu\otimes\St_{\GL_2})$ is $\IC(\1_{C_{\ref{infcase:1-short}b}})$, where $\1_{C_{\ref{infcase:1-short}b}}$ is the constant local system on the open orbit $C_{\ref{infcase:1-short}b}$.
Taking the case $n=1$ of Section~\ref{geocase-1} and then using the base-change argument in the proof of Proposition~\ref{prop:PHV}, it follows that the simple objects in $\Perv_{H_\lambda}(V_\lambda)$ are $\IC(\1_{C_{\ref{infcase:1-short}a}})$ and $\IC(\1_{C_{\ref{infcase:1-short}b}})$ only, and that $\NEvs_{C_{\ref{infcase:1-short}a}} \IC(\1_{C_{\ref{infcase:1-short}a}}) = \1_{\Lambda_{C_{\ref{infcase:1-short}a}}}$ and $\NEvs_{C_{\ref{infcase:1-short}b}} \IC(\1_{C_{\ref{infcase:1-short}b}}) = \1_{\Lambda_{C_{\ref{infcase:1-short}b}}}$ while $\NEvs_{C_{\ref{infcase:1-short}a}} \IC(\1_{C_{\ref{infcase:1-short}b}}) =0$ and $\NEvs_{C_{\ref{infcase:1-short}b}} \IC(\1_{C_{\ref{infcase:1-short}a}}) = 0$. 
Therefore, there are no coronal representations for the $L$-packets above, and we have
\[
\begin{array}{rcl}
\Pi^\ABV_{\phi_{\ref{infcase:1-short}a}}(G_2(F)) &=& \Pi_{\phi_{\ref{infcase:1-short}a}}(G_2(F)),\\
\Pi^\ABV_{\phi_{\ref{infcase:1-short}b}}(G_2(F)) &=& \Pi_{\phi_{\ref{infcase:1-short}b}}(G_2(F)).
\end{array}
\]
These calculations follow from the results of Section~\ref{geocase-1}. 

\item[(Case \ref{infcase:2-long})]

Suppose $\phi$ is a Langlands parameter with infinitesimal parameter $\lambda_{\ref{infcase:2-long}}$ from Case~\ref{infcase:2-long} of Proposition~\ref{prop:infcases}.
This case is quite similar to the previous case.
The infinitesimal parameter $\lambda_{\ref{infcase:2-long}}$ can be represented by $\lambda_{\ref{infcase:2-long}}(\Frob)=\wh m(u q^a, u^{-1}q^{1-a})$ with $u\in \C^1, a\in \R$ satisfying $\wh \gamma(\lambda_{\ref{infcase:2-long}}(\Frob)=\wh m(u q^a, u^{-1}q^{1-a}) ) \ne q$ for all $\wh \gamma \in \wh R-\{2 \wh\gamma_1+3\wh \gamma_2\}$, which is equivalent to 
$$\{u^{-3}q^{2-3a}, u^2 q^{2a-1}, u^{-1}q^{1-a},uq^a,u^3q^{3a-1}\}\cap \{q,q^{-1}\}=\emptyset.$$
We have $V_{\lambda_{\ref{infcase:2-long}}}=\mathrm{Span}\{ X_{2{{\hat\gamma}_1}+3{{\hat\gamma}_2}} \}$.
In this case, $H_{\lambda_{\ref{infcase:2-long}}}=\dualgroup{T}$ if $u^2q^{2a-1}\ne 1$, and $H_{\lambda_{\ref{infcase:2-long}}}=\GL_{2,{{\hat\gamma}_2}}$ if $u^2q^{2a-1}=1$. The action of $H_{\lambda_{\ref{infcase:2-long}}}$ on $ V_{\lambda_{\ref{infcase:2-long}}}$ is given by 
\[
\begin{array}{rcll}
\wh m(x,y).r X_{2{{\hat\gamma}_1}+3{{\hat\gamma}_2}}&=&xy r X_{2{{\hat\gamma}_1}+3{{\hat\gamma}_2}}, & \textrm{ if } u^2q^{2a-1}\ne 1,\\
g.rX_{2{{\hat\gamma}_1}+3{{\hat\gamma}_2}}&=& \det(g)rX_{2{{\hat\gamma}_1}+3{{\hat\gamma}_2}}, &\textrm{ if } u^2q^{2a-1}= 1.
\end{array}
\]
This action has two orbits -- the zero orbit and the open orbit -- and each Langlands parameter has a trivial component group.
Representatives for the two $L$-parameters are
\[
\begin{array}{rcl}
\phi_{\ref{infcase:2-long}b}(w,x)&=& \lambda_{\ref{infcase:2-long}}(w) \\
\phi_{\ref{infcase:2-long}b}(w,x)&=& \wh m(u^{\ord(w)}|w|^{a-1/2}, u^{-\ord(w)}|w|^{-a+1/2})\cdot \iota_{2{{\hat\gamma}_1}+3{{\hat\gamma}_2}}(x).
\end{array}
\]

By Equation~\eqref{eqn:dualcharacter}, the dual character of $\lambda_{\ref{infcase:2-long}}$ is $ \mu^{-1}\nu^{1-a}\otimes \mu^2 \nu^{2a-1},$ where $\mu$ is the  unitary unramified character of $F^\times$ determined by $\mu(\varpi)=u$. In the Grothendieck group, we have 
\begin{align*}
I(\mu^{-1}\nu^{1-a}\otimes \mu^2 \nu^{2a-1})
&=I(\mu\nu^a\otimes \mu\nu^{a-1})\\
&=I_{\gamma_1}(a-1/2, I^{\GL_2}(\mu \nu^{1/2}\otimes \mu \nu^{-1/2}))\\
&=I_{\gamma_1}(a-1/2, \mu\otimes\St_{\GL_2})+I_{\gamma_1}(a-1/2,\mu\circ\det).
\end{align*}
By \cite{Muic}*{Theorem 3.1 (i)}, 
the representations $I_{\gamma_1}(a-1/2, \mu\otimes\St_{\GL_2})$ and $I_{\gamma_1}(a-1/2,\mu\cdot \det)$ are irreducible.
The Langlands correspondence is given in given in Table \ref{table:LA-packets}.  
Arguing as in Case~\ref{infcase:1-short}  or using our microlocal analysis of the prehomogeneous vector space \ref{geocase-1}, the ABV-packet coefficients are given in Table~\ref{table:ABV}.

\item[(Case \ref{infcase:3})]

Suppose $\phi$ is an unramified Langlands parameter with infinitesimal parameter $\lambda_{\ref{infcase:3}}$ from Case~\ref{infcase:3} of Proposition~\ref{prop:infcases}.
We can take $\lambda_{\ref{infcase:3}}(\Frob)=\wh m(\theta q^{1/3},\theta^2 q^{2/3})$ with $\theta=\theta_3^n$ for $n=0,1,2$ and $H_{\lambda_{\ref{infcase:3}}}=\GL_{2}^{{{\hat\gamma}_1}+3{{\hat\gamma}_2}}$. 
To see this note that ${{\hat\gamma}_1}(\lambda_{\ref{infcase:3}}(\Frob)) = q$ and $(2{{\hat\gamma}_1}+3{{\hat\gamma}_2})(\lambda(\Frob)) = q$ and ${\hat\gamma}(\lambda_{\ref{infcase:3}}(\Frob)) \ne q$ for every other ${\hat\gamma}\in {\hat R}$, so if $\lambda_{\ref{infcase:3}}(\Frob) = {\wh m}(u_1q^{a_1}, u_2 q^{a_2})$ then
$u_2=u_1^2 \in\{1,\theta_3,\theta_3^2\}$ and $a_1=\frac{1}{3}$, $a_2=\frac{2}{3}$. 
The action of $H_{\lambda_{\ref{infcase:3}}}$ on $V_{\lambda_{\ref{infcase:3}}}$ is $\det\otimes \Sym^1$, which is to say, $h\cdot x=\det(h)hx$ where $hx$ is matrix multiplication.
This is case \ref{geocase-2} for $n=1$.
In particular, there are two orbits of this action: the zero orbit and the open orbit. The corresponding Langlands parameters are given by 
\[
\begin{array}{rcl}
\phi_{\ref{infcase:3}a}(w,x) &=&\lambda_{\ref{infcase:3}}(w), \\
\phi_{\ref{infcase:3}b}(w,x) &=& \wh m(\theta^{\ord(w)}|w|^{1/3}, \theta^{2\ord(w)}|w|^{1/6})\cdot \iota_{{{\hat\gamma}_1}}(x).
\end{array}
\]

By Equation~\eqref{eqn:dualcharacter}, the dual character of $\lambda_{\ref{infcase:3}}$ is $\chi=\vartheta \nu^{2/3}\otimes \vartheta ^{2}\nu^{-1/3}$ where $\vartheta$ is the unramified character of $F^\times$ dual to $\theta$. In the Grothendieck group, we have
\begin{align*}
I(\theta_3^n \nu^{2/3}\otimes \theta_3^{2n}\nu^{-1/3} )
&=I_{\gamma_1}(1/6, I^{\GL_2}(\theta_3^n\nu^{1/2}\otimes \theta_3^{-n}\nu^{-1/2}))\\
&=I_{{\gamma_1}}(1/6,\theta_3^n\otimes \St_{\GL_2})+I_{\gamma_1}(1/6,\theta^n_3\circ \det).
\end{align*}
The Langlands correspondence is given in given in Table \ref{table:LA-packets}. 
Using our microlocal analysis of the prehomogeneous vector space \ref{geocase-2}, for $n=1$, the ABV-packet coefficients are given in Table~\ref{table:ABV}.
\item[(Case \ref{infcase:4-D2})]

Let $\lambda_{\ref{infcase:4-D2}}$ be the infinitesimal parameter appearing in 
Case \ref{infcase:4-D2} of Proposition~\ref{prop:infcases}.
Then $\lambda_{\ref{infcase:4-D2}}(\Frob)=\wh m(q,-q)$ and $V_{\lambda_{\ref{infcase:4-D2}}}=\mathrm{Span}\{X_{{{\hat\gamma}_1}}, X_{{{\hat\gamma}_1}+2{{\hat\gamma}_2}}\}$ and $H_{\lambda_{\ref{infcase:4-D2}}}=\dualgroup{T}$. The action of $\dualgroup{T}$ on $V_{\lambda_{\ref{infcase:4-D2}}}$ is given by 
 \[
\wh m (x,y).(r_1X_{{{\hat\gamma}_1}}+r_2X_{{{\hat\gamma}_1}+2{{\hat\gamma}_2}})=y^2x^{-1}r_1X_{{{\hat\gamma}_1}}+xr_2X_{{{\hat\gamma}_1}+2{{\hat\gamma}_2}}.
\] 
There are 4 orbits: $C_{\ref{infcase:4-D2}a}=\{0\}$, $C_{\ref{infcase:4-D2}b}=\{r_1X_{{{\hat\gamma}_1}}\tq r_1\ne 0\}$, $C_{\ref{infcase:4-D2}c}=\{r_2 X_{{{\hat\gamma}_1}+2{{\hat\gamma}_2}} \tq r_2\ne 0\}$ and  $C_{\ref{infcase:4-D2}d}=\{r_1X_{{{\hat\gamma}_1}}+r_2X_{{{\hat\gamma}_1}+2{{\hat\gamma}_2}}\tq r_1,r_2\ne 0\}$. 
We found $\Perv_{H_{\lambda_{\ref{infcase:4-D2}}}}(V_{\lambda_{\ref{infcase:4-D2}}})$ in Proposition~\ref{prop:PHV}.
 In particular, there are four $H_{\lambda_{\ref{infcase:4-D2}}}$-orbits in $V_{\lambda_{\ref{infcase:4-D2}}}$: $C_\text{\ref{infcase:4-D2}a}=\{0\}$, $C_\text{\ref{infcase:4-D2}b}=\{r_1X_{{{\hat\gamma}_1}}\tq r_1\ne 0\}$, $C_\text{\ref{infcase:4-D2}c}=\{r_2 X_{{{\hat\gamma}_1}+2{{\hat\gamma}_2}} \tq r_2\ne 0\}$ and  $C_\text{\ref{infcase:4-D2}d} = \o{V}=\{r_1X_{{{\hat\gamma}_1}}+r_2X_{{{\hat\gamma}_1}+2{{\hat\gamma}_2}}\tq r_1,r_2\ne 0\}$. 
 Corresponding Langlands parameters are given by
\[
\begin{array}{rcl}
\phi_\text{\ref{infcase:4-D2}a}(w,x)&=& \lambda(w),\\
\phi_\text{\ref{infcase:4-D2}b}(w,x)&=&\wh m(|w|,(-1)^{\ord(w)}|w|^{1/2})\cdot \iota_{{{\hat\gamma}_1}}(x),\\
\phi_\text{\ref{infcase:4-D2}c}(w,x)&=&\wh m(1,(-1)^{\ord(w)}|w|^{1/2})\cdot \iota_{{{\hat\gamma}_1}+2{{\hat\gamma}_2}}(x),\\
\phi_\text{\ref{infcase:4-D2}d}(w,x)&=&\wh m(1,(-1)^{\ord(w)})\iota_{{{\hat\gamma}_1}}(x)\iota_{{{\hat\gamma}_1}+2{{\hat\gamma}_2}}(x).
\end{array}
\]
Note that we have $A_{\phi_{\ref{infcase:4-D2}d}}=\langle {\widehat m}(1,-1)\rangle \iso \langle \theta_2\rangle$ and the component groups of the other 3 Langlands parameters are trivial.

The dual character of $\lambda_{\ref{infcase:4-D2}}$ is $\theta_2 \nu\otimes \theta_2$ by Equation~\eqref{eqn:dualcharacter}. 
Recall that $\theta_2$ is the unique unramified order 2 character of $F^\times$. 
By \cite{Muic}*{Proposition 4.1}, we have
\begin{align}\label{eqn:muicprop411}
\begin{split}
I(\theta_2 \nu\otimes \theta_2)&=I_{\gamma_1}(1/2,\theta_2\otimes \St_{\GL_2})+I_{\gamma_1}(1/2,\theta_2\circ \det)\\
&=I_{\gamma_2}(1/2,\theta_2\otimes \St_{\GL_2})+I_{{\gamma_2}}(1/2,\theta_2\circ \det),
\end{split}
\end{align}
and
\begin{equation}\label{eqn:muicprop412}
\begin{split}
\begin{array}{rcl}
I_{\gamma_1}(1/2,\theta_2\otimes \St_{\GL_2})&=&\pi(\theta_2)+J_{\gamma_1}(1/2,\theta_2\otimes \St_{\GL_2}),\\
I_{{\gamma_2}}(1/2,\theta_2\otimes \St_{\GL_2})&=&\pi(\theta_2)+J_{\gamma_2}(1/2,\theta_2\otimes \St_{\GL_2}),\\
I_{\gamma_1}(1/2,\theta_2\circ \det)&=&J_{\gamma_2}(1,I^{\GL_2}(1\otimes \theta_2))+J_{\gamma_2}(1/2,\theta_2\otimes \St_{\GL_2}),\\
I_{\gamma_2}(1/2,\theta_2\circ\det)&=&J_{\gamma_2}(1,I^{\GL_2}(1\otimes \theta_2))+J_{\gamma_1}(1/2,\theta_2\otimes \St_{\GL_2}),
\end{array}
\end{split}
\end{equation}
As in \cite{Muic}, we write $\pi(\theta_2)$ for the unique irreducible sub-representation of $I(\theta_2\nu\otimes\theta_2)$; note that $\pi(\theta_2)$ is square integrable by \cite{Muic}*{Proposition 4.1}.
 and all of the above equations are in the Grothendieck group of representations of $G_2(F)$. 

The proof of the Langlands correspondence $\Pi_\phi(G_2(F)) \to \widehat{A_\phi}$ for all unramified Langlands parameters for $G_2(F)$ with infinitesimal parameter $\lambda_{\ref{infcase:4-D2}}$ is similar to that of \cite{CFZ:cubics}*{Theorem 2.5}, which follows from some well-known facts. 
The Langlands correspondence is given in Table~\ref{table:LLC}; see also Table \ref{table:LA-packets}.
Here we only give some details of the proof of $\pi(\phi_{\text{\ref{infcase:4-D2}d}},\vartheta_2)=I_0(G_2[-1]) \ceq \mathrm{cInd}_{G_2(\CO_F)}^{G_2(F)}(G_2[-1])$. 
Using \cite{Lusztig:Intersectioncohomology}*{p.270}, the local system $(\phi_{\text{\ref{infcase:4-D2}d}},\vartheta_2)$ is cuspidal and thus $\pi(\phi_{\text{\ref{infcase:4-D2}d}},\vartheta_2)$ is a depth zero supercuspidal representation of the form
$
I_0(\sigma)\ceq \mathrm{cInd}_{G_2(\CO_F)}^{G_2(F)}(\sigma),
$
where $\sigma$ is a unipotent cuspidal representation of $G_2(\mathbb{F}_q)$. 
There are exactly four unipotent cuspidal representations of $G_2(\mathbb{F}_q)$; see \cite{Carter}*{p.460}.
 To determine which one the representation $\sigma$ is, we use the formal degree conjecture of \cite{HII}, which is known in this case by \cite{FOS}. 
  A simple calculation shows that the $L$-factor is given by $$L(s,\phi_{\text{\ref{infcase:4-D2}d}},\Ad)=\frac{1}{(1+q^{-2-s})(1+a^{-1-s}) (1-q^{-1-s})^2},$$
and the epsilon $\epsilon(s,\phi_{\text{\ref{infcase:4-D2}d}},\Ad)=q^{10(1/2-s)}.$ Thus $$\gamma(0,\phi_{\text{\ref{infcase:4-D2}d}},\Ad)=\frac{q^9}{(q+1)(q^3+1)}.$$ Here the $\epsilon$-factor and the $\gamma$-factor are computed with respect to a fixed additive character of $F$, which is omitted from the notation.
The formal degree conjecture asserts that
$$\frac{\deg(\sigma)}{\mu(G_2(\CO_F))}=\frac{\dim(\vartheta_2)}{|A_{\phi_{\text{\ref{infcase:4-D2}d}}}|}|\gamma(0,\phi_{\text{\ref{infcase:4-D2}d}},\Ad)|,$$
where $\mu(G_2(\CO_F))$ is the Haar measure of $G_2(\CO_F)$ normalized as in \cite{HII}, {\it i.e.}, 
$$\mu(G_2(\CO_F))=q^{-\dim(G_2(F))}\cdot |G_2(\mathbb{F}_q)|=q^{-8}(q^6-1)(q^2-1).$$
Thus we get $$\deg(\sigma)=\frac{q(q^6-1)(q^2-1)}{(q^3+1)(q+1)}.$$
Comparing it with the 4 unipotent cupsidal representations of $G_2(\mathbb{F}_q)$, we get $\sigma=G_2[-1]$.

Using the microlocal analysis of the prehomogeneous vector space \ref{geocase-toric} for $n=2$ from Section~\ref{ssec:geometry}, the ABV-packet coefficients are given by in Table~\ref{table:ABV}.

\item[(Case \ref{infcase:5})]

We now find the ABV-packets and coefficients for Langlands parameters with infinitesimal parameter $\lambda_{\ref{infcase:5}}$ from Case \ref{infcase:5} of Proposition~\ref{prop:infcases}. 
The infinitesimal parameter $\lambda_{\ref{infcase:5}}$ is determined by $\lambda_{\ref{infcase:5}}(\Frob)=\wh m(q^2,q)$. In this case, we have $V_{\lambda_{\ref{infcase:5}}}=\mathrm{Span}\{X_{\hat\gamma_2},X_{\hat\gamma_1+\hat\gamma_2}\}$ and $H_{\lambda_{\ref{infcase:5}}}=\GL_{2}^{\hat\gamma_1}$. The action of $H_{\lambda_{\ref{infcase:5}}}$ on $V_{\lambda_{\ref{infcase:5}}}$ is equivalent to matrix multiplication, depending on an isomorphism of $\GL_{2}^{\hat\gamma_1}$ with $\GL_2$. 
Category $\Perv_{H_{\lambda_{\ref{infcase:5}}}}(V_{\lambda_{\ref{infcase:5}}}) \iso \Perv_{\GL_2}(\mathbb{A}^2)$, with reference to the prehomogeneous vector space \ref{geocase-2} for $n=0$.
In particular, this action has two orbits: the zero orbit $C_{\ref{infcase:5}a}$ and the open orbit $C_{\ref{infcase:5}b}$. 
Corresponding Langlands parameters are given by 
\[
\begin{array}{rcl}
\phi_{\ref{infcase:5}a}(w,x) &=& \lambda_{\ref{infcase:5}}(w),\\
\phi_{\ref{infcase:5}b}(w,x) &=& \wh m(|w|^{3/2},|w|^{3/2})\cdot \iota_{\hat\gamma_2}(x).
\end{array}
\]

By Equation~\eqref{eqn:dualcharacter}, the dual character of $\lambda_{\ref{infcase:5}}$ is $\nu\otimes \nu$. In the Grothendieck group, we have
\begin{align*}
I(\nu\otimes \nu)&=I_{\gamma_2}(3/2,\pi(\nu^{1/2},\nu^{-1/2}))\\
&=I_{\gamma_2}(3/2, \St_{\GL_2})+I_{\gamma_2}(3/2,1_{\GL_2}).
\end{align*}
The Langlands correspondence is given in given in Table \ref{table:LA-packets}. Using our microlocal analysis of the prehomogeneous vector space \ref{geocase-2}, for $n=0$, the ABV-packet coefficients are given in Table~\ref{table:ABV}.

\item[(Case \ref{infcase:6-A2})]

The ABV-packets and coefficients for Langlands parameters with infinitesimal parameter $\lambda_{\ref{infcase:6-A2}}$  defined by $\lambda_{\ref{infcase:6-A2}}(\Frob)=\wh m(\theta_3 q, \theta_3^2 q)$ as in Case \ref{infcase:6-A2} of Proposition~\ref{prop:infcases} are calculated as follows.  
In this case, we have $V_{\lambda_{\ref{infcase:6-A2}}}=\mathrm{Span}\{X_{{{\hat\gamma}_1}},X_{{{\hat\gamma}_1}+3{{\hat\gamma}_2}}\}$ and $H=\dualgroup{T}$ and the action of $\dualgroup{T}$ on $V_{\lambda_{\ref{infcase:6-A2}}}$ is given by 
$$\wh m(x,y).(r_1X_{{{\hat\gamma}_1}}+r_2 X_{{{\hat\gamma}_1}+3{{\hat\gamma}_2}})=y^2 x^{-1}r_1X_{{{\hat\gamma}_1}}+x^2y^{-1}r_2 X_{{{\hat\gamma}_1}+3{{\hat\gamma}_2}}.$$
Using the base-change argument in the proof of Proposition~\ref{prop:PHV} we find that $\Perv_{H_{\lambda_{\ref{infcase:6-A2}}}}(V_{\lambda_{\ref{infcase:6-A2}}})$ is equivalent to the category of equivariant perverse sheaves on the prehomogeneous vector space \ref{geocase-toric} for $n=3$.
There are 4 orbits of this action: $C_{\ref{infcase:6-A2}a}=\{0\}, C_{\ref{infcase:6-A2}b}=\{r_1X_{{{\hat\gamma}_1}}\tq r_1\ne 0\}$, $C_{\ref{infcase:6-A2}c}=\{r_2X_{{{\hat\gamma}_1}+3{{\hat\gamma}_2}}, r_2\ne 0\}$ and $C_{\ref{infcase:6-A2}d}=\{r_1X_{{{\hat\gamma}_1}}+r_2 X_{{{\hat\gamma}_1}+3{{\hat\gamma}_2}},r_1,r_2\ne 0 \}$. 
Corresponding Langlands parameters are given by 
\[
\begin{array}{rcl}
\phi_{\ref{infcase:6-A2}a}(w,x)&=&\lambda(w),\\
\phi_{\ref{infcase:6-A2}b}(w,x)&=&\wh m(\theta_3^{\ord(w)}|w|, \theta_3^{2\ord(w)}|w|^{1/2})\cdot \iota_{{\hat\gamma}_1}(x),\\
\phi_{\ref{infcase:6-A2}c}(w,x)&=&\wh m(\theta_3^{\ord(w)}|w|^{1/2}, \theta_3^{2\ord(w)}|w|)\cdot \iota_{{{\hat\gamma}_1}+3{{\hat\gamma}_2}}(x),\\
\phi_{\ref{infcase:6-A2}d}(w,x)&=&\wh m(\theta_3^{\ord(w)}, \theta_3^{2\ord(w)})\varphi_{X_{{{\hat\gamma}_1}}+X_{{{\hat\gamma}_1}+3{{\hat\gamma}_2}}}(x),
\end{array}
\]
where $\varphi_{X_{{{\hat\gamma}_1}}+X_{{{\hat\gamma}_1}+3{{\hat\gamma}_2}}}(x):\SL_2(\C)\to \wh{G_2} $ is a group homomorphism determined by the $\mathfrak{sl}_2$-triple $(e,f,h)$ with $e=X_{{\hat\gamma}_1}+X_{{{\hat\gamma}_1}+3{{\hat\gamma}_2}}$, $f=2X_{-{{\hat\gamma}_1}}+2X_{-({{\hat\gamma}_1}+3{{\hat\gamma}_2})}$ and $h=2H_{{{\hat\gamma}_1}}+2H_{{{\hat\gamma}_1}+3{{\hat\gamma}_2}}$. 
We find that $A_{\phi_{\ref{infcase:6-A2}d}}$ has order $3$ and the component groups of the other 3 parameters are trivial.


The local Langlands correspondence in this case is obtained by arguments similar to the proof in case \ref{infcase:4-D2} and is given in Table \ref{table:LA-packets}.
Using the microlocal analysis of the prehomogeneous vector space \ref{geocase-toric} for $n=3$ from Section~\ref{ssec:geometry}, the ABV-packet coefficients are given in Table~\ref{table:ABV}.

\item[(Case \ref{infcase:7-reg})]

Recall Case~\ref{infcase:7-reg} from Proposition~\ref{prop:infcases}.
The infinitesimal parameter $\lambda_{\ref{infcase:7-reg}}$ is given by $\lambda_{\ref{infcase:7-reg}}(w)=\wh m(|w|^3, |w|^2)$. In this case, we have $V=\{r_1 X_{\hat\gamma_1}+r_2 X_{\hat\gamma_2}\tq r_1,r_2\in \C\}$ and $H_{\lambda_{\ref{infcase:7-reg}}}=\dualgroup{T}$. The action of $H_{\lambda_{\ref{infcase:7-reg}}}$ on $V_{\lambda_{\ref{infcase:7-reg}}}$ is given by 
$$\wh m(x,y).(r_1 X_{\hat\gamma_1}+r_2 X_{\hat\gamma_2})=y^2x^{-1}r_1X_{\hat\gamma_1}+xy^{-1}r_2X_{\hat\gamma_2}.$$
This action has 4 orbits $C_{\ref{infcase:7-reg}a}=\{0\}$, $C_{\ref{infcase:7-reg}b}=\{r_1X_{\hat\gamma_1}, r_1\ne 0\}$, $C_{\ref{infcase:7-reg}c}=\{r_2X_{\hat\gamma_2}, r_2\ne 0\}$ and $C_{\ref{infcase:7-reg}d}=\{r_1X_{\hat\gamma_1}+r_2X_{\hat\gamma_2}\tq r_1, r_2\ne 0\}.$ The 4 corresponding Langlands parameters are given by 
\[\begin{array}{rcl}
\phi_{\ref{infcase:7-reg}a}(w,x)&=&\lambda_{\ref{infcase:7-reg}}(w),\\
\phi_{\ref{infcase:7-reg}b}(w,x)&=&\wh m(|w|^3,|w|^{3/2})\cdot \iota_{\hat\gamma_1}(x),\\
\phi_{\ref{infcase:7-reg}c}(w,x)&=& \wh m(|w|^{5/2}, |w|^{5/2})\cdot \iota_{\hat\gamma_2}(x),\\
\phi_{\ref{infcase:7-reg}d}(w,x)&=&\varphi_{\mathrm{reg}}(x),
\end{array}
\]
where $\varphi_{\mathrm{reg}}: \SL_2(\C)\to \wh{G_2}$ is the group homomorphism determined by the regular orbit, {\it i.e.}, $\varphi_{\mathrm{reg}}$ is group homomorphism determined by an $\mathfrak{sl}_2$-triple $(e,f,h)$ with $e=X_{\hat\gamma_1}+X_{\hat\gamma_2}$. Each Langlands parameter has trivial component group.

The dual character of $\lambda_{\ref{infcase:7-reg}}$ is $\nu^2\otimes \nu$. 
All of the representations of with infinitesimal Langlands parameter $\lambda_{\ref{infcase:7-reg}}$ are components of $I(\nu^2\otimes \nu)$. 
From the decomposition of $I(\nu^2\otimes\nu)$ given in \cite{Muic}*{Proposition 4.4},
 we easily find the local Langlands correspondence in this case which is given in Table \ref{table:LA-packets}.
Using \ref{geocase-toric} with $n=1$ we find the ABV-coefficients which are given in Table~\ref{table:ABV}

\item[(Case \ref{infcase:8-sub})]

Case~\ref{infcase:8-sub} from Proposition~\ref{prop:infcases} was studied in \cite{CFZ:cubics}.
The Langlands correspondence is given in Table \ref{table:LA-packets} and the ABV-coefficient appear as the last block of Table~\ref{table:ABV}.
We remark that $A^\ABV_{\phi_{\ref{infcase:8-sub}a}} = A^\ABV_{\phi_{\ref{infcase:8-sub}d}} = S_3$ and $A^\ABV_{\phi_{\ref{infcase:8-sub}b}} = A^\ABV_{\phi_{\ref{infcase:8-sub}c}} = \langle \theta_2 \rangle$.
\end{enumerate}

\subsection{Fundamental properties of ABV-packet coefficients for $G_2$}

Let $x_\phi$ be the point for $\phi$ in the moduli space $V_{\lambda_\phi}$; let $C_\phi$ be the $H_{\lambda_\phi}$-orbit of $x_\phi$ in $V_{\lambda_\phi}$.
Pick $y\in V^*_{\lambda_\phi}$ such that $(x_\phi,y)\in T^*_{C_\phi}(V_{\lambda_\phi})_\text{sreg}$.
Then 
\[
A^\ABV_\phi =  \pi_0(T^*_{C_\phi}(V_{\lambda_\phi}),(x_\phi,y)).
\]
Let $p : T^*_{C_\phi}(V_{\lambda_\phi})\to C_\phi$ be projection.
Then $p$ induces the group homomorphism $\pi_1(p,(x_\phi,y)) : \pi_1(T^*_{C_\phi}(V_{\lambda_\phi}),(x_\phi,y)) \to \pi_1(C_\phi,x_\phi)$ which in turn induces a group homomorphism between the equivariant quotients, $A^\ABV_\phi \to A_\phi$.
Pre-composition with this group homomorphism defines the map
$
\Rep(A_\phi) \to \Rep(A^\ABV_\phi)
$
appearing in Theorem~\ref{thm:coefficients}.
\begin{table}[tp]
\caption{$L$-packets and ABV-packets for all unipotent representations of $G_2(F)$, arranged by the cases appearing Proposition~\ref{prop:infcases}. Each row gives an ABV-packet, formed as the union of an $L$-packet and its coronal representations.
Notation and proofs are presented in Section~\ref{ssec:LA-packets} drawing on results from Proposition~\ref{prop:PHV}.
}
\label{table:LA-packets}
\begin{center}
$
\begin{array}{| c | c | c |} 
\hline
\text{$L$-parameter}  & \text{$L$-packet}  &  \text{Coronal}\\ 
\phi &  \Pi_{\phi}(G_2(F)) &   \text{representations} \\ 
\hline\hline
\phi_{\ref{infcase:0}} 	&   I(\mu_1\nu^{a_1}\otimes \mu_2\nu^{a_2})  	&   \\
\hline\hline
\phi_{\ref{infcase:1-short}a}	&   I_{\gamma_2}(a-1/2, \mu\circ \det)  &    \\
\phi_{\ref{infcase:1-short}b}	&   I_{\gamma_2}(a-1/2, \mu\otimes\St_{\GL_2}) &    \\
\hline
\phi_{\ref{infcase:2-long}a}	&   I_{\gamma_1}(a-1/2, \mu\circ \det) &   \\
\phi_{\ref{infcase:2-long}b}	&   I_{\gamma_1}(a-1/2, \mu\otimes\St_{\GL_2})&   \\
\hline\hline
\phi_{\ref{infcase:3}a}  &   I_{{\gamma_1}}(1/6, \theta_3^n \circ \det)  	&   \\
\phi_{\ref{infcase:3}b} 	&   I_{{\gamma_1}}(1/6, \theta_3^n\otimes\St_{\GL_2})  &   \\
\hline
\phi_{\ref{infcase:4-D2}a}     &   J_{\gamma_2}(1,I^{\GL_2}(1\otimes\theta_2))   &   I_0(G_2[-1])  \\
\phi_{\ref{infcase:4-D2}b} 	&   J_{{\gamma_1}}(1/2,\theta_2\otimes \St_{\GL_2})    &   I_0(G_2[-1])  \\
\phi_{\ref{infcase:4-D2}c} 	&   J_{{\gamma_2}}(1/2,\theta_2\otimes \St_{\GL_2})     &   I_0(G_2[-1])   \\
\phi_{\ref{infcase:4-D2}d} 	&   \pi(\theta_2), I_0(G_2[-1])     &    \\
\hline
\phi_{\ref{infcase:5}a} 	&   I_{\gamma_2}(3/2,1_{\GL_2})  	&    \\
\phi_{\ref{infcase:5}b} 	&   I_{{\gamma_2}}(3/2,\St_{\GL_2})  &    \\
\hline
\phi_{\ref{infcase:6-A2}a}  &   J_{\gamma_2}(1,I^{\GL_2}(\theta_3\otimes\theta_3^{-1}))   &   I_0(G_2[\theta_3]), I_0( G_2[\theta_3^2])   \\
\phi_{\ref{infcase:6-A2}b} 	&   J_{{\gamma_1}}(1/2,\theta_3^{2}\otimes\St_{\GL_2})    &   I_0(G_2[\theta_3]), I_0( G_2[\theta_3^2])    \\
\phi_{\ref{infcase:6-A2}c} 	&   J_{{\gamma_1}}(1/2,\theta_3^{}\otimes\St_{\GL_2})    &   I_0(G_2[\theta_3]), I_0(G_2[\theta_3^2])    \\
\phi_{\ref{infcase:6-A2}d} 	&  \pi(\theta_3),  I_0(G_2[\theta_3]), I_0(G_2[\theta_3^{2}])   &   \\
\hline
\phi_{\ref{infcase:7-reg}a}  &   1_{G_2}   &  \\
\phi_{\ref{infcase:7-reg}b} 	&   J_{\gamma_1}(3/2,\St_{\GL_2})   &   \\
\phi_{\ref{infcase:7-reg}c} 	&   J_{{\gamma_2}}(5/2,\St_{\GL_2})   &  \\
\phi_{\ref{infcase:7-reg}d} 	&   \St_{G_2}   &  \\
\hline\hline
\phi_{\ref{infcase:8-sub}a}  &  J_{{\gamma_2}}(1,I^{\GL_2}(1\otimes 1))   &   J_{\gamma_1}(1/2,\St_{\GL_2}), I_0(G_2[1])   \\
\phi_{\ref{infcase:8-sub}b} &  J_{\gamma_1}(1/2,\St_{\GL_2})   &   J_{{\gamma_2}}(1/2,\St_{\GL_2}), I_0(G_2[1])  \\
\phi_{\ref{infcase:8-sub}c} 	&  J_{{\gamma_2}}(1/2,\St_{\GL_2})   &   \pi(1), I_0(G_2[1])   \\
\phi_{\ref{infcase:8-sub}d}	&  \pi(1)' , \pi(1), I_0(G_2[1])   &   \\
\hline
\end{array}
$
\end{center}
\end{table}%

\begin{theorem}\label{thm:coefficients}
Let $\phi : W'_F \to \Lgroup{G_2}$ be an unramified Langlands parameter.

\begin{enumerate}
\labitem{(LLC)}{G2:LLC}
The function $\pi \mapsto \langle \ , \pi\rangle $ extends the local Langlands correspondence:
for all unramified Langlands parameters $\phi : W'_F \to \Lgroup{G_2}$, the diagram
\begin{equation}\label{eqn:LLC}
\begin{tikzcd}
\Pi^\ABV_\phi(G_2(F)) \arrow{rr}{\pi \mapsto \langle \ , \pi \rangle } && \widehat{A^\ABV_\phi} \\
\Pi_\phi(G_2(F)) \arrow[>->]{u} \arrow{rr}{\pi \mapsto \langle \ , \pi \rangle_\phi}[swap]{\mathrm{LLC}} && \arrow[->]{u} \widehat{A_\phi}
\end{tikzcd}
\end{equation}
commutes, where $\widehat{A^\ABV_\phi}$ (resp. $\widehat{A_\phi}$) denotes the set of characters of irreducible representations of $A^\ABV_\phi$ (resp. $A_\phi$).
\labitem{(Open)}{G2:open} 
If $\phi: W'_F \to \Lgroup{G_2}$ is open or closed then $\pi \mapsto \langle \ , \pi\rangle $ is a bijection
$\Pi^\ABV_\phi(G_2(F)) \to \widehat{A^\ABV_\phi}$.
\labitem{(Temp)}{G2:tempered} 
If $\phi$ is bounded upon restriction to $W_F$ then all the representations in $\Pi^\ABV_\phi(G_2(F))$ are tempered.
If $\phi$ is not bounded upon restriction to $W_F$ then $\Pi^\ABV_\phi(G_2(F))$ may contain non-tempered representations and $\langle \ , \ \rangle $ may not be bijective.
\labitem{(Norm)}{G2:spherical} 
A representation $\pi  \in \Pi_{\phi}(G_2(F))$ is spherical if and only if $\phi$ is closed and $\langle \ , \pi \rangle $ is the trivial representation of $A^\ABV_\phi$.
In this case, the only ABV-packet that contains $\pi$ is $\Pi^\ABV_{\phi}(G_2(F))$.
A representation $\pi  \in \Pi_{\phi}(G_2(F))$ is generic if and only if $\phi$ is open and $\langle \ , \pi \rangle $ is the trivial representation of $A^\ABV_\phi$.
\labitem{(Dual)}{G2:Aubert} 
For every unramified Langlands parameter $\phi$ there is an unramified Langlands parameter ${\hat \phi}$, with the same infinitesimal parameter as $\phi$, such that the Aubert involution defines a bijection 
\[
\Pi^\ABV_{\phi}(G_2(F)) \to \Pi^\ABV_{\hat \phi}(G_2(F)).
\]
The $Z_{\dualgroup{G_2}}(\lambda_\phi)$-orbit of $\phi$ in $V_{\lambda_\phi}$ is dual to the transpose of the $Z_{\dualgroup{G_2}}(\lambda_\phi)$-orbit of ${\hat \phi}$ in $V_{\lambda_\phi}$.

\end{enumerate}
\end{theorem}

\begin{proof}
All these results follow from a study of Table~\ref{table:ABV} together with results from \cite{Muic}.
\begin{enumerate}
\item[\ref{G2:LLC}]
Table~\ref{table:ABV} shows that $r: \Rep(A_\phi) \to \Rep(A^\ABV_\phi)$ preserves irreducibility and is injective.
It follows from \cite[Theorem 7.10.1 (d)]{CFMMX} that if $\pi\in \Pi_\phi$ then
$
\langle \ , \pi\rangle= r(\mathop{LLC}(\pi))
$
where $\mathop{LLC}(\pi) \in \widehat{A_\phi}$ is the representation given by the local Langlands correspondence applied to $\pi$.
Alternately, this can be seen by comparing Tables~\ref{table:LLC} and \ref{table:ABV}.

\item[\ref{G2:open}]
In Cases~\ref{infcase:0}, \ref{infcase:1-short},  \ref{infcase:2-long}, \ref{infcase:3}, \ref{infcase:5} and \ref{infcase:7-reg}, $\Pi^\ABV_\phi(G_2(F)) \to \widehat{A^\ABV_\phi}$ is a bijection since $\Pi^\ABV_\phi(G_2(F)) = \Pi_\phi(G_2(F))$.
Observe also that every Langlands parameter is either open or closed in cases Cases~\ref{infcase:0}, \ref{infcase:1-short}, \ref{infcase:2-long}, \ref{infcase:3} and \ref{infcase:5}.
Inspect Table~\ref{table:ABV} for Cases~\ref{infcase:4-D2}, \ref{infcase:6-A2} and \ref{infcase:8-sub}.

\item[\ref{G2:tempered}]
Note that when $\phi|_{W_F}$ is bounded, then $\phi$ is open and $\Pi_\phi^{\ABV}(G_2(F))=\Pi_{\phi}(G_2(F))$. It follows from the previous explicit local Langlands correspondence that each $\Pi_\phi(G_2(F))$ consists of tempered representations when $\phi|_{W_F}$ is bounded. 

\item[\ref{G2:spherical}]
Recall the general fact that if $P=MN$ is a parabolic subgroup of a reductive group $G$ over $F$ and if $\sigma$ is an irreducible representation of $M$, then the parabolic induction $\Ind_P^G(\sigma)$ has a unique spherical component if $\sigma$ is spherical and has a unique generic component if $\sigma$ is generic. The assertion follows from this general fact. For example, in the Case~\ref{infcase:4-D2}, from Equation~\eqref{eqn:muicprop411} and Equation~\eqref{eqn:muicprop412}, we see that $\pi(\mu)$ is generic and $J_{\gamma_2}(1,\pi(1,\mu))$ is spherical. All of the other cases are similar.

\item[\ref{G2:Aubert}]
The Aubert involution is computed in \cite{Muic}. The involution $\phi \mapsto {\hat\phi}$ is given by: ${\hat\phi}_{\ref{infcase:0}} = \phi_{\ref{infcase:0}}$; ${\hat\phi}_{\ref{infcase:1-short}a}= \phi_{\ref{infcase:1-short}b}$; ${\hat\phi}_{\ref{infcase:2-long}a}= \phi_{\ref{infcase:2-long}b}$; ${\hat\phi}_{\ref{infcase:3}a}= \phi_{\ref{infcase:3}b}$; ${\hat\phi}_{\ref{infcase:4-D2}a}= \phi_{\ref{infcase:4-D2}d}$,  ${\hat\phi}_{\ref{infcase:4-D2}b}= \phi_{\ref{infcase:4-D2}c}$; ${\hat\phi}_{\ref{infcase:5}a}= \phi_{\ref{infcase:5}b}$; ${\hat\phi}_{\ref{infcase:6-A2}a}= \phi_{\ref{infcase:6-A2}d}$,  ${\hat\phi}_{\ref{infcase:6-A2}b}= \phi_{\ref{infcase:6-A2}c}$; ${\hat\phi}_{\ref{infcase:7-reg}a}= \phi_{\ref{infcase:7-reg}d}$,  ${\hat\phi}_{\ref{infcase:7-reg}b}= \phi_{\ref{infcase:7-reg}c}$ ; ${\hat\phi}_{\ref{infcase:8-sub}a}= \phi_{\ref{infcase:8-sub}d}$,  ${\hat\phi}_{\ref{infcase:8-sub}b}= \phi_{\ref{infcase:8-sub}c}$. 
The result now follows from Table~\ref{table:ABV}.\qedhere
\end{enumerate}
\end{proof}

\begin{remark}\label{rem:bijective}
With reference to \ref{G2:open}, all unramified Langlands parameters of $G_2(F)$ are either open or closed with the exception of $\phi_{\ref{infcase:4-D2}b}$, $\phi_{\ref{infcase:4-D2}c}$, $\phi_{\ref{infcase:6-A2}b}$, $\phi_{\ref{infcase:6-A2}c}$, $\phi_{\ref{infcase:8-sub}b}$ and $\phi_{\ref{infcase:8-sub}c}$. For all but the last two, $\pi \mapsto \langle \ , \pi\rangle $ defines a bijection $\Pi^\ABV_\phi(G_2(F)) \to \widehat{A^\ABV_\phi}$ and in these two cases, the map is surjective.
\end{remark}

\subsection{A geometric interpretation of Muic's reducibility points}\label{ssec:reducibility}

Many results of \cite{Muic} on reducibility points of induced representations of $G_2$ can be explained by the geometry of the action of $H_{\lambda}$ on $V_{\lambda}$. The following corollary is a collection of previous discussions, which explains the reducibility results of \cite{Muic}*{Proposition 3.1 and Theorem 3.1} using geometry.
\begin{proposition}\label{prop:reducibility} 
Let $\lambda : W_F \to \dualgroup{T}\to \dualgroup{G_2}$ be an unramified infinitesimal parameter with Vogan variety $V_\lambda$ and the group $H_{\lambda}$ acting on it. Let $\chi=\chi_1\otimes \chi_2:T(F)\ra \C^\times$ be the character dual to $\lambda$, which is determined by Equation~\eqref{eqn:dualcharacter}. 
\begin{enumerate}
\labitem{(i)}{(i)} The representation $I(\chi_1\otimes \chi_2)$ is irreducible if and only if $V_{\lambda}=\{0\};$
\labitem{(ii)}{(ii)} Suppose that $V_{\lambda}\ne \{0\}$. Then we can write $I(\chi_1\otimes \chi_2)=I_\gamma(s,\mu\otimes \St_{\GL_2})+I_{\gamma}(s,\mu\circ \det)$ for $\gamma\in \{{\gamma_1},{\gamma_2}\}$, $s\in \C$ and a unitary character $\mu$. Moreover, the two representations, $I_\gamma(s,\mu\otimes\St_{\GL_2})$ and $I_\gamma(s,\mu\circ \det)$, are irreducible if and only if the action of $H_{\lambda}$ on $V_{\lambda}$ has exactly two orbits.
\end{enumerate}
\end{proposition}

\begin{proof}
Item~\ref{(i)} is a corollary of the discussion in Case~\ref{infcase:0}, Section~\ref{ssec:LA-packets}. 
If $V_{\lambda}\ne 0$, from the above case by case study, we see that we can write $I(\chi_1\otimes \chi_2)$ as a sum, $I_\gamma(s,\mu\otimes\St_{\GL_2})+I_\gamma(s,\mu\circ \det)$.  Note that in the cases in \ref{infcase:1-short}, \ref{infcase:2-long}, \ref{infcase:5} and \ref{infcase:3}, the action of $H_{\lambda}$ on $V_{\lambda}$ has exactly two orbits, and the corresponding $I(s,\mu\otimes\St_{\GL_2})$ and $I(s,\mu\circ \det)$ are irreducible. While in the cases in subsections \ref{infcase:4-D2}, \ref{infcase:6-A2}, \ref{infcase:7-reg} and \ref{infcase:8-sub}, the action of $H_{\lambda}$ on $V_{\lambda}$ has more than two orbits and the corresponding $I(s,\mu\otimes\St_{\GL_2})$ and $I(s,\mu\circ \det)$ are reducible.
\end{proof}

This corollary shows that how the geometry of the action of $H_\lambda$ on $V_{\lambda}$ can determine the representations, which is exactly the spirit of \cite{CFMMX}.


\subsection{Unramified Arthur parameters for $G_2$}\label{ssec:Arthur parameters}

\begin{definition}\label{def:Arthur type}
Let us say that a Langlands parameter $\phi:W_F\times \SL_2(\C)\to \Lgroup{G_2}$ {\it is of Arthur type $\psi$} if there exists an Arthur parameter $\psi:W_F\times \SL_2(\C)\times \SL_2(\C)\to \Lgroup{G_2}$,  such that $\phi=\phi_\psi$, where $\phi_\psi(w,x)=\psi(w,x,\diag(|w|^{1/2},|w|^{-1/2}))$.
We say that an irreducible admissible representation $\pi$ of $G_2(F)$ is of Arthur type if there is a Langlands parameter $\phi$ of Arthur type such that $\pi \in \Pi^\ABV_\phi(G_2(F))$; if such a $\phi$ exists, it is generally not unique.
\end{definition}

\begin{remark}
If $\phi$ is of Arthur type then the Arthur parameter $\psi$, above, is unique.
If $\pi$ is of Arthur type then the Langlands parameter $\phi$, above, is generally not unique.
\end{remark}

\begin{remark}\label{rem:Arthur type}
Definition~\ref{def:Arthur type} opens the door to a dangerous misapprehension. Every Arthur parameter $\psi$ determines a Langlands parameter $\phi_\psi$; this function from Arthur parameters to Langlands parameters is injective. We say that a Langlands parameter is of Arthur type if it lies in the image of this function. On the other hand, an irreducible admissible representation is said to be of Arthur type if it lies in an Arthur packet - typically more than one. It can happen that a representation of Arthur type may have a Langlands parameters that is not of Arthur type. An example of this phenomenon appears in \cite{CFMMX}*{Section 16}. Specifically, the Arthur packet $\Pi_{\psi_2}(\SO_7(F))$ contains two representations, $\pi_2^+$ and $\pi_3^-$; the Langlands parameter for $\pi_2^+$ is of Arthur type $\psi_2$, but the Langlands parameter for $\pi_3^-$ is not of Arthur type; see \cite{CFMMX}*{Table 16.5.1}. Thus $\pi_3^-$ is an Arthur type representations whose Langlands parameter is not of Arthur type. 
On the other hand, this phenomenon does not happen for any unipotent representations of $G_2(F)$ - see Corollary~\ref{cor:Arthur type}.
\end{remark}

We now determine which unramified Langlands parameters for $G_2(F)$ are of Arthur type.

\begin{proposition}\label{prop:Arthur parameters}
Table~\ref{table:AE} presents necessary and sufficient conditions for an unramified Langlands parameter $\phi$ of $G_2(F)$ to be of Arthur type.
\end{proposition}

\begin{proof}
Note that $\psi|_{W_F\times \{1\}\times \{1\}}$ is bounded by definition of Arthur parameter. Moreover, we can write $\psi(w,x,y)=\psi(w,1,1)\psi(1,x,1)\psi(1,1,y)$. The image of $\psi(1,x,1)$ and $\psi(1,1,y)$ is either 1 or a copy of $\SL_2(\C)$ in $\Lgroup{G_2}$. Observe that the images of $\psi(w,1,1)$, $ \psi(1,x,1)$  and $\psi(1,1,y)$ are commutative and both $d\psi\left(1,\left(\begin{smallmatrix} 0& 1\\ 0 &0 \end{smallmatrix}\right),1\right)$ and $d\psi\left(1,1,\left(\begin{smallmatrix} 0& 1\\ 0 &0 \end{smallmatrix}\right)\right)$ are in $V_{\lambda}$ with $\lambda=\lambda_{\psi}$, where $d\psi$ is the differential of $\psi$.

\begin{enumerate}
\item[(Case \ref{infcase:0})] 
In this case, we have $ \phi_{\ref{infcase:0}}(w,x)=\wh m(u_1^{\ord(w)},u_2^{\ord(w)})\cdot \wh m(|w|^{a_1},|w|^{a_2})$ with $u_1,u_2\in \C^1, a_1,a_2\in \R$ and the Vogan variety is $\wpair{0}$. From the above observation, $\psi(1,x,y)=1$. Thus in this case, $\phi_{\ref{infcase:0}}$ is of Arthur type if and only if $\psi_{\ref{infcase:0}}(w,x,y)=\phi_{\ref{infcase:0}}(w,x)$ is an Arthur parameter, which is equivalent to $a_1=a_2=0$.

\item[(Case \ref{infcase:1-short})]
In this case, $V_{\lambda_{\ref{infcase:1-short}}}$ has dimension one. Let $\psi$ be an Arthur parameter with $\lambda_\psi=\lambda_{\ref{infcase:1-short}}$. From the above observation, we get at least one of $\psi(1,x,1)$ and $\psi(1,1,y)$ is trivial. Suppose that $\psi_{\ref{infcase:1-short}b}$ is an Arthur parameter with $\phi_{\ref{infcase:1-short}b}=\phi_{\psi_{\ref{infcase:1-short}b}}$. Then $\psi_{\ref{infcase:1-short}b}(1,x,1)=\iota_{\hat \gamma_1+2\hat \gamma_2}(x)$ and thus $\psi_{\ref{infcase:1-short}b}(1,1,y)=1$. Now the condition $\phi_{\ref{infcase:1-short}b}=\phi_{\psi_{\ref{infcase:1-short}b}} $ implies that $\psi(w,1,1)=\wh m(1,u^{\ord(w)}|w|^{a-1/2})$ with $u\in \C^1,a\in \R$, which is supposed to be bounded. Thus we get $a=1/2$. Note that, when $a=1/2$, the corresponding Arthur parameter $\psi_{\ref{infcase:1-short}b}(w,x,y)=\phi_{\ref{infcase:1-short}}(w,x)$. It is easy to see that $\phi_{\ref{infcase:1-short}a}$ is of Arthur type if and only if  $a=1/2$ and in this case the corresponding Arthur parameter is given by $\psi_{\ref{infcase:1-short}a}(w,x,y)=\psi_{\ref{infcase:1-short}b}(w,y,x)$.

\item[(Case \ref{infcase:2-long})] 
This case is similar to Case \ref{infcase:1-short} and thus omitted.

\item[(Case \ref{infcase:3})] 
Recall that in this case, $V_\lambda$ is spanned by $X_{\wh \gamma_1}$ and $X_{2\wh \gamma_1+3\wh \gamma_2}$. 
Notice that these two roots are not orthogonal to each other. Suppose that $\phi_{\ref{infcase:3}b}$ is of Arthur type with Arthur parameter $\psi_{\ref{infcase:3}b}$. 
We have $\psi_{\ref{infcase:3}b}(1,x,1)=\phi_{\ref{infcase:3}b}(1,x)=\iota_{\wh \gamma_1}(x)$. 
Since $\psi_{\ref{infcase:3}b}(1,1,y)$ is commutative with $\iota_{\wh \gamma_1}$ and $(d\psi_{\ref{infcase:3}b}) \left(1,1,\left(\begin{smallmatrix} 1&1\\ 0 &1 \end{smallmatrix}\right)\right)\in V_{\lambda}$, we must have $\psi_{\ref{infcase:3}b}(1,1,y)=1$. 
This implies that $\psi_{\ref{infcase:3}b}(w,1,1)=\phi_{\ref{infcase:3}b}(w,1)$ which is unbounded. 
Thus there is no such Arthur parameter and $\phi_{\ref{infcase:3}b}$ is not of Arthur type. Similarly, $\phi_{\ref{infcase:3}a}$ is also not of Arthur type.

\item[(Case \ref{infcase:4-D2})] 
Every Langlands parameter in this case is of Arthur type and the corresponding Arthur parameter can be given explicitly,
\begin{align*}
\psi_{\ref{infcase:4-D2}a}(w,x,y)&=\wh m(1,(-1)^{\ord(w)})\iota_{\wh \gamma_1}(y)\iota_{\wh \gamma_1+2\wh \gamma_2}(y),\\
\psi_{\ref{infcase:4-D2}b}(w,x,y)&=\wh m(1,(-1)^{\ord(w)})\iota_{\wh \gamma_1}(x)\iota_{\wh \gamma_1+2\wh \gamma_2}(y),\\
\psi_{\ref{infcase:4-D2}c}(w,x,y)&=\wh m(1,(-1)^{\ord(w)})\iota_{\wh \gamma_1}(y)\iota_{\wh \gamma_1+2\wh \gamma_2}(x),\\
\psi_{\ref{infcase:4-D2}d}(w,x,y)&=\wh m(1,(-1)^{\ord(w)})\iota_{\wh \gamma_1}(x)\iota_{\wh \gamma_1+2\wh \gamma_2}(x).
\end{align*}

\item[(Case \ref{infcase:5})] 
In this case, both parameters are not of Arthur type and the reason is the same as in the case \ref{infcase:3}.

\item[(Case \ref{infcase:6-A2})]
The Langlands parameters $\phi_{\ref{infcase:6-A2}a}$ and $\phi_{\ref{infcase:6-A2}d}$ are of Arthur type and the corresponding Arthur parameters are given by 
$$\psi_{\ref{infcase:6-A2}d}(w,x,y)=\phi_{\ref{infcase:6-A2}d}(w,x), \quad\psi_{\ref{infcase:6-A2}a}(w,x,y)=\psi_{\ref{infcase:6-A2}d}(w,y,x).$$
The parameters $\phi_{\ref{infcase:6-A2}b}$ and $\phi_{\ref{infcase:6-A2}c}$ are not of Arthur type for the following reason. Suppose that, for example, $\phi_{\ref{infcase:6-A2}c}$ is of Arthur type and the corresponding Arthur parameter is $\psi_{\ref{infcase:6-A2}c}$. Similar reason as above, $\psi_{\ref{infcase:6-A2}c}(1,1,y)$ should be commutative with $\iota_{\wh \gamma_1+3\wh \gamma_2}(x)$ and also $(d\psi_{\ref{infcase:6-A2}c})\left(1,1,\left(\begin{smallmatrix} 1&1\\ 0&1 \end{smallmatrix}\right)\right)$ should be in $V_{\lambda_{\ref{infcase:6-A2}}}$. One can check that such $\psi_{\ref{infcase:6-A2}c}(1,1,y)$ does not exist. 

\item[(Case \ref{infcase:7-reg})] 
The situation is similar to the above case. The parameters $\phi_{\ref{infcase:7-reg}a}$ and $\phi_{\ref{infcase:7-reg}d}$ are of Arthur type and the corresponding Arthur parameters are given by 
$$\psi_{\ref{infcase:7-reg}d}(w,x,y)=\phi_{\ref{infcase:7-reg}d}(w,x), \quad \psi_{\ref{infcase:7-reg}a}(w,x,y)=\psi_{\ref{infcase:7-reg}d}(w,y,x).$$ 
The parameters $\phi_{\ref{infcase:7-reg}b}$ and $\phi_{\ref{infcase:7-reg}c}$ are not of Arthur type for a similar reason as in the last case.

\item[(Case \ref{infcase:8-sub})] 
All of the 4 Langlands parameters are of Arthur type as explained in \cite{CFZ:cubics}.\qedhere
\end{enumerate}

\end{proof}

\begin{remark}
For a unipotent conjugacy class of $\dualgroup{G_2}$, there is a corresponding conjugacy class of group homomorphisms, $\varphi:\SL_2(\C)\to \dualgroup{G_2}$, which determine Arthur parameters $\psi$ by the formula $\psi(w,x,y)=\varphi(y)$. There are a total of 5 unipotent conjugacy classes of $\dualgroup{G_2}$, see \cite{GG}*{2.7}, for example. The Arthur parameters of these 5 unipotent conjugacy classes are:  $\psi_{\ref{infcase:0}}$ when $ u_1=u_2=1$ and $a_1=a_2=0$ (corresponding to the trivial conjugacy class), $\psi_{\ref{infcase:1-short}a} $ when $u=1$ (corresponding to the unipotent conjugacy class of the short root, {\it i.e.}, $\mathcal{O}_{short}$ in the notation of \cite{GG}*{2.7}), $\psi_{\ref{infcase:2-long}a}$ when $u=1$ (corresponding to $\mathcal{O}_{long}$), $\psi_{\ref{infcase:7-reg}a}$ (corresponding to $\mathcal{O}_{reg}$), and $\psi_{\ref{infcase:8-sub}a}$ (corresponding to $\mathcal{O}_{subreg}$).
\end{remark}


\begin{corollary}\label{cor:Arthur type}
A unipotent representation $\pi$ of $G_2(F)$ is of Arthur type if and only if its Langlands parameter is of Arthur type.
\end{corollary}

\begin{proof}
We have found all unramified Arthur parameters in Proposition~\ref{prop:Arthur parameters}. Using Table~\ref{table:ABV}, we find all unipotent representations of Arthur type. Using Table~\ref{table:LLC} we find the Langlands parameters of these representations. Using Proposition~\ref{prop:Arthur parameters} again, we find that these Langlands parameters are all of Arthur type.
\end{proof}

Arthur has conjectured \cite{Arthur:Conjectures} that an irreducible admissible representation is unitary if and only if it is of Arthur type. In Corollary~\ref{cor:unitary} we confirm this conjecture for unipotent representations of $G_2(F)$. 

\begin{corollary}\label{cor:unitary}
A unipotent representation of $G_2(F)$ is unitary if and only if is of Arthur type.
\end{corollary}

\begin{proof}
We use \cite{Muic} to find which unipotent representations are unitary. 
We use Proposition~\ref{prop:Arthur parameters} and Table~\ref{table:ABV} to find which unipotent representations are or Arthur type.
Direct inspection reveals that these two sets are the same.
\end{proof}


Let $\phi : W'_F\to \Lgroup{G_2}$ be a Langlands parameter for $G_2$ over $F$. We set $\dim(\phi) \ceq \dim C_\phi$. 
For $\pi\in \Pi(G_2(F))$, set $\dim(\pi) \ceq \dim(\phi_\pi)$ where $\phi_\pi$ is the Langlands parameter for $\pi$.

\begin{proposition}\label{prop:Arthur type}
If $\phi$ is an unramified Langlands parameter of Arthur type $\psi$ then $\mathcal{S}^\ABV_\phi = \pi_0(Z_{\dualgroup{G}}(\psi))$ and $s_\phi \ceq s_\psi \ceq \psi(1,1,-1) \in \mathcal{S}^\ABV_\phi$ has the property
\[
\langle s_\phi, \pi\rangle 
=
(-1)^{\dim(\phi)-\dim(\pi) } \langle 1, \pi\rangle,
\]
for every $\pi \in \Pi^\ABV_{\phi}(G_2(F))$.
\end{proposition}

\begin{proof}
According to Definition~\ref{def:SABV}, the value of $\langle s, \pi\rangle $, for any $s\in \mathcal{S}^\ABV_\phi$, depends only on the image of $s$ under the map $\mathcal{S}^\ABV\to A^\ABV_\phi$, so we identify $s$ with its image in the following argument.
The proof is based on a case by case consideration. If $A_{\phi}^{\ABV}=1$, this follows from that $\langle~,\pi\rangle \ne 0$ if and only if  $\Pi_\phi=\{\pi\}$. The case \ref{infcase:8-sub} is considered in \cite{CFZ:cubics}. Thus it suffices to consider the case (\ref{infcase:4-D2}),  (\ref{infcase:6-A2}a) and (\ref{infcase:6-A2}d). We only consider the case (\ref{infcase:6-A2}a) and omit the details for the other cases. Notice that $s_{\phi_{\ref{infcase:6-A2}a}}^2 =1$ while $A_{\phi_{\ref{infcase:6-A2}a}}^{\ABV}$ has order 3 and thus $\langle s_{\ref{infcase:6-A2}a},\pi\rangle_{\phi_{\ref{infcase:6-A2}a}}=\langle 1,\pi \rangle_{\phi_{\ref{infcase:6-A2}a}}$ for all $\pi\in \Pi^{\ABV}_{\phi_{\ref{infcase:6-A2}a}}$. Thus it suffices to check that $(-1)^{\dim(\phi_{\ref{infcase:6-A2}a})-\dim(\pi)}=1$ for all $\pi\in \Pi_{\phi_{\ref{infcase:6-A2}a}}^{\ABV}(G_2(F))$. Note that $\dim({\phi_{\ref{infcase:6-A2}a}})=0$ and $\dim(\pi)=0$ or $2$ by the description of $\Pi_{\phi_{\ref{infcase:6-A2}a}}^{\ABV}.$ The result follows.
\end{proof}

\begin{remark}\label{remark:nonexistence of s} Clearly, Proposition \ref{prop:Arthur type} is true for many Langlands parameters that are not of Arthur type. 
For example, when $A_{\phi}^{\ABV}=1$, Proposition \ref{prop:Arthur type} is true for $s=1$, regardless of whether $\phi$ is of Arthur type or not.
However, there indeed exist $\phi$ for which $(-1)^{\dim(\phi)-\dim(\pi)} \langle 1, \pi\rangle $ does not take the form $\langle s, \pi\rangle $ for any $s\in \mathcal{S}_\phi^\ABV$. 
Here is one example. 
Let $\phi=\phi_\text{\ref{infcase:6-A2}b}$ and $\pi=\mathrm{cInd}_{G_2(\CO_F)}^{G_2(F)}(G_2[\theta_3])$. 
Then $\langle~,\pi\rangle =\vartheta_3$ according Table~\ref{table:ABV}. 
Note that $\dim(\phi)=1$ and $\dim(\pi)=2$ and thus $(-1)^{\dim(\phi)-\dim(\pi)} \langle 1, \pi\rangle =-1$. 
On the other hand, we have $\langle s, \pi\rangle =\vartheta_3(s)\in \left\{1, \theta_3, \theta_3^2 \right\}.$ 
Thus $(-1)^{\dim(\phi)-\dim(\pi)} \langle 1, \pi\rangle $ does not take the form $\langle s, \pi\rangle $ for any $s\in \mathcal{S}_\phi^\ABV$, in this case. 
The same is true for $\phi=\phi_\text{\ref{infcase:6-A2}c}$ by a similar argument. 
\end{remark}

\subsection{Distributions attached to ABV-packets}\label{ssec:stable}

\begin{definition}\label{def:Thetaphis}
For any Langlands parameter $\phi : W_F'\to \Lgroup{G_2}$ and any $s\in \mathcal{S}^\ABV_\phi$ we define
\begin{equation}\label{eqn:distributions}
\Theta^{}_{\phi,s} \ceq  \sum_{\pi \in\Pi^\ABV_{\phi_\psi}(G_2(F))} \trace_s \left( \NEvs_{C_\phi}[\dim(\phi)] \mathcal{P}(\pi)[-\dim(\pi)] \right) \Theta_{\pi}.
\end{equation}
Using Definition~\ref{def:SABV} this takes the equivalent form
\begin{equation}
\Theta^{}_{\phi,s} =   \sum_{\pi \in\Pi^\ABV_{\phi}(G_2(F))}  (-1)^{\dim(\phi)-\dim(\pi)}\ \langle s,\pi\rangle\ \Theta_{\pi}.
\end{equation}
We will also use the notation $\Theta^{G_2}_{\phi} \ceq \Theta_{\phi,1}^{G_2}$.
\end{definition}


\begin{proposition}\label{prop:Arthurdistribution}
If $\phi$ is of Arthur type $\psi$ then
\begin{equation}\label{eq:Arthurdistribution}
\Theta^{}_{\phi,s} =\mathop{\sum}\limits_{\pi \in\Pi^\ABV_{\phi}(G_2(F))}\langle s_\psi\, s,\pi\rangle\ \Theta_{\pi}.
\end{equation}
\end{proposition}

\begin{proof}
Using Definitions~\ref{def:SABV}, \ref{def:Thetaphis} and Proposition~\ref{prop:Arthur type}, we have
\[
\begin{array}[b]{rcl}
\Theta^{}_{\phi,s} 
&=&  
\mathop{\sum}\limits_{\pi \in\Pi^\ABV_{\phi}(G_2(F))} \trace_s \left( \NEvs_{C_\phi}[\dim(\phi)] \mathcal{P}(\pi)[-\dim(\pi)] \right)\ \Theta_{\pi}\\
&=&
\mathop{\sum}\limits_{\pi \in\Pi^\ABV_{\phi}(G_2(F))} \trace_s \left( \NEvs_{C_\phi}[\dim(\phi)] \mathcal{P}(\pi)[-\dim(\pi)] \right)\ \Theta_{\pi}\\
&=&
 \mathop{\sum}\limits_{\pi \in\Pi^\ABV_{\phi}(G_2(F))} (-1)^{\dim(\phi)-\dim(\pi)} \trace_s\left( \NEvs_{C_\phi}\mathcal{P}(\pi) \right)\ \Theta_{\pi}\\
&=&
 \mathop{\sum}\limits_{\pi \in\Pi^\ABV_{\phi}(G_2(F))} (-1)^{\dim(\phi)-\dim(\pi)} \langle s, \pi\rangle\ \Theta_{\pi}\\
&=&
 \mathop{\sum}\limits_{\pi \in\Pi^\ABV_{\phi}(G_2(F))}  \langle s_\psi\, s, \pi\rangle\ \Theta_{\pi} .
\end{array}\qedhere
\]
\end{proof}

The right hand side of Equation~\eqref{eq:Arthurdistribution} recalls the distributions used by Arthur for classical groups in \cite{Arthur:book} and parameters of Arthur type. As we noticed in Remark~\ref{remark:nonexistence of s}, there are Langlands parameters $\phi$ for which the conclusion of Proposition~\ref{prop:Arthur type} does not hold. Thus Definition~\ref{def:Thetaphis} provides a generalization of the distributions used by Arthur. In the the next result we will try to justify that it is indeed the correct generalization. 

In order to state the next result, let us say that an Langlands parameter for $G_2(F)$ is elliptic if it does not factor through the $L$-group of Levi subgroup of $G_2(F)$. We will discuss endoscopy for $G_2(F)$ in Sections~\ref{sec:endoscopy} and \ref{sec:geoendo}. We simply enumerate the elliptic unramified Langlands parameters for $G_2(F)$ in Table~\ref{table:elliptic} . 
 
\begin{table}[htp]
\caption{Elliptic unramified Langlands parameters for $G_2(F)$ and their $L$-packets. The representations appearing in these $L$-packets form the complete list of elliptic unipotent representations of $G_2(F)$; see \cite{CO}, where they are denoted by $v_1 =\St_{G_2} $, $v_2=\pi(1)' $, $v_3 =\pi(1) $, $v_4=\pi(\theta_3) $, $v_5 =\pi(\theta_2) $, $v_6 =I_0(G_2[1])$, $v_7=I_0(G_2[-1])$, $v_8 = I_0(G_2[\theta_3])$ and $v_9 = I_0(G_2[\theta^2_3])$.} 
\label{table:elliptic} 
\begin{center}
$
\begin{array}{| c | l| c | c |}
\hline
\text{$L$-parameter} & \text{$L$-packet = ABV-packet} & \text{Endoscopic} & \text{Component groups} \\
\phi & \Pi_\phi(G_2(F)) = \Pi^\ABV_\phi(G_2(F)) & \text{groups} & A_\phi = A^\ABV_\phi \\
\hline\hline
\phi_{\ref{infcase:4-D2}d} & \pi(\theta_2), I_0(G_2[-1]) & \SO_4 & \langle \theta_2\rangle  \\
\phi_{\ref{infcase:6-A2}d} & \pi(\theta_3),  I_0(G_2[\theta_3]), I_0(G_2[\theta^2_3])  &  \PGL_3  &  \langle \theta_3\rangle \\
\phi_{\ref{infcase:7-reg}d} & \St_{G_2}  & G_2 & 1 \\
\phi_{\ref{infcase:8-sub}d} &  \pi(1)' , \pi(1), I_0(G_2[1])  &   \SO_4, \PGL_3 & S_3 \\
\hline
\end{array}
$
\end{center}
\end{table}%

\begin{theorem}\label{thm:stable}
Let $\phi$ be an unramified Langlands parameter for $G_2(F)$; recall $\mathcal{S}^\ABV_\phi$ from Definition~\ref{def:SABV} and $\Theta^{}_{\phi,s}$ from Definition~\ref{def:Thetaphis}, for $s\in \mathcal{S}^\ABV_\phi$.
\begin{enumerate}
\labitem{(Basis)}{G2:basis}
The span of the distributions 
\[
\Theta^{}_{\phi,s} = \sum_{\pi\in \Pi^\ABV_\phi(G_2(F))} (-1)^{\dim(\phi)-\dim(\pi)} \langle s,\pi\rangle\ \Theta_{\pi},
\qquad\qquad s\in \mathcal{S}^\ABV_\phi,
\]
is equal to the span of the distributions $\Theta_\pi$, as $\pi$ ranges over $\Pi^\ABV_\phi(G_2(F))$ if and only if $\Pi^\ABV_\phi(G_2(F)) \to \widehat{A^\ABV_\phi}$ is a bijection. 
In this case, the inverse of the linear system of equations above is
\[
\Theta^{}_{\pi} = \sum_{s\in \mathcal{A}^\ABV_\phi} (-1)^{\dim (C_\phi^s)-\dim(\pi)} \frac{\overline{\langle s,\pi\rangle}}{\abs{Z_{A^\ABV_\phi}(s)}} \ \Theta_{\phi,s},
\qquad\qquad \pi \in \Pi^\ABV_\phi(G_2(F)),
\]
where $\phi$ is a Langlands parameter for $\pi$, the taken over representatives $s$ of the fibres of $\mathcal{S}^\ABV_\phi \to A^\ABV_\phi$; here $\dim (C_\phi^s)$ is the dimension of the $s$-fixed points of the orbit $C_\phi\subseteq V_{\lambda_\phi}$. 
More generally, if $\pi$ is a unipotent representation of $G_2(F)$ then its distribution character $\Theta_\pi$ can be written as a linear combination of the distributions $\Theta^{}_{\phi,s}$, by letting $\phi$ range over unramified $L$-parameters for $G_2(F)$ with the same infinitesimal parameter as $\pi$, and letting $s$ range over $\mathcal{S}^\ABV_\phi$.
\labitem{(Stable)}{G2:stable} 
Suppose $\Theta^{}_{\phi}$ is stable when $\phi$ is elliptic and open. 
Then $\Theta^{}_{\phi}$ is stable for all unramified $\phi$. Moreover, these distributions form a basis for the space of stable unipotent distributions, letting $\phi$ range over unramified $L$-parameters for $G_2(F)$.
\end{enumerate}
\end{theorem}

\begin{proof}
Let $\phi$ be an unramified Langlands parameter for $G_2(F)$.
\begin{enumerate}
\item[\ref{G2:basis}]
$\Pi^\ABV_\phi(G_2(F)) \to \widehat{A^\ABV_\phi}$ is a bijection in all cases except when $\phi= \phi_{\ref{infcase:8-sub}b}$ and $\phi= \phi_{\ref{infcase:8-sub}c}$.
Moreover, $A^\ABV_\phi$ is trivial unless the infinitesimal parameter for $\phi$ is given by Cases~\ref{infcase:4-D2}, \ref{infcase:6-A2} or \ref{infcase:8-sub}. Accordingly, there are precisely 10 non-trivial cases of the first claim: $\phi_{\ref{infcase:4-D2}a}$, $\phi_{\ref{infcase:4-D2}b}$, $\phi_{\ref{infcase:4-D2}c}$, $\phi_{\ref{infcase:4-D2}d}$, $\phi_{\ref{infcase:6-A2}a}$, $\phi_{\ref{infcase:6-A2}b}$, $\phi_{\ref{infcase:6-A2}c}$, $\phi_{\ref{infcase:6-A2}d}$, $\phi_{\ref{infcase:8-sub}a}$ and $\phi_{\ref{infcase:8-sub}d}$. 
Consider the first of these cases, namely $\phi = \phi_{\ref{infcase:6-A2}a}$.
Recall that $A^\ABV_{\phi_{\ref{infcase:6-A2}a}} = \langle \theta_3 \rangle$ and let $s_3\in \mathcal{S}^\ABV_{\phi_{\ref{infcase:6-A2}a}}$ be a representative for $\theta_3$ above $\mathcal{S}^\ABV_{\phi_{\ref{infcase:6-A2}a}} \to A^\ABV_{\phi_{\ref{infcase:6-A2}a}}$.
Then
\[
\begin{array}{rcl}
\Theta^{}_{\phi_{\ref{infcase:6-A2}a}} 
	&=& \Theta_{\pi(\phi_{\ref{infcase:6-A2}a})} - \Theta_{\pi(\phi_{\text{\ref{infcase:6-A2}d}},\vartheta_3)} - \Theta_{\pi(\phi_{\text{\ref{infcase:6-A2}d}},\vartheta_3^2)} \\
\Theta^{}_{\phi_{\ref{infcase:6-A2}a},s_3} 
	&=& \Theta_{\pi(\phi_{\ref{infcase:6-A2}a})} - \theta_3 \Theta_{\pi(\phi_{\text{\ref{infcase:6-A2}d}},\vartheta_3)} - \theta_3^2 \Theta_{\pi(\phi_{\text{\ref{infcase:6-A2}d}},\vartheta_3^2)} \\
\Theta^{}_{\phi_{\ref{infcase:6-A2}a},s_3^2} 
	&=& \Theta_{\pi(\phi_{\ref{infcase:6-A2}a})} - \theta_3^2 \Theta_{\pi(\phi_{\text{\ref{infcase:6-A2}d}},\vartheta_3)} - \theta_3 \Theta_{\pi(\phi_{\text{\ref{infcase:6-A2}d}},\vartheta_3^2)} ,
\end{array}
\]
where $\pi(\phi_{\ref{infcase:6-A2}a})= J_{\gamma_2}(1,I^{\GL_2}(\theta_3\otimes\theta_3^{-1}))$, $\pi(\phi_{\ref{infcase:6-A2}d},\vartheta_3) = I_0(G_2[\theta_3])$ and  $\pi(\phi_{\ref{infcase:6-A2}d},\vartheta_3^2) =I_0( G_2[\theta_3^2])$.
The transition matrix for this system of equations is the product of the matrix for $\langle s,\pi\rangle$, which is the character table for $A^\ABV_{\phi_{\ref{infcase:6-A2}a}} = \langle \theta_3  \rangle$, and the diagonal, self-inverse matrix of signs $(-1)^{\dim(\phi)-\dim(\pi)}$:
\[
\begin{pmatrix}
1 & -1 & -1 \\
1 & -\theta_3 & -\theta^2_3 \\
1 & -\theta_3^2 & -\theta_3 
\end{pmatrix}
=
\begin{pmatrix}
1 & 1 & 1 \\
1 & \theta_3 & \theta^2_3 \\
1 & \theta_3^2 & \theta_3 
\end{pmatrix}
\begin{pmatrix}
1 & 0 & 0 \\
0 & -1 & 0 \\
0 &0 & 1 
\end{pmatrix}.
\]
So,
\[
\begin{pmatrix}
1 & -1 & -1 \\
1 & -\theta_3 & -\theta^2_3 \\
1 & -\theta_3^2 & -\theta_3 
\end{pmatrix}^{-1}
=
\begin{pmatrix}
1 & 0 & 0 \\
0 & -1 & 0 \\
0 &0 & 1 
\end{pmatrix}
\begin{pmatrix}
1/3 & 1/3 & 1/3 \\
1/3 & {\bar\theta}_3/3 & {\bar\theta}^2_3/3 \\
1/3 & {\bar\theta}^2_3/3 & {\bar\theta}_3/3 
\end{pmatrix}
\]
To see the formula for the inverse given in the Theorem is correct, it only remains to verify that the diagonal entries in the first matrix are given by $(-1)^{\dim(C_\phi^s) -\dim(C_\pi)}$, which is done case-by-case.
All the other case follow by the same argument.
%
We illustrate a non-Abelian case, $\phi = \phi_{\ref{infcase:8-sub}d}$, for which $A^\ABV_{\phi_{\ref{infcase:8-sub}d}} = S_3$. Let $s_2$ and $s_3$ be representatives for the two non-trivial conjugacy classes in $S_3$, of orders $2$ and $3$, respectively. 
Then
\[
\begin{array}{rcl}
\Theta^{}_{\phi_{\ref{infcase:8-sub}d}} 
	&=& \Theta_{\pi(\phi_{\ref{infcase:8-sub}d},1)} + 2\Theta_{\pi(\phi_{\ref{infcase:8-sub}d},\varrho)} + \Theta_{\pi(\phi_{\ref{infcase:8-sub}d},\varepsilon)} \\
\Theta^{}_{\phi_{\ref{infcase:8-sub}d},s_2} 
	&=& \Theta_{\pi(\phi_{\ref{infcase:8-sub}d},1)}  - \Theta_{\pi(\phi_{\ref{infcase:8-sub}d},\varepsilon)} \\	
\Theta^{}_{\phi_{\ref{infcase:8-sub}d},s_3} 
	&=& \Theta_{\pi(\phi_{\ref{infcase:8-sub}d},1)} - \Theta_{\pi(\phi_{\ref{infcase:8-sub}d},\varrho)} + \Theta_{\pi(\phi_{\ref{infcase:8-sub}d},\varepsilon)},
\end{array}
\]
where $\pi(\phi_{\ref{infcase:8-sub}d},1)=  \pi(1)'$, $\pi(\phi_{\ref{infcase:8-sub}d},\varrho)= \pi(1)$ and $\pi(\phi_{\ref{infcase:8-sub}d},\varepsilon) = I_0(G_2[1])$.
The matrix for this system of equations is the character table for $S_3$. 
The two examples above illustrate that both the complex conjugation and the denominator $\abs{Z_{A^\ABV_\phi}(s)}$ are needed for $G_2(F)$.

Note from Table~\ref{table:ABV} that $\Pi^\ABV_\phi(G_2(F)) \to \widehat{A^\ABV_\phi}$ is a bijection for all unramified Langlands parameters for $G_2(F)$ except $\phi= \phi_{\ref{infcase:8-sub}b}$ and $\phi= \phi_{\ref{infcase:8-sub}c}$.
For $\phi = \phi_{\ref{infcase:8-sub}b}$ we have $A^\ABV_{\phi_{\ref{infcase:8-sub}b}} = \langle \theta_2 \rangle$ and, after picking representative for $\theta_2$ above $\mathcal{S}^\ABV_{\phi_{\ref{infcase:8-sub}b}} \to A^\ABV_{\phi_{\ref{infcase:8-sub}b}}$, we find
\[
\begin{array}{rcl}
\Theta^{}_{\phi_{\ref{infcase:8-sub}b}} 
	&=& \Theta_{\pi(\phi_{\ref{infcase:8-sub}b})} - \Theta_{\pi(\phi_{\ref{infcase:8-sub}c})} + \Theta_{\pi(\phi_{\ref{infcase:8-sub}d},\varepsilon)} \\
\Theta^{}_{\phi_{\ref{infcase:8-sub}b},s_2} 
	&=& \Theta_{\pi(\phi_{\ref{infcase:8-sub}b})} - \theta_2 \Theta_{\pi(\phi_{\ref{infcase:8-sub}c})} + \Theta_{\pi(\phi_{\ref{infcase:8-sub}d},\varepsilon)},
\end{array}
\]
where $\pi(\phi_{\ref{infcase:8-sub}b})=  J_{\gamma_1}(1/2,\St_{\GL_2})$, $\pi(\phi_{\ref{infcase:8-sub}c})= J_{{\gamma_2}}(1/2,\St_{\GL_2})$ and, as above, $\pi(\phi_{\ref{infcase:8-sub}d},\varepsilon) = I_0(G_2[1])$.
While it is not possible to express $\Theta_{\pi(\phi_{\ref{infcase:8-sub}b})}$, $\Theta_{\pi(\phi_{\ref{infcase:8-sub}b})}$ and $\Theta_{\pi(\phi_{\ref{infcase:8-sub}b})}$ as a linear combination of $\Theta^{}_{\phi_{\ref{infcase:8-sub}b}}$ and $\Theta^{}_{\phi_{\ref{infcase:8-sub}b},s_2}$, we can write $\Theta_{\pi(\phi_{\ref{infcase:8-sub}d},\varepsilon)}$ in terms of $\Theta^{}_{\phi_{\ref{infcase:8-sub}d}}$ , $\Theta^{}_{\phi_{\ref{infcase:8-sub}d},s_2}$  and $\Theta^{}_{\phi_{\ref{infcase:8-sub}d},s_3}$, as above, and then express $\Theta_{\pi(\phi_{\ref{infcase:8-sub}b})}$, $\Theta_{\pi(\phi_{\ref{infcase:8-sub}b})}$ and $\Theta_{\pi(\phi_{\ref{infcase:8-sub}b})}$ in terms of $\Theta^{}_{\phi_{\ref{infcase:8-sub}b}}$, $\Theta^{}_{\phi_{\ref{infcase:8-sub}b},s_2}$, $\Theta^{}_{\phi_{\ref{infcase:8-sub}d}}$ , $\Theta^{}_{\phi_{\ref{infcase:8-sub}d},s_2}$  and $\Theta^{}_{\phi_{\ref{infcase:8-sub}d},s_3}$.  

The case $\phi= \phi_{\ref{infcase:8-sub}c}$ is similar in that there are more representations in the packet $\Pi^\ABV_{\phi_{\ref{infcase:8-sub}c}}(G_2(F))$ than there are conjugacy classes in $A^\ABV_{\phi_{\ref{infcase:8-sub}c}}$, but again, we can write all the distribution characters for representations in $\Pi^\ABV_{\phi_{\ref{infcase:8-sub}c}}(G_2(F))$ using the distributions $\Theta_{\phi_{\ref{infcase:8-sub}c}}$ and $\Theta_{\phi_{\ref{infcase:8-sub}c},s_2}$ together with the distributions $\Theta_{\phi_{\ref{infcase:8-sub}d}}$, $\Theta_{\phi_{\ref{infcase:8-sub}d},s_2}$ and $\Theta_{\phi_{\ref{infcase:8-sub}d},3}$, as above.

\item[\ref{G2:stable}] 
This follows from calculations that are made by working through the classification of unipotent representations of $G_2(F)$ appearing in Section~\ref{ssec:infcases}. 
In Case \ref{infcase:8-sub} this is \cite{CFZ:cubics}*{Theorem 2.16}. We give the details here of the proof in case \ref{infcase:6-A2}.
By \cite{Muic}*{Proposition 4.2}, in the Grothendieck group of representations of $G_2(F)$, we have
\begin{equation*}
\begin{array}{rcl}
 M(\phi_{\ref{infcase:6-A2}a})&=&I(\nu\theta_3\otimes \theta_3)=\pi(\phi_{\text{\ref{infcase:6-A2}d}})+\pi(\phi_\text{\ref{infcase:6-A2}b})+\pi(\phi_\text{\ref{infcase:6-A2}c})+\pi(\phi_{\ref{infcase:6-A2}a}),\\
M(\phi_\text{\ref{infcase:6-A2}b})&=&I_{\gamma_1}(1/2,\theta_3^{-1}\otimes \St_{\GL_2})=\pi(\phi_\text{\ref{infcase:6-A2}d})+\pi(\phi_\text{\ref{infcase:6-A2}b}),\\
M(\phi_\text{\ref{infcase:6-A2}c})&=&I_{\gamma_1}(1/2,\theta_3\otimes\St_{\GL_2})=\pi(\phi_\text{\ref{infcase:6-A2}d})+\pi(\phi_\text{\ref{infcase:6-A2}c}),
\end{array}
\end{equation*}
where $M(\phi)$ means the standard module of $\pi(\phi)$ for a Langlands parameter $\phi$. From the expression of $\Theta_{\phi}$ for each $\phi$ appeared in the proof of last part, we easily see that
\[
\bpm 
\Theta^{}_{\phi_\text{\ref{infcase:6-A2}a}}\\ 
\Theta^{}_{\phi_\text{\ref{infcase:6-A2}b}}\\ 
\Theta^{}_{\phi_\text{\ref{infcase:6-A2}c}}\\ 
\Theta^{}_{\phi_\text{\ref{infcase:6-A2}d}} 
\epm
= 
\bpm 1&-1&-1&1\\ 0 &1&0 &-1\\ 0 & 0 &1&-1\\ 0 &0 &0 &1 \epm 
\bpm 
\Theta_{M(\phi_\text{\ref{infcase:6-A2}a})}\\ 
\Theta_{M(\phi_\text{\ref{infcase:6-A2}b})}\\ 
\Theta_{M(\phi_\text{\ref{infcase:6-A2}c})}\\ 
\Theta^{}_{\phi_{\text{\ref{infcase:6-A2}d}}}\epm .
\]
Note that each standard module is stable since it is a representation either induced from $\GL_2(F)$ or $T=\GL_1(F)\times \GL_1(F)$. This shows that if $\Theta^{}_{\phi_\text{\ref{infcase:6-A2}d}}$ is stable, then each $\Theta^{}_{\phi}$ is stable for every $\phi$ in Case \ref{infcase:6-A2}. For the rest cases, the proof is similar and thus omitted. 
\end{enumerate}
\end{proof}

\begin{remark}\label{rem:basis}
The appearance of $\dim(C_\phi^s)$ in Theorem~\ref{thm:EC} is a foreshadowing of the geometric constructions in Section~\ref{sec:geoendo} on geometric endoscopy.
\end{remark}

\begin{remark}\label{rem:stability}
It is widely expected that $\Theta^{}_{\phi}$ is stable when $\phi$ is elliptic, but it seems that this has not been proved; see \cite{CFZ:cubics}*{Remark~2.17}. Note that, if confirmed \ref{G2:stable}, could then be strengthened to: $\Theta^{}_{\phi}$ is stable for all unramified $\phi$. 
\end{remark}

\section{Unipotent representations of endoscopic groups for $G_2$}\label{sec:endoscopy}

\subsection{Endoscopic groups for $G_2$}\label{ssec:endoscopy}

Let $(G,s,\xi)$ be an endoscopic triple for $G_2$.
Since $s\in \dualgroup{G_2}$ is semisimple we choose $s\in \dualgroup{T}$.
Since $Z_{\dualgroup{G_2}}(s)$ is connected, $G$ is split over $F$ and $\xi : \Lgroup{G} \to \Lgroup{G_2}$ is given by $\dualgroup{G} \ceq Z_{\dualgroup{G_2}}(s) \hookrightarrow \dualgroup{G_2}$; see \cite{Kottwitz-Shelstad}*{(2.1.4a) and (2.1.4b)}.
The endoscopic triples for $G_2$ are naturally arranged into the following six families, of which the last three are elliptic.
\begin{enumerate}
\labitem{(T)}{endcase:torus}
$G= T= \GL_1^2$, $s= {\hat m}(x,y)$ regular, $x^{-1}y^2\ne 1$, $xy^{-1}\ne 1$, $x\ne 1$, $y\ne 1$, $x^2y^{-1}\ne 1$, $xy\ne 1$;
\labitem{(A\textsubscript{1}\hskip-5pt\textsuperscript{s})}{endcase:short}
$G=\GL_2^{\hat\gamma_2}$, $s= {\hat m}(x,x)$, $x^2\ne 1$;
\labitem{(A\textsubscript{1}\hskip-5pt\textsuperscript{l})}{endcase:long} 
$G=\GL_2^{\hat\gamma_1}$, $s= {\hat m}(x^2,x)$, $x^2\ne1$, $x^3\ne 1$;
\labitem{(D\textsubscript{2})}{endcase:D2} 
$G=\SO_4$ split over $F$, $s= {\hat m}(1,-1)$;
\labitem{(A\textsubscript{2})}{endcase:A2} 
$G=\PGL_3$, $s= {\hat m}(\theta_3^2,\theta_3)$;
\labitem{(G\textsubscript{2})}{endcase:G2} 
$G=G_2$ over $F$ (necessarily split), $s= 1$.
\end{enumerate}
In Sections~\ref{ssec:GL2}, \ref{ssec:SO4} and \ref{ssec:PGL3} we study the unipotent representations of $\GL_2$, $\SO_4$ and $\PGL_3$ and their inner forms, respectively, including the Langlands correspondence and ABV-packets for these representations.

We will also study distributions attached to ABV-packets, for which we require the following definition, related closely to \cite{CFMMX}*{Definition 3}.

\begin{table}[tp]\label{table:endoscopy}
\caption{Endoscopic relations for $G_2$. The dotted arrows $G\dashrightarrow G'$ indicate that $G$ is an endoscopic group for $G'$ and therefore the existence of an injective admissible homomorphism $\xi : \Lgroup{G} \to \Lgroup{G'}$. There $T'$ is the form of the torus $T=\GL_1^2$ that splits over the quadratic unramified extension of $F$.}
\begin{center}
$
\begin{tikzcd}
&& G_2(F) && \\
\PGL_3^\delta(F)& \arrow[dashed]{ld}  \arrow[dashed]{l} \PGL_3(F) \arrow[dashed]{ur} && \SO_4(F) \arrow[dashed]{ul} \arrow[dashed]{r} &  \SO_4^{\delta}(F) \\
\PGL_3^{\delta'}(F)  &&& & \\
  & \GL_2^{\hat\gamma_1}(F) \arrow[dashed]{uurr}
  \arrow[dashed]{uu}   && 
   \GL_2^{\hat\gamma_2}(F) 
    \arrow[dashed]{uu} &  \\
& & T(F) \arrow[dashed]{ul}  \arrow[dashed]{ur}  &  & \arrow[dashed]{ll} \arrow[dashed]{uuul}  T'(F) 
\end{tikzcd}
$
\end{center}
\end{table}%

Let $(G,s,\xi)$ be an endoscopic triple for $G_2$.
Recall from \cite{CFMMX}*{Section 8.1, Definition 1} that $\Pi^\ABV_{\phi}(G/F)$ is a subset of $\Pi^\mathrm{pure}(G/F)$, thus contained in the set of equivalence classes of irreducible admissible representations of $G(F)$ and all its pure inner forms.
Let $\delta \in Z^1(F,G)$ be a pure inner form of $G$; recall that $G^\delta$ denotes form of $G$ determined by the image of $\delta$ under $Z^1(F,G) \to Z^1(F,\Aut(G))$. 
Now let $\phi : W'_F \to \Lgroup{G}$ be an unramified Langlands parameter.
It is possible that the $L$-packet $\Phi_\phi(G^\delta(F))$ is empty, since we have not demanded that the Langlands parameter $\phi : W'_F \to \Lgroup{G}$ is relevant to $G(F)$ in the sense of \cite{Borel:Automorphic}*{Section 8.3 (ii)}. If $\phi$ is relevant to $G^\delta(F)$, then $\Phi_\phi(G^\delta(F))$ is non-empty.

We write $\Pi^\ABV_{\phi}(G^\delta(F))$ for the subset of $\Pi^\ABV_{\phi}(G/F)$ that contains only representations of $G^\delta(F)$.
Again, it is possible that $\Pi^\ABV_{\phi}(G^\delta(F))$ is empty when $\phi$ is not relevant to $G^\delta(F)$.

\begin{definition}\label{def:Thetaphis-endo}
Let $\phi : W'_F \to \Lgroup{G}$ be an unramified Langlands parameter; let $s$ be an element of $\mathcal{S}^\ABV_{\xi\circ\phi}$.
Set
\[
\Theta_{\phi,s}^{G^\delta}
\ceq  
e(\delta) 
\sum_{\pi \in \Pi^\ABV_{\phi}(G^\delta(F))} \hskip-0pt  \trace_{s} \left(\NEvs_{\phi}[\dim(\phi)] \mathcal{P}(\pi)[-\dim(\pi)] \right)\  \Theta_{\pi}.
\]
When written using ABV-packet coefficients for $G^\delta$ using Definition~\ref{def:SABV}, the definition of $\Theta_{\phi,s}^{G^\delta}$, above, takes the form:
\[
\Theta_{\phi,s}^{G^\delta}
= 
e(\delta) 
\sum_{\pi \in \Pi^\ABV_{\phi}(G^\delta(F))} \hskip-0pt(-1)^{\dim(\phi)-\dim(\pi)}\  \langle s,\pi\rangle\  \Theta_{\pi}.
\]
We also set $\Theta_{\phi}^{G^\delta}\ceq\Theta_{\phi,1}^{G^\delta}$.
\end{definition}

\subsection{Unipotent representations of $T=\GL_1^2$}\label{ssec:T}

Unipotent representations of $T(F)$ are simply unramified characters $ \chi_1\otimes\chi_2: T(F)\to \CC^\times$. The Langlands parameter for this character is $W'_F \to \dualgroup{T}$ defined by $(w,x)\mapsto \varphi_1(w)\times \varphi_2(w)$ where $\varphi_i$ corresponds to $\chi_i$ under class field theory.

\subsection{Unipotent representations of $\GL_2$}\label{ssec:GL2}

The category of unipotent representations of the group $\GL_2(F)$ is precisely the category of unramified principal series representations of $\GL_2(F)$, corresponding therefore to the subcategory of smooth representations of $\GL_2(F)$ appearing in the Bernstein decomposition indexed by the inertial class of the trivial representation of $T(F)$.
In particular, there are no supercuspidal unipotent representations of $\GL_2(F)$.
In this section we partition the set $\Pi(\GL_2(F))_\text{unip}$ of unipotent representations of $\GL_2(F)$ into $L$-packets and also describe the corresponding $L$-parameters, thus giving the Langlands correspondence for these representations.
Since this case is very well known, we omit all proofs here.

\begin{table}[htp]
\caption{$L$-packets for unipotent representations of $\GL_2(F)$: each row in an $L$-packet. The notation and the corresponding $L$-parameters are explained below. }
\label{table:GL2:$L$-packets}
\begin{center}
$
\begin{array}{| c | c | l | } 
\hline
\text{$L$-parameter} & \text{$L$-packet} & \text{Arthur} \\ 
\phi & \Pi_{\phi}(\GL_2(F)) & \text{type?} \\ 
\hline\hline
\phi_{\ref{ssec:GL2}.0} 	& \Ind_{B_2(F)}^{\GL_2(F)}(\chi_1\otimes \chi_2) & \text{unitary $\chi_1$, $\chi_2$}\\
\hline\hline
\phi_{\ref{ssec:GL2}.1a}	&  (\chi\circ\det)\otimes \1_{\GL_2}  & \text{unitary $\chi$ } \\
\phi_{\ref{ssec:GL2}.1b}	&  (\chi\circ\det)\otimes \St_{\GL_2} & \text{unitary $\chi$} \\
\hline
\end{array}
$
\end{center}
\end{table}%

We fix the standard maximal split torus $T(F)$ of $\GL_2(F)$. 
Let $B_2(F)$ be the standard Borel subgroup of $\GL_2(F)$.

\begin{enumerate}
\item[\ref{ssec:GL2}.0]
Let $\chi_1$ and $\chi_2$ be unramified complex characters of $\GL_1(F) = F^\times$ and let $\chi_1\otimes\chi_2$ be the associated character of $T(F)$.
Then the parabolically induced representation
\[
I^{\GL_2}(\chi_1\otimes \chi_2)\ceq\Ind_{B_2(F)}^{\GL_2(F)}(\chi_1\otimes \chi_2)
\]
is irreducible if and only if $\chi_1 \chi_2^{-1} \ne \nu^{\pm 1}$, where $\nu$ is the unramified character of $\GL_1(F)$ defined by $\nu(\varpi) = q$.
In this case, $\Ind_{B_2(F)}^{\GL_2(F)}(\chi_1\otimes \chi_2)$ is its own $L$-packet.
The $L$-parameter $\phi : W'_F \to \dualgroup{\GL_2}$ for this irreducible representation is
\[
\phi_{\ref{ssec:GL2}.0}(w,x) = \begin{pmatrix} \varphi_1(w) & 0 \\ 0 & \varphi_2(w) \end{pmatrix}
\]
where $\varphi_i : W_F \to \dualgroup{\GL_1}$ is the character corresponding to $\chi_i : \GL_1(F) \to \CC^\times$ under local class field theory.
This $L$-parameter is of Arthur type if and only if $\chi_1$ and $\chi_2$ are unitary.

\item[\ref{ssec:GL2}.1]
Now suppose $\chi_1$ and $\chi_2$ are unramified and $\chi_1 \chi_2^{-1} = \nu$. 
Then we may write $\chi_1 =  \chi \nu^{1/2}$ and  $\chi_2 =  \chi \nu^{-1/2}$, in which case
\[
\Ind_{B_2(F)}^{\GL_2(F)}(\chi \nu^{1/2}\otimes \chi \nu^{-1/2}) 
= (\chi\circ\det)\otimes \Ind_{B_2(F)}^{\GL_2(F)}(\nu^{1/2}\otimes \nu^{-1/2}).
\]
The unique irreducible sub-representation of $\Ind_{B_2(F)}^{\GL_2(F)}(\nu^{1/2}\otimes \nu^{-1/2})$ is the Steinberg representation, $\St_{\GL_2}$, whereas the trivial representation, $\1_{\GL_2(F)}$, is the unique irreducible quotient.
The $L$-parameter for $(\chi\circ\det)\otimes \1_{\GL_2(F)}$ is  
\[
\phi_{\ref{ssec:GL2}.1a}(w,x) = \varphi(w) \begin{pmatrix} \abs{w}^{1/2} & 0 \\ 0 &  \abs{w}^{-1/2} \end{pmatrix},
\]
where $\varphi : W_F \to \C^\times$ is the character corresponding to $\chi : F^\times \to \CC^\times$ under class field theory.
The $L$-parameter for $(\chi\circ\det)\otimes \St_{\GL_2}$ is 
\[
\phi_{\ref{ssec:GL2}.1b}(w,x) = \varphi(w) x .
\]
Both $L$-parameters appearing in this case are of Arthur type if and only if $\chi$ is unitary.
We remark that $(\chi\circ\det)\otimes \1_{\GL_2(F)}$ and $(\chi\circ\det)\otimes \St_{\GL_2}$ are interchanged by Aubert-Zelevinski duality.
\end{enumerate}

This completes the description of the Langlands correspondence for unipotent representations of $\GL_2(F)$.
Each $L$-packet for $\GL_2(F)$ is its own ABV-packet. 

Theorems~\ref{thm:coefficients} and \ref{thm:stable} are largely trivial in the case of $\GL_2(F)$ so we refrain from stating them here except to say that \cite{CFMMX}*{Conjecture 1} is true for unramified representations of $\GL_2(F)$. 

\begin{remark}
The category of unipotent representations of $\GL_2(F)$ is equivalent to the category of modules over the affine Hecke algebra for $\GL_2(F)$.  
As such, Table~\ref{table:GL2:$L$-packets} lists the simple objects in this module category.
\end{remark}

\subsection{Unipotent representations of $\SO_4$ and its pure inner forms}\label{ssec:SO4}

In this section we find the ABV-packets and ABV-packet coefficients for all unipotent representations of $\SO_4$ and its pure inner forms.
The $p$-adic group $\SO_4$ has two pure inner forms, since $H^1(F,\SO_4)$ has order $2$, and $H^1(F,\SO_4) \to H^1(F,\Aut(\SO_4))$ is injective.
Let $\delta \in Z^1(F,\SO_4)$ be a representative for the non-trivial class in $H^1(F,\SO_4)$. 
Let $\SO_4^\delta$ be the inner form of $\SO_4$ attached to this cocycle.
We partition the set $\Pi(\SO_4(F))_\text{unip}$ of unipotent representations of the split orthogonal group $\SO_4(F)$ into $L$-packets and also describe the corresponding $L$-parameters.
We do the same for the non-split inner form $\SO_4^\delta$ of $\SO_4$. 
The group $\SO_4^\delta$ is the orthogonal group for the quadratic space coming from the norm on the unique quaternion division algebra over $F$.
 
\begin{table}[htp]
\caption{ABV-packets for all unipotent representations of $\SO_4(F)$ and its inner form $\SO_4^\delta(F)$}
\label{table:LA-packetsD}
\begin{center}
$
\begin{array}{| c | c | c | l | } 
\hline
\text{$L$-parameter}	& \text{Pure $L$-packet} & \text{Coronal} &\text{Parameter of}\\ 
\phi & \Pi^\text{pure}_\phi(\SO_4/F) & \text{representations} & \text{Arthur type?} \\ 

\hline\hline
\phi_{\ref{ssec:SO4}.0} 	& I^{\SO_4}(\chi_1\otimes \chi_2)		&   &\text{unitary $\chi_1$, $\chi_2$}\\
\hline\hline
\phi_{\ref{ssec:SO4}.0'} 	& \pi_4,\pi_4' 	&   & \text{yes}\\
\hline\hline
\phi_{\ref{ssec:SO4}.1a} 	& I_{\beta_1}(\chi\circ \det) 	&  &\text{unitary $\chi$} \\
\phi_{\ref{ssec:SO4}.1b}	& I_{\beta_1}(\chi\otimes \St_{\GL_2})   & &\text{unitary $\chi$} \\
\hline\hline
\phi_{\ref{ssec:SO4}.2a} 	& I_{\beta_2}(\chi\circ\det)	&  & \text{unitary $\chi$} \\
\phi_{\ref{ssec:SO4}.2b} 	& I_{\beta_2}(\chi\otimes \St_{\GL_2})   &  & \text{unitary $\chi$} \\
\hline\hline
\phi_{\ref{ssec:SO4}.3a} & \chi_{\SO_4(F)}  & \chi_{\SO_4^\delta(F)} & \text{yes} \\
\phi_{\ref{ssec:SO4}.3b} &  J_{\beta_1}(1/2,\chi\otimes \St_{\GL_2})  & \chi_{\SO_4^\delta(F)} & \text{yes}  \\
\phi_{\ref{ssec:SO4}.3c} &  J_{\beta_2}(1/2,\chi\otimes \St_{\GL_2})  & \chi_{\SO_4^\delta(F)} & \text{yes} \\
\phi_{\ref{ssec:SO4}.3d} &  \chi_{\SO_4(F)}\otimes\St_{\SO_4},   \chi_{\SO_4^\delta(F)}  & & \text{yes}\\
\hline
\end{array}
$
\end{center}
\end{table}%

\subsubsection*{Langlands correspondence} 
The local Langlands correspondence for $\SO_4(F)$ may be assembled as follows. 
We start with the observation that $\GSO_4(F)\cong (\GL_2(F)\times \GL_2(F))/\Delta(F^\times)$, see \cite{Gross-Prasad}*{(15.1)} for example, and therefore the local Langlands correspondence for $\GSO_4(F)$ follows from that of $\GL_2(F)$. 
By the general theory of \cite{Gelbart-Knapp} and \cite{Gan-Takeda}, the local Langlands correspondence for $\SO_4(F)$ can be obtained in the following way. 
Notice that $\wh{\GSO_4(F)}\cong \GSpin_4(\C)$ and $\wh{\SO_4(F)}\cong \SO_4(\C)$. 
Let $\mathrm{std}:\GSpin_4(\C)\to \SO_4(\C)$ be the standard homomorphism. 
Let $\phi:W_F\times \SL_2(\C)\to  \SO_4(\C)$ be a local Langlands parameter for $\SO_4(F)$ and let $\wt \phi$ be a local Langlands parameter for $\GSO_4(F)$ such that $\mathrm{std}\circ \wt \phi=\phi$ and let $\Pi_{\wt \phi}(\GSO_4(F))$ be the corresponding local $L$-packet for $\GSO_4(F)$, which is a singleton. 
Assume that $\Pi_{\wt \phi}(\GSO_4(F))=\{\wt \pi\}$. 
Then $\wt \pi\vert_{\SO_4}$ is multiplicity free and $\Pi_\phi(\SO_4(F))=\mathrm{JH}(\wt \pi)$, where $\mathrm{JH}(\wt \pi)=\{\mathrm{constituents~ of~ } \wt\pi|_{\SO_4(F)}\}$. Consider the subgroup $F^\times \cdot \SO_4(F)\subset \GSO_4(F)$, we have $\GSO_4(F)/(F^\times \cdot \SO_4(F))\cong F^\times/F^{\times,2}$. By \cite{Gelbart-Knapp}*{Lemma 2.1}, we have 
$$\big\vert \Pi_{\phi}(\SO_4(F)) \big\vert = \big\vert\{\chi:F^\times/F^{\times,2}\ra \BC^\times: \wt \pi\otimes \chi=\wt \pi\}\big\vert.$$
This allows us to find the local Langlands correspondence for unipotent representations of $\SO_4(F)$ explicitly, as follows.

We realize $\GSO_4(F)$ and $\SO_4(F)$ as 
$\GSO_4(F)=\{g\in \GL_4(F): g^t Jg=\lambda(g)J, \lambda(g)\in F^\times\},$ and $\SO_4(F)=\{g\in \GSO_4(F): \lambda(g)=1\}$,
where $$J=\begin{pmatrix} 0&0&0&1\\ 0&0&1&0\\ 0&1&0&0\\1&0&0&0 \end{pmatrix}.$$ 
We denote the two simple roots of $\SO_4$ (and $\GSO_4$) by $\beta_1$ and $\beta_2$ and let $r\mapsto x_{\beta_i}(r)$ be the one parameter subgroup associated with $\beta_i$, which can be realized as 
\[
x_{\beta_1}(r)=\begin{pmatrix}1 &r&0&0\\ 0&1&0&0\\ 0&0&1&-r\\ 0&0&0&1 \end{pmatrix}, \quad
x_{\beta_2}(r)=\begin{pmatrix} 1&0&r&0\\ 0&1&0&-r\\ 0&0&1&0\\ 0&0&0&1 \end{pmatrix}.
\]
For $i=1,2$, the simple root $\beta_i$ define an embedding $\iota_{\delta_i}: \GL_2(F)\to \GSO_4(F)$ which can be explicitly described as 
\[
\iota_{\beta_1}\left(\bpm a& b\\ c&d \epm \right)=\begin{pmatrix}a &b&0&0\\ c&d&0&0\\ 0&0&a&-b\\ 0&0&-c&d \end{pmatrix}, \quad
x_{\beta_2}\left(\bpm a&b\\ c&d \epm \right)=\begin{pmatrix} a&0&b&0\\ 0&a&0&-b\\ c&0&d&0\\ 0&-c&0&d \end{pmatrix}.
\]
Note that $\lambda(\iota_{\beta_i}(g))=\det(g)$ for $g\in \GL_2(F)$. In particular, $\iota_{\beta_i}$ defines an embedding $\iota_{\beta_i}:\SL_2(F)\to \SO_4(F)$. Moreover, the map $\iota_{\beta_i}$ define an isomorphism $\GSO_4(F)\cong (\GL_2(F)\times \GL_2(F))/(\Delta(F^\times))$, where $\Delta(F^\times)=\{(a,a): a\in F^\times\}\subset \GL_2(F)\times \GL_2(F)$. Thus an irreducible representation of $\GSO_4(F)$ is of the form $\pi_1\boxtimes \pi_2$ with $\omega_{\pi_1}\cdot\omega_{\pi_2}=1$, where $\pi_i$ is an irreducible representation of $\GL_2(F)$ and $ \omega_{\pi_i}$ is the central character of $\pi_i$. Let $\chi$ be a character of $F^\times$, we denote $I^{\GL_2}(\chi)=\Ind_{B_{\GL_2}}^{\GL_2(F)}(\chi\otimes \chi^{-1})$.

Let $T_{\SO_4}$ be the maximal torus of $\SO_4(F)$ which consists of elements of the form  $ t(a,b):=\diag(a,b,b^{-1},a^{-1})$ for $a,b\in F^\times$. Let $B_{\SO_4} $ be the upper triangular Borel subgroup of $\SO_4(F)$.  Let $\chi_1,\chi_2$ be a pair of characters of $F^\times$. We denote $\chi_1\otimes \chi_2$ the character of $T_{\SO_4}$ defined by $\chi_1\otimes \chi_2(t(a,b))=\chi_1(a)\chi_2(b)$. Denote $I^{\SO_4}(\chi_1\otimes \chi_2)\ceq\Ind_{B_{\SO_4}}^{\SO_4(F)}(\chi_1\otimes \chi_2)$.  Note that $I^{\SO_4}(\chi_1\otimes \chi_2)\cong I^{\SO_4}(\chi_2\otimes \chi_1)$ in the Grothendieck group of representations of $\SO_4(F)$. 

Up to conjugacy, the group $\SO_4$ has 2 maximal parabolics $P_i=M_iN_i$ with $x_{\beta_i}\in M_i$ for $i=1,2$, where $M_i$ is the Levi subgroup of $P_i$. Note that $ M_i\cong \GL_2(F)$. Let $\sigma$ be an irreducible smooth representation of $\GL_2(F)\cong M_i$. Denote $I_{P_i}(\sigma):=\Ind_{P_i}^{\SO_4(F)}(\sigma)$.

The dual group $\SO_4(\C)$ is realized in a similar way as $\SO_4(F)$ described above. In particular, if we call $\wh \beta_i$ the root of $\SO_4(\C)$ which is dual to $\beta_i$, then the root space $x_{\hat \beta_i}$ is described using the same matrix as $x_{\beta_i}$. Similarly, $\iota_{\wh \beta_i}$ defines an embedding $\SL_2(\C)\to \SO_4(\C)$. Denote ${\hat t}(a,b)=\diag(a,b,b^{-1},a^{-1})\in \SO_4(\C)$. 

\begin{enumerate}
\item[\ref{ssec:SO4}.0] 

Let $\chi_1$ and $\chi_2$ be unramified characters of $F^\times$ with $\chi_1\chi_2, \chi_1\chi_2^{-1}\ne \nu^{\pm 1}$, we consider the representation $I^{\SO_4}(\chi_1\otimes \chi_2)$. 
Then
\[
I^{\SO_4}(\chi_1\otimes \chi_2)=(I^{\GL_2}(\chi'_1)\boxtimes I^{\GL_2}(\chi'_2))|_{\SO_4(F)},
\]
where $\chi'_1,\chi'_2$ are unramified characters of $F^\times$ such that $\chi_1=\chi'_1\chi'_2, \chi_2=\chi'_2/\chi'_1$.  
We remark that $(\chi'_1,\chi'_2)$ are uniquely determined by $(\chi_1,\chi_2)$ up to a twist by $\theta_2$, where $\theta_2$ is the unique unramified quadratic character of $F^\times$.  
There are two cases to consider. 
\begin{enumerate}
\item[(i)]
Suppose $(\chi_1,\chi_2)\ne (\theta_2,1)$ and $(\chi_1,\chi_2)\ne (1,\theta_2)$.
Then the representation $I^{\SO_4}(\chi_1\otimes \chi_2)$ is irreducible. 
The Langlands parameter is given by 
\[
\phi_{\ref{ssec:SO4}.0}(w,x)={\hat t}(\varphi_1(w),\varphi_2(w)),
\]
where $\varphi_i:W_F\to \C^\times$ is the character dual to $\chi_i$,
and
\[
\Pi_{\phi_{\ref{ssec:SO4}.0}}(\SO_4(F)) = \{I^{\SO_4}(\chi_1\otimes \chi_2)\}.
\]
 The moduli space of Langlands parameters with the infinitesimal parameter $\lambda_{\ref{ssec:SO4}.0}$ for $\phi_{\ref{ssec:SO4}.0}$ is \ref{geocase-0}: $V_{\lambda_{\ref{ssec:SO4}.0}} = \{ 0 \}$ and where $H_{\lambda_{\ref{ssec:SO4}.0}}$ is equal to:  
 $\GL_2$, when exactly one of $\varphi_1\varphi_2, \varphi_1\varphi_2^{-1}$ is 1; 
$\SO_4$, when $\varphi_1\varphi_2 = \varphi_1\varphi_2^{-1} = 1$; 
$\dualgroup{T} = \GL_1^2$ in all other cases where $(\varphi_1,\varphi_2)\ne (\theta_2,1)$ and $(\varphi_1,\varphi_2)\ne (1,\theta_2)$.

\item[(ii)]
Suppose $(\chi_1,\chi_2)= (\theta_2,1)$ or $(\chi_1,\chi_2)= (1,\theta_2)$.
Then 
\[
I^{\SO_4}(\theta_2\otimes 1)=I^{\SO_4}(1\otimes \theta_2)
\]
and this representation is a direct sum of two irreducible tempered representations of $\SO_4(F)$. We denote these two representations by $\pi_4$ and $\pi_4'$. Only one of $\pi_4,\pi_4'$ is $\SO_4(\CO_F)$-spherical, {\it i.e.}, has a nonzero vector fixed by the hyperspecial group $\SO_4(\CO_F)$, where $\CO_F$ is the ring of integers of $F$, and we choose notation so that $\pi_4$ is $\SO_4(\CO_F)$-spherical. 
The Langlands parameter for $\pi_4$ and $\pi'_4$ is
\[
\phi_{\ref{ssec:SO4}.0'}(w,x)={\hat t}(\varphi_1(w),\varphi_2(w)),
\]
where $\varphi_i$ corresponds to $\chi_i$ under class field theory. 
and
\[
\Pi_{\phi_{\ref{ssec:SO4}.0'}}(\SO_4(F)) = \{\pi_4,  \pi'_4\}.
\]
The moduli space of Langlands parameters with the infinitesimal parameter $\lambda_{\ref{ssec:SO4}.0'}$ for $\phi_{\ref{ssec:SO4}.0'}$ is \ref{geocase-0}: $V_{\lambda_{\ref{ssec:SO4}.0'}} = \{ 0 \}$ where $H_{\lambda_{\ref{ssec:SO4}.0}}$ is the disconnected group ${\rm S}({\rm O}_2\times{\rm O}_2)$, namely, the group generated by  $\dualgroup{T}$ and $J$.
\end{enumerate}

\item[\ref{ssec:SO4}.1]  
Let $\chi$ be an unramified character of $F^\times$ with $\chi^2\ne  \nu^{\pm 1}$. The irreducible components of $I(\chi \nu^{1/2}\otimes \chi \nu^{-1/2})$ are $I_{\beta_1}(\chi\otimes \St_{\GL_2})$ and $I_{\beta_1}(\chi\circ \det)$, where $I_{\beta_1}(\chi\otimes \St_{\GL_2})=(\St_{\GL_2}\boxtimes I^{\GL_2}(\chi))|_{\SO_4(F)}$ and $I_{\beta_1}(\chi\circ \det)=(1_{\GL_2}\boxtimes I^{\GL_2}(\chi))|_{\SO_4(F)}$.
The local Langlands parameter of $I_{\beta_1}(\chi\circ \det)$ is given by
\[
\phi_{\ref{ssec:SO4}.1a}(w,x)={\hat t}(\varphi(w)|w|^{1/2}, \varphi(w)|w|^{-1/2}),
\]
and the Langlands parameter of $I_{\beta_1}(\chi\otimes \St_{\GL_2})$ is given by
\[
\phi_{\ref{ssec:SO4}.1b}(w,x)\to {\hat t}(\varphi(w),\varphi(w)) \iota_{\wh \beta_1}(x)
,
\]
where $\varphi:W_F\to \C^\times$ is the dual character of $\chi : F^\times \to \C^\times$. 
The Vogan variety $V_{\lambda_{\ref{ssec:SO4}.1}}$ is the root space $\mathfrak{u}_{\hat\beta_1}$ with $H_{\lambda_{\ref{ssec:SO4}.1}} = \dualgroup{T}$  action given by $t\cdot x = \beta_1(t)x$. This is equivalent to \ref{geocase-1}, and the two orbits in this prehomogenous vector space, $C_0$ and $C_1$, match $\phi_{\ref{ssec:SO4}.1a}$ and $\phi_{\ref{ssec:SO4}.1b}$, respectively.

\item[\ref{ssec:SO4}.2]  
Let $\chi$ be an unramified character with $\chi^2\ne  \nu^{\pm 1}$. The irreducible components of $I(\chi \nu^{1/2}\otimes \chi^{-1} \nu^{1/2})$ are $I_{\beta_2}(\chi\otimes \St_{\GL_2})$ and $I_{\beta_2}(\chi\circ \det)$ where $I_{\beta_2}(\chi\otimes \St_{\GL_2})=(I^{\GL_2}(\chi)\boxtimes\St_{\GL_2})|_{\SO_4(F)}$ and $I_{\beta_2}(\chi\circ \det)=( I^{\GL_2}(\chi)\boxtimes 1_{\GL_2})|_{\SO_4(F)}$.
The local Langlands parameter for $I_{\beta_2}(\chi\circ \det)$ is 
\[
\phi_{\ref{ssec:SO4}.2a}\left(w, x \right) = {\hat t}(\varphi(w)|w|^{1/2}, \varphi(w)^{-1}|w|^{1/2}),
\]
and the Langlands parameter for $I_{\beta_2}(\chi\otimes \St_{\GL_2})$ is 
\[
\phi_{\ref{ssec:SO4}.2b}\left(w,\bpm a& b\\ c&d \epm\right)={\hat t}(\varphi(w),\varphi(w)^{-1})\iota_{\wh \beta_2}(x),
\]
where $\varphi:W_F\to \C^\times$ is the dual character of $\chi$.
The Vogan variety $V_{\lambda_{\ref{ssec:SO4}.2}}$ is the root space $\mathfrak{u}_{\hat\beta_2}$ with $H_{\lambda_{\ref{ssec:SO4}.2}} = \dualgroup{T}$  action given by $t\cdot x = \wh\beta_2(t)x$. This is equivalent to \ref{geocase-1}, and the two orbits in this prehomogenous vector space, $C_0$ and $C_1$, match $\phi_{\ref{ssec:SO4}.2a}$ and $\phi_{\ref{ssec:SO4}.2b}$, respectively.

\item[\ref{ssec:SO4}.3]
Let $\chi$ be an unramified character of $F^\times$ such that $\chi^2=1$. Thus $\chi$ is either trivial or $\theta_2$. The representation $I(\chi\nu\otimes \chi)$ has 4 irreducible components:
\[
\chi_{\SO_4(F)},\ 
J_{\beta_1}(1/2,\chi\otimes \St_{\GL_2}),\ 
J_{\beta_2}(1/2,\chi\otimes \St_{\GL_2}),\  
\chi_{\SO_4(F)}\otimes\St_{\SO_4},
\]
where $\chi_{\SO_4}$ is the character $\chi$ composed with the spinor norm of $\SO_4(F)$, $\St_{\SO_4}$ is the Steinberg representation of $\SO_4(F)$, which is the unique subrepresentation of $I^{\SO_4}(\nu\otimes 1)$, and $J_{\beta_i}(s,\chi\circ \St_{\GL_2})$ is the unique quotient of $I_{\beta_i}(\nu^s\chi\otimes \St_{\GL_2})$ for $s>0$. These 4 representations can also expressed as 
\[\begin{array}{ll}
\chi_{\SO_4}=(\chi\circ \det\boxtimes 1_{\GL_2})|_{\SO_4}, & J_{\beta_1}(1/2,\chi\otimes \St_{\GL_2})=(\St_{\GL_2}\boxtimes \chi\circ \det)|_{\SO_4(F)},\\
 \St_{\SO_4}=( \St_{\GL_2}\boxtimes \St_{\GL_2})|_{\SO_4(F)}, &J_{\beta_2}(1/2,\chi\otimes \St_{\GL_2})=(\chi\circ \det\boxtimes \St_{\GL_2})|_{\SO_4}.
\end{array}\]
The local Langlands parameters of these 4 representations are given, in order, by
\begin{align*}
\phi_{\ref{ssec:SO4}.3a}(w,x)&= {\hat t}(\varphi(w) |w|, \varphi(w)),
\\
\phi_{\ref{ssec:SO4}.3b}(w,x)&={\hat t}(\varphi(w)|w|^{1/2}, \varphi(w)|w|^{1/2}) \iota_{\wh \beta_1}(x)
,\\
\phi_{\ref{ssec:SO4}.3c}(w,x)&={\hat t}(\varphi(w)|w|^{1/2}, \varphi(w)|w|^{-1/2})\iota_{\wh \beta_2}(x)
,\\
\phi_{\ref{ssec:SO4}.3d}(w,x)&= {\hat t}(\varphi(w),\varphi(w))\iota_{\wh \beta_1}(x)\iota_{\wh \beta_2}(x)
.
\end{align*}
Here $\varphi:W_F\to \C^\times$ is the dual character of the quadratic character $\chi : F^\times \to \C^\times$. 
Let $\lambda_{\ref{ssec:SO4}.3} : W_F \to \Lgroup{\SO_4}$ be the infinitesimal parameter for any of these Langlands parameters; thus,
\[
\lambda_{\ref{ssec:SO4}.3}(w) ={\hat t}(\varphi(w)|w|, \varphi(w)).
\]
The moduli space of Langlands parameters with infinitesimal parameter $\lambda_{\ref{ssec:SO4}.3}$ is the direct sum $\mathfrak{u}_{\hat\beta_1} \oplus \mathfrak{u}_{\hat\beta_2}$ of the root space for the simple roots $\hat\beta_1$ and $\hat\beta_2$ equipped with the group action of $\dualgroup{T}$ given by $t\cdot (x_1,x_2) \ceq (\hat\beta_1(t)x_1, \hat\beta_2(t)x_2)$. 
As a prehomogeneous vector space, $V_{\lambda_{\ref{ssec:SO4}.3}} = \mathbb{A}^2$ with this $H_{\lambda_{\ref{ssec:SO4}.3}} = \GL_1^2$ action is equivalent to \ref{geocase-toric} for $n=2$.
\end{enumerate}

\subsubsection*{ABV-packets}

Table~\ref{table:LA-packetsD} presents the ABV-packets for all unipotent representations of $\SO_4$ and its pure inner form $\SO_4^\delta$, each as a union of a pure $L$-packet and its coronal representations. 

We now sketch how the calculations in Table~\ref{table:coefficients:SO4} are made; the arguments are very similar to those in Section~\ref{ssec:PGL3} so we will be brief here.
The bijection between $\Pi^\mathrm{pure}_{\lambda_{\ref{ssec:SO4}.3}}(\SO_4/F)$ and the six simple perverse sheaves in $\Perv_{\dualgroup{T}}(\mathfrak{u}_{\hat\beta_1} \oplus \mathfrak{u}_{\hat\beta_2})$ is given by the following table.
\[
\begin{array}{r|cccc}
&  C_{\ref{ssec:SO4}.3a} & C_{\ref{ssec:SO4}.3b} & C_{\ref{ssec:SO4}.3c} & C_{\ref{ssec:SO4}.3d} \\
\hline
\chi_{\SO_4(F)}  & \IC(\1_{C_{\ref{ssec:SO4}.3a}}) &&& \\
J_{\beta_1}(1/2, \chi\otimes \St_{\GL_2})  &&\IC(\1_{C_{\ref{ssec:SO4}.3b}}) && \\
J_{\beta_2}(1/2, \chi\otimes \St_{\GL_2}) &&&\IC(\1_{C_{\ref{ssec:SO4}.3c}}) & \\
\chi_{\SO_4(F)}\otimes\St_{\SO_4} &&&&\IC(\1_{C_{\ref{ssec:SO4}.3d}})  \\
\chi_{\SO_4^{\delta}(F)} &&&& \IC(\mathcal{D}_{C_{\ref{ssec:SO4}.3d}})  
\end{array}
\]
where $\mathcal{D}$ is the character of $A_{\phi_{\ref{ssec:SO4}.3}} = \{\pm 1\}$ corresponding to the pure inner form $\delta\in Z^1(F,\SO_4)$.


To illustrate the calculations that this table summarizes, consider the case $\phi_{\ref{ssec:SO4}.3}$. The component group of $\phi_{\ref{ssec:SO4}.3}$ is $\{\pm 1\}$ and thus the pure $L$-packet of $\phi_{\ref{ssec:SO4}.3}$ has 2 elements:
$$\Pi_{\phi_{\ref{ssec:SO4}.3}}^{\mathrm{pure}}(\SO_4)=\{\St_{\SO_4}(\chi), \chi_{\SO_4^\delta(F)}  \}.$$
where $\chi_{\SO_4^\delta(F)}=1$ if $\chi=1$ and $\chi_{\SO_4^\delta(F)}$ is the unique unramified quadratic character of ${\SO_4^\delta(F)}$ if $\chi\ne 1$.
Table~\ref{table:coefficients:SO4} presents the ABV-packet coefficients promised in the Introduction as it pertains to the Langlands parameter $\phi_{\ref{ssec:SO4}.3}$.
From Table~\ref{table:coefficients:SO4} we read off the ABV-packets in Case~\ref{ssec:SO4}.3; the result appears in Table~\ref{table:LA-packetsD}.

\begin{table}[htp]
\caption{
ABV-packet coefficients $\langle \ , \ \rangle $ for unipotent representations of $\SO_4$ with infinitesimal parameter $\lambda_{\ref{ssec:SO4}.3}$.
The first four rows refer to representations of $\SO_4(F)$ while the fifth row refers to a representation of the inner form $\SO_4^\delta(F)$.
}
\begin{center}
$
\begin{array}{|c||c|c|c|c|}
\hline
\Pi^\mathrm{pure}_{\lambda_{\ref{ssec:SO4}.3}}(\SO_4/F) & \widehat{A^\ABV_{\phi_{\ref{ssec:SO4}.3a}}}  & \widehat{A^\ABV_{\phi_{\ref{ssec:SO4}.3b}}} & \widehat{A^\ABV_{\phi_{\ref{ssec:SO4}.3c}}}  & \widehat{A^\ABV_{\phi_{\ref{ssec:SO4}.3d}}}  \\
\hline\hline
 \chi_{\SO_4(F)}  & \1 & 0 & 0 & 0  \\
J_{\beta_1}(1/2, \chi\otimes \St_{\GL_2})  & 0 & \1 & 0 & 0  \\
J_{\beta_2}(1/2, \chi\otimes \St_{\GL_2})  & 0 & 0 & \1 & 0  \\
\St_{\SO_4}(\chi) & 0 & 0 & 0 & \1  \\
\chi_{\SO_4^{\delta}(F)} & \delta   & \delta & \delta & \delta \\
\hline
\end{array}
$
\end{center}
\label{table:coefficients:SO4}
\end{table}%

Properties~\ref{G2:LLC}, \ref{G2:open}, \ref{G2:tempered} and \ref{G2:spherical} are elementary for $\SO_4(F)$ and its pure inner forms. 
Proposition~\ref{G2:stable} is interesting for $\SO_4(F)$ and its pure inner forms.
\begin{proposition}
The distributions $\Theta^{\SO_4}_{\phi}$ (resp. $\Theta^{\SO_4^\delta}_{\phi}$) form a basis for the space of invariant distributions spanned by characters of unipotent irreducible admissible representations $\pi$ of $\SO_4(F)$ (resp. $\SO_4^\delta(F)$).
The distributions $\Theta^{\SO_4}_{\phi}$ and $\Theta^{\SO_4^\delta}_{\phi}$ are stable.
\end{proposition}

\begin{proof}
In all cases except $\phi= \phi_{\ref{ssec:SO4}.0'}$, $\Pi_\phi(\SO_4(F))$ and $\Pi_\phi(\SO^\delta_4(F))$ are singletons (or empty) and $\Theta^{\SO_4}_{\phi} = \Theta_\pi$. 
To see that these are all stable distributions, we may argue as follows. 

In Case~\ref{ssec:SO4}.0, the representations $I^{\SO_4}(\chi_1\otimes \chi_2)$ are stable because they are standard modules, stably induced from representations of $T(F)$.

In Case~\ref{ssec:SO4}.1 we have 
$\displaystyle
\Theta_{I^{\SO_4}(\chi\nu^{1/2}\otimes \chi\nu^{-1/2})} = \Theta_{I_{P_1}(\chi\otimes \St_{\GL_2})} + \Theta_{I_{P_1}(\chi\circ \det)}.
$\\
Since the distribution $\Theta_{I^{\SO_4}(\chi\nu^{1/2}\otimes \chi\nu^{-1/2})}$ is stable (by an argument as above) and $\Theta_{I_{P_1}(\chi\circ \det)}$ is stable by hand, $\Theta_{I_{P_1}(\chi \circ \St_{\GL_2})}$ is also stable.
For Case~\ref{ssec:SO4}.2 argue as in Case~\ref{ssec:SO4}.1.
In Case~\ref{ssec:SO4}.3 we have 
\[
\Theta_{I^{\SO_4}(\chi\nu\otimes \chi)}=\Theta_{J_{\emptyset}(\chi\nu\otimes \chi)}+\Theta_{J_{P_1}(1/2, \chi\otimes \St_{\GL_2})}+\Theta_{J_{P_2}(1/2, \chi\otimes \St_{\GL_2})} + \Theta_{\St_{\SO_4}(\chi)}.
\]
Since $I^{\SO_4}(\chi\nu\otimes \chi)$ is stable (argue as in Case~\ref{ssec:SO4}.0) and $J_{\emptyset}(\chi\nu\otimes \chi)$, $J_{P_1}(1/2, \chi\otimes \St_{\GL_2})$, and $J_{P_2}(1/2, \chi\otimes \St_{\GL_2})$ are stable (argue as in Case~\ref{ssec:SO4}.1) it follows that $\Theta_{\St_{\SO_4}(\chi)}$ is stable.
Finally, $\chi_{\SO_4^\delta(F)} $ is stable because $\1_{\SO_4^\delta(F)}$ is stable.
Case $\phi= \phi_{\ref{ssec:SO4}.0'}$ is a well known: $\Pi_{\phi_{\ref{ssec:SO4}.0'}}(\SO_4(F)) = \{ \pi_4, \pi'_4\}$ where $\pi_4$ and $\pi_4'$ are the two irreducible representations appearing in the induced representation $I^{\SO_4}(\chi_2\otimes 1)$ and the stable distribution attached to this $L$-packet is $\Theta^{\SO_4}_{\phi_{\ref{ssec:SO4}.0'}} = \Theta_{\pi_4} + \Theta_{\pi'_4}$.
\end{proof}

\subsection{Unipotent representations of $\PGL_3$ and its pure inner forms}\label{ssec:PGL3}
  
In this section we find the ABV-packets and ABV-packet coefficients for all unipotent representations of $\PGL_3$ and its pure inner forms.

As in Section~\ref{ssec:GL2}, the category of unipotent representations of $\PGL_3(F)$ is precisely the category of unramified principal series representations of $\PGL_3(F)$.
In this section we partition the set $\Pi(\PGL_3(F))_\text{unip}$ of unipotent representations of $\PGL_3(F)$ into $L$-packets and also describe the corresponding $L$-parameters.
We do the same for the non-split inner form $D^\times/F^\times$, where $D$ is a central division algebra of degree $3$ over $F$.
The $p$-adic group $\PGL_3$ has three pure inner forms but only two inner forms: $H^1(F,\PGL_3)$ has cardinality while its image in
$
H^1(F,\PGL_3) \to H^1(F,\Aut(\PGL_3))
$
has cardinality $2$. We remark that this function of pointed sets is not surjective.
Let $\delta,\delta' \in Z^1(F,\PGL_3)$ be representatives for the two non-trivial classes in $H^1(F,\PGL_3)$. 
Let $\PGL_3^\delta, \PGL_3^{\delta'}$ be the inner forms of $\PGL_3$ attached to these cocycles; then $\PGL_3^\delta\iso \PGL_3^{\delta'}$ as algebraic groups over $F$, although $[\delta]\ne [\delta']$ in $H^1(F,\PGL_3)$.
We can take $\PGL_3^\delta(F)= D^\times/F^\times$ and $\PGL_3^{\delta'}(F)=(D')^\times/F^\times$, where $D'$ is the opposite algebra of $D$.

\subsubsection*{Langlands correspondence}

\begin{table}[htp]
\caption{Pure Arthur packets for unipotent representations of $\PGL_3(F)$ and its pure inner forms $\PGL_3^\delta$ and $\PGL_3^{\delta'}$: each row gives a pure $L$-packet for an $L$-parameter and its coronal representations and therefore determines a pure Arthur packet. Notation is explained below. Cases, indicated in the left-hand column, gather representations together by infinitesimal parameter.}
\label{table:LA-packetsA}
\begin{center}
$
\begin{array}{|c|c|c| l |}
\hline
\text{$L$-parameter} & \text{Pure $L$-packet} & \text{Coronal} &\text{Parameter of } \\
\phi & \Pi^\text{pure}_\phi(\PGL_3/F) & \text{representations} & \text{Arthur type?} \\
\hline\hline
\phi_{\ref{ssec:PGL3}.0} 	&  \Ind_{B_3(F)}^{\GL_3(F)}(\chi_1\otimes \chi_2\otimes \chi_1^{-1}\chi_2^{-1})	& & \text{unitary $\chi_1$, $\chi_2$}\\
\hline\hline
\phi_{\ref{ssec:PGL3}.1a}	& I_{\alpha_1}((\chi \circ \det) \otimes \chi^{-2})  &  & \text{unitary $\chi$} \\
\phi_{\ref{ssec:PGL3}.1b} &  I_{\alpha_1}(\chi \St_{\GL_2}\otimes \chi^{-2}) & & \text{unitary $\chi$} \\
\hline\hline
\phi_{\ref{ssec:PGL3}.2a} &  I_{\alpha_1}((\chi \nu^{1/6}\circ \det)\otimes \chi\nu^{-1/3})  &  &\text{no}  \\
\phi_{\ref{ssec:PGL3}.2b}	&    I_{\alpha_1} (\chi \nu^{1/6} \St_{\GL_2}\otimes \chi\nu^{-1/3})  &  &\text{no} \\
\hline
\phi_{\ref{ssec:PGL3}.3a} &  \chi_{\PGL_3}  & \chi_{\PGL_3^\delta(F)},\ \chi_{\PGL_3^{\delta'}(F)}   &\text{yes}\\
\phi_{\ref{ssec:PGL3}.3b}	&  \chi\otimes J_{\alpha_1}(\nu\otimes \nu^{-1/2}\St_{\GL_2})  &  \chi_{\PGL_3^\delta(F)},\ \chi_{\PGL_3^{\delta'}(F)}   &\text{no} \\
\phi_{\ref{ssec:PGL3}.3c}	& \chi\otimes J_{\alpha_1}(\nu^{1/2}\St_{\GL_2}\otimes \nu^{-1})   & \chi_{\PGL_3^\delta(F)},\ \chi_{\PGL_3^{\delta'}(F)}    &\text{no}  \\
\phi_{\ref{ssec:PGL3}.3d}	&  \chi\otimes \St_{\PGL_3}, \  \chi_{\PGL_3^\delta(F)}  ,\ \chi_{\PGL_3^{\delta'}(F)}  &      &\text{yes}\\
\hline
\end{array}
$
\end{center}
\end{table}

Let $T$ be the standard maximal torus in $\PGL_3$ and let $B_3$ be the standard Borel subgroup of $\GL_3$.
We denote the matching positive roots of $\GL_3$ by $\alpha_1$ and $\alpha_2$. Let $P_{\alpha_i}$ be the standard maximal parabolic subgroup of $\GL_3$ such that the root space $\alpha_i$ is in the Levi of $P_{\alpha_i}$.
Since $\PGL_3(F)=\GL_3(F)/F^\times$, a representation of $\PGL_3(F)$ is just a representation of $\GL_3(F)$ with trivial central character. 
The cases treated below group representations together by infinitesimal parameter.

\begin{enumerate}
\item[\ref{ssec:PGL3}.0]
Let $\chi_1,\chi_2$ are characters of $F^\times$ such that $\chi_1\chi_2^{-1}$, $\chi_1^2\chi_2$ and $\chi_1\chi_2^2$ are not equal to $\nu$ and not equal to $\nu^{-1}$. Then the induced representation 
\[
\Ind_{B_3(F)}^{\GL_3(F)}(\chi_1\otimes \chi_2\otimes \chi_1^{-1}\chi_2^{-1})
\]
is irreducible.
Note that this representation has trivial central character and thus can be viewed as a representation of $\PGL_3(F)$. We denote this representation of $\PGL_3(F)$ by $I^{\PGL_3}(\chi_1\otimes \chi_2)$. The corresponding Langlands parameter for $I^{\PGL_3}(\chi_1\otimes \chi_2)$ is given by 
\[
\phi_{\ref{ssec:PGL3}.0}(w,x)=\bpm \varphi_1(w) &0 &0 \\ 0 &\varphi_2(w)&0 \\0 &0 &\varphi_1(w)^{-1}\varphi_2(w)^{-1} \epm\in \SL_3(\C),
\]
where $\varphi_i:W_F\to \C^\times$ corresponds to $\chi_i : F^\times \to \C^\times$ under class field theory.

\item[\ref{ssec:PGL3}.1]  Let $\chi$ be a character of $F^\times$ with $\chi^3\ne \nu^{\pm 1/2},\nu^{\pm 3/2}$. We consider the representation 
\[
\Ind_{B_3(F)}^{\GL_3(F)}(\chi\nu^{1/2}\otimes \chi\nu^{-1/2}\otimes \chi^{-2})
\]
of $\GL_3(F)$, which is also viewed as a representation of $\PGL_3(F)$. This representation has length 2, and its two irreducible components are $I_{\alpha_1}(\chi \St_{\GL_2}\otimes \chi^{-2})$ and $I_{\alpha_1}(\chi\circ \det\otimes \chi^{-2})$, which, adapting the notation above, are written as 
\[
I^{\PGL_3}_{\alpha_1}(\chi \St_{\GL_2}),\qquad \text{and}\qquad I^{\PGL_3}_{\alpha_1}(\chi\circ \det).
\]
Here $I^{\PGL_3}_{\alpha_1}$ stands for $\Ind_{P_{\alpha_1}}^{\GL_3}$, and a representation of $\GL_3(F)$ is viewed as a representation of $\PGL_3(F)$ if it has trivial central character. The Langlands parameter of $ I_{\alpha_1}(\chi\circ \det\otimes \chi^{-2}) $ is given by 
\[
\phi_{\ref{ssec:PGL3}.1.a}(w,x)= \bpm \varphi(w)|w|^{1/2}&0 &0 \\ 0 &\varphi(w)|w|^{-1/2}&0 \\ 0 &0 &\varphi(w)^{-2} \epm\in \SL_3(\C),
\]
where $\varphi:W_F\ra \C^\times$ is the dual character of $\chi$. The Langlands parameter for $I_{\alpha_1}(\chi \St_{\GL_2}\otimes \chi^{-2}) $ is given by 
\[
\phi_{\ref{ssec:PGL3}.1b}\left(w,\left(\begin{smallmatrix} a & b \\ c & d \end{smallmatrix}\right)\right)= 
\begin{pmatrix}
\varphi(w)a & \varphi(w)b & 0 \\
\varphi(w)c & \varphi(w)d & 0 \\ 
0 & 0 & \varphi(w)^{-2} 
\end{pmatrix}
= \bpm \varphi(w)x&\\ &\varphi(w)^{-2} \epm\in  \SL_3(\C).
\]
Note that the infinitesimal parameter for these representations is
\[
\lambda_{\ref{ssec:PGL3}.1a}(w)= \bpm \varphi(w)|w|^{1/2}&0 &0 \\ 0 &\varphi(w)|w|^{-1/2}&0 \\ 0 &0 &\varphi(w)^{-2} \epm.
\]

\item[\ref{ssec:PGL3}.2] 
 Let $\chi$ be a character of $F^\times$ with $\chi^3=1$. Consider the representation 
 \[
 \Ind_{B_3(F)}^{\GL_3(F)}(\chi\nu^{2/3}\otimes \chi \nu^{-1/3}\otimes \chi\nu^{-1/3}),
 \]
 which is also viewed as a representation of $\PGL_3(F)$. This representation has length 2 and its two components are 
 \[
 I_{\alpha_1}(\chi\nu^{1/6}\St_{\GL_2}\otimes \chi\nu^{-1/3})
 \quad\text{and}\quad
 I_{\alpha_1}(\chi\nu^{1/6}\circ \det\otimes \chi\nu^{-1/3}).
\] 
The Langlands parameter of $I_{\alpha_1}(\chi\nu^{1/6}\circ \det\otimes \chi\nu^{-1/3})$ is given by 
\[
\phi_{\ref{ssec:PGL3}.2a}(w,x)= \bpm \varphi(w)|w|^{2/3} & 0 & 0 \\ 0 &\varphi(w)|w|^{-1/3} & 0 \\ 0 & 0 & \varphi(w)|w|^{-1/3} \epm\in \SL_3(\C),
\]
and the Langlands parameter of $I_{\alpha_1}(\chi\nu^{1/6}\St_{\GL_2}\otimes \chi\nu^{-1/3})$ is given by 
\[ 
\phi_{\ref{ssec:PGL3}.2b}(w,x)= \bpm \varphi(w)|w|^{1/6}x &\\ & \varphi(w)|w|^{-1/3} \epm\in \SL_3(\C).
\]

\item[\ref{ssec:PGL3}.3] Let $\chi$ be a character of $F^\times$ with $\chi^3=1$. Consider the induced representation 
\[
\Ind_{B_3(F)}^{\GL_3(F)}(\chi\nu\otimes \chi\otimes \chi\nu^{-1})
\]
of $\GL_3(F)$ (and also of $\PGL_3(F)$). This representation has length 4 and its components are given by 
\[
\begin{array}{cc}
\chi_{\PGL_3(F)}, & \chi\otimes J_{\alpha_2}(\nu\otimes \nu^{-1/2}\St_{\GL_2}), \\
\chi\otimes J_{\alpha_1}(\nu^{1/2}\St_{\GL_2}\otimes \nu^{-1}), &  \chi\otimes\St_{\PGL_3},
\end{array}
\]
where $\chi_{\PGL_3(F)}=\chi\circ\det$, $J_{\alpha_1}( \nu^{1/2}\St_{\GL_2}\otimes \nu^{-1})$ is the unique quotient of $I_{\alpha_1}(\nu^{1/2}\St_{\GL_2}\otimes \nu^{-1})$ and $ J_{\alpha_2}(\nu\otimes \nu^{-1/2}\St_{\GL_2}$ is defined similarly. The local Langlands parameters of these 4 representations are given, in order,  by 
\begin{align*}
\phi_{\ref{ssec:PGL3}.3a}(w,x)&= \varphi(w) \bpm |w| &&\\ &1&\\ &&|w|^{-1} \epm \in \SL_3(\C),\\
\phi_{\ref{ssec:PGL3}.3b}(w,x)&= \varphi(w) \bpm |w| &\\ &|w|^{-1/2}x \epm\in \SL_3(\C),\\
\phi_{\ref{ssec:PGL3}.3c}(w,x)&= \varphi(w) \bpm |w|^{1/2}x &\\ &|w|^{-1} \epm\in \SL_3(\C),\\
\phi_{\ref{ssec:PGL3}.3d}(w,x)&= \varphi(w) \Sym^2(x)\in \SL_3(\C),
\end{align*}
where $\varphi : W_F\to C^\times$ is Langlands parameter for the quadratic character $\chi$ and $\Sym^2$ is the symmetric square representation of $\SL_2(\C)$. 
\end{enumerate}

This completes the description of the Langlands correspondence for unipotent representations of $\PGL_3(F)$.

\subsubsection*{ABV-packets}

Table~\ref{table:LA-packetsA} presents the $L$-packets for all unipotent representations of $\PGL_3(F)$.
Returning to pure $L$-packets, we only explain the case {\ref{ssec:PGL3}.3} since the others proceed by similar, simpler arguments.
Denote the last Langlands parameter in Case~\ref{ssec:PGL3}.3 by ${\phi}_{\ref{ssec:PGL3}.3d}$. This Langlands parameter has component group $\mu_3$ so its pure $L$-packet contains 3 admissible irreducible representations:
$$\Pi_{{\phi}_{\ref{ssec:PGL3}.3d}}^{\mathrm{pure}}(\PGL_3/F) =\{\chi\otimes \St_{\PGL_3}, \chi_{\PGL_3^\delta(F)}, \chi_{\PGL_3^{\delta'}(F)} \},$$
where $\chi_{\PGL_3^{\delta}(F)}$ (resp. $\chi_{\PGL_3^{\delta'}(F)}$) is the character of $\PGL_3^{\delta}(F)$ (resp. $\PGL_3^{\delta'}(F)$) obtained by composing $\chi$ with the reduced norm.

Table~\ref{table:coefficients:PGL3} presents the ABV-packet coefficients for the Langlands parameter $\phi_{\ref{ssec:PGL3}.3}$.
From Table~\ref{table:coefficients:PGL3} we read off the ABV-packets in Case~\ref{ssec:PGL3}.3; the result appears in Table~\ref{table:LA-packetsA}, where $\mathcal{D}$ (resp. $\mathcal{D}'$) is the character of $A_{\phi_{\ref{ssec:PGL3}.3d}}$ corresponding to the pure inner form $\delta$ (resp. $\delta'$) in $Z^1(F,\PGL_3)$.

\begin{table}[tp]
\caption{ABV-packet coefficients  $\langle \ , \ \rangle $ for unipotent representations of $\PGL_3$ and its pure inner forms with infinitesimal parameter $\lambda_{\ref{ssec:PGL3}.3}$.
The first four rows refer to representations of $\PGL_3(F)$, while the fifth and sixth rows are refer to the pure inner forms $\delta$ and $\delta'$, respectively. 
Below, we identify $\delta$ (resp. $\delta'$) with the character of $A_{\phi_{\ref{ssec:PGL3}.3d}}$ corresponding to the pure inner form $\delta$ (resp. $\delta'$) in $Z^1(F,\PGL_3)$.
All microlocal fundamental groups are order $3$
}
\begin{center}
$
\begin{array}{|c||cccc|}
\hline
\text{Irreps with inf parameter $\lambda_{\ref{ssec:PGL3}.3}$} & \widehat{A^\ABV_{\phi_{\ref{ssec:PGL3}.3a}}}  & \widehat{A^\ABV_{\phi_{\ref{ssec:PGL3}.3b}}}   & \widehat{A^\ABV_{\phi_{\ref{ssec:PGL3}.3c}}}  & \widehat{A^\ABV_{\phi_{\ref{ssec:PGL3}.3d}}}  \\
\hline\hline
\chi_{\PGL_3(F)}  & \1 & 0 & 0 & 0 \\
\chi\otimes J_{\alpha_1}(\nu\otimes \nu^{-1/2}\St_{\GL_2})  & 0 & \1 & 0 & 0 \\
\chi\otimes J_{\alpha_1}(\nu^{1/2}\St_{\GL_2}\otimes \nu^{-1}) & 0 & 0 & \1 & 0 \\
\chi\otimes \St_{\PGL_3} & 0 & 0 & 0 & \1 \\
\chi_{\PGL_3^{\delta}(F)} & \delta  & \delta & \delta & \delta \\
\chi_{\PGL_3^{\delta'}(F)} & \delta'  & \delta' & \delta' & \delta' \\
\hline
\end{array}
$
\end{center}
\label{table:coefficients:PGL3}
\end{table}%

We now explain the calculations behind Table~\ref{table:coefficients:PGL3}, as a special case of the techniques of \cite{CFMMX}; see also \cite{CFZ:cubics}.
Let $\lambda_{\ref{ssec:PGL3}.3}: W_F \to \SL_3(\C)$ be the infinitesimal parameter of $\phi_{\ref{ssec:PGL3}.3d}$:
\[
\lambda_{\ref{ssec:PGL3}.3}(w) 
\ceq 
\phi_{\ref{ssec:PGL3}.3d} \left(w,\diag( \abs{w}^{1/2}, \abs{w}^{-1/2})\right)
= 
\diag(\abs{w},1, \abs{w}^{-1}).
\]
The moduli space of Langlands parameters with infinitesimal parameter $\lambda_{\ref{ssec:PGL3}.3}$ is the direct sum $\mathfrak{u}_{\hat\alpha_1} \oplus \mathfrak{u}_{\hat\alpha_2}$ of the root space for the simple roots $\hat\alpha_1$ and $\hat\alpha_2$, equipped with the group action of $\dualgroup{T}$ given by $t\cdot (x_1,x_2) \ceq (\hat\alpha_1(t)x_1, \hat\alpha_2(t)x_2)$.

Vogan's version of the Langlands correspondence establishes a bijection between simple objects in the category $\Perv_{\dualgroup{T}}(\mathfrak{u}_{\hat\alpha_1} \oplus \mathfrak{u}_{\hat\alpha_2})$ of $\dualgroup{T}$-equivariant perverse sheaves on $\mathfrak{u}_{\hat\alpha_1} \oplus \mathfrak{u}_{\hat\alpha_2}$ and irreducible admissible representations of $\PGL_3$ and its pure inner forms with infinitesimal parameter $\lambda_{\ref{ssec:PGL3}.3}$.
We have already seen there are six of the latter -- the four irreducible subquotients of $
\Ind_{B_3(F)}^{\GL_3(F)}(\chi\nu\otimes \chi\otimes \mu\nu^{-1})$ together with the characters $\chi_{\PGL_3^{\delta}(F)}$ and $\chi_{\PGL_3^{\delta'}(F)}$. Let us enumerate the corresponding simple objects in $\Perv_{\dualgroup{T}}(\mathfrak{u}_{\hat\alpha_1} \oplus \mathfrak{u}_{\hat\alpha_2})$.
Let $C_{\ref{ssec:PGL3}.3a}$ be the trivial orbit in $V_{\lambda_{\ref{ssec:PGL3}.3}}$, 
let $C_{\ref{ssec:PGL3}.3b}$ be the $\dualgroup{T}$-orbit of $X_{\hat\alpha_1}$, which is the point in the moduli space for the Langlands parameter $\phi_1$ for $\chi\otimes J_{\alpha_1}(\nu\otimes \nu^{-1/2}\St_{\GL_2})$; let $C_{\ref{ssec:PGL3}.3c}$ be the $\dualgroup{T}$-orbit of $X_{\hat\alpha_2}$, which is the point in the moduli space for the Langlands parameter $\phi_2$ for $\chi\otimes J_{\alpha_1}(\nu^{1/2}\St_{\GL_2}\otimes \nu^{-1})$.
Finally, let $C_{\ref{ssec:PGL3}.3d}$ be the $\dualgroup{T}$-orbit of $X_{\hat\alpha_1}+X_{\hat\alpha_2}$; this orbit of the Langlands parameter $\phi_{\ref{ssec:PGL3}.3d}$ for $\mu\otimes \St_{\PGL_3}$.
The equivariant fundamental groups of these orbits are trivial with the exception of the open orbit $C_{\ref{ssec:PGL3}.3d}$, which has equivariant fundamental group 
\[
A_{\phi_{\ref{ssec:PGL3}.3d}} \ceq \pi_1^{\dualgroup{T}}(C_{\ref{ssec:PGL3}.3},X_{\hat\alpha_1}+X_{\hat\alpha_2}) = \pi_0(Z_{\dualgroup{T}}(X_{\hat\alpha_1}+X_{\hat\alpha_2})) = \{ 1, \theta_3, \theta_3^2\}.
\]
The bijection between $\Pi^\mathrm{pure}_{\lambda_{\ref{ssec:PGL3}.3}}(\PGL_3/F)$ and the six simple perverse sheaves in $\Perv_{\dualgroup{T}}(\mathfrak{u}_{\hat\alpha_1} \oplus \mathfrak{u}_{\hat\alpha_2})$ are given by the following table.
\[
\begin{array}{r|cccc}
& C_{\ref{ssec:PGL3}.3a} & C_{\ref{ssec:PGL3}.3b} & C_{\ref{ssec:PGL3}.3c} & C_{\ref{ssec:PGL3}.3d} \\
\hline
\chi_{\PGL_3(F)}  & \IC(\1_{C_{\ref{ssec:PGL3}.3a}}) &&& \\
 \chi\otimes J_{\alpha_1}(\nu\otimes \nu^{-1/2}\St_{\GL_2})  && \IC(\1_{C_{\ref{ssec:PGL3}.3b}}) && \\
 \chi\otimes J_{\alpha_1}(\nu^{1/2}\St_{\GL_2}\otimes \nu^{-1}) &&& \IC(\1_{C_{\ref{ssec:PGL3}.3c}}) & \\
 \chi\otimes \St_{\PGL_3} &&&& \IC(\1_{C_{\ref{ssec:PGL3}.3d}})  \\
 \chi_{\PGL_3^{\delta}(F)} &&&& \IC(\mathcal{D}_{C_{\ref{ssec:PGL3}.3d}})  \\
 \chi_{\PGL_3^{\delta'}(F)} &&&& \IC(\mathcal{D}'_{C_{\ref{ssec:PGL3}.3d}}),  
\end{array}
\]
where $\mathcal{D}$ (resp. $\mathcal{D}'$) is the local system 
corresponding to the choice of primitive cube root of unity $\delta$ (resp. $\delta'$) appearing in the pure inner form $\delta\in Z^1(F,\PGL_3)$.

Now, for each irreducible admissible $\pi$ above, and for each Langlands parameter $\phi$ with infinitesimal parameter $\lambda_{\ref{ssec:PGL3}.3}$, we compute $\NEvs_{C_\phi}\mathcal{P}(\pi)$, where $\mathcal{P}(\pi)$ is the corresponding simple object in $\Perv_{\dualgroup{T}}(\mathfrak{u}_{\hat\alpha_1} \oplus \mathfrak{u}_{\hat\alpha_2})$. The result is a local system in $\Lambda_{C_\phi}^\text{sreg} \ceq T^*_{C_\phi}(\mathfrak{u}_{\hat\alpha_1} \oplus \mathfrak{u}_{\hat\alpha_2})_\text{sreg}$ and this may be viewed as a representation of the equivariant fundamental group of $\Lambda_{C_\phi}^\text{sreg}$, which is $A^\ABV_\phi$. 
The results are summarized in Table~\ref{table:coefficients:PGL3}.

The properties in Theorems~\ref{thm:coefficients} and \ref{thm:stable} are elementary for $\PGL_3(F)$ and its pure inner forms as is \cite{CFMMX}*{Conjecture 1}.


\section{Geometric endoscopy}\label{sec:geoendo}


\subsection{Lifting Langlands parameters}\label{ssec:liftingparameters}

Let $(G,s,\xi)$ be an endoscopic triple for $G_2$ and let $\phi : W'_F \to \Lgroup{G}$ be an unramified Langlands parameter such that $s\in \mathcal{S}^\ABV_{\xi\circ\phi}$.
The Langlands parameter $\phi$ determines, or lifts to, a Langlands parameter $\xi\circ\phi : W'_F \to \Lgroup{G_2}$. Matching the cases that appear in our classification of Langlands parameters for $G$ and $G_2$ is a little complicated, because we have grouped parameters into families based on the geometry created by the corresponding infinitesimal parameters.
Consequently, it happens that different members of these families may lift to different families for $G_2$.
To illustrate this phenomenon, here we give the details for one of the six endoscopic groups for $G_2$.
Together with all other cases, this information is summarized in Table~\ref{table:lifting parameters}.

\begin{enumerate}

\item[\ref{endcase:A2}]
$G=\PGL_3$, $s= {\hat m}(\theta_3^2,\theta_3)$.
\begin{enumerate}
\item[(0)]
The Langlands parameter $\phi_{\ref{ssec:PGL3}.0}$ lifts to $\phi_{\ref{infcase:0}}$ if and only if  $\varphi_1(\Frob),\varphi_2(\Fr),\varphi_1(\Fr)\varphi_2(\Fr)\notin\{ q^{\pm 1}\}$; lifts to $\phi_{\ref{infcase:1-short}a}$ if and only if  one of $\varphi_1(\Frob),\varphi_2(\Fr),\varphi_1(\Fr)\varphi_2(\Fr)$ is in $\{q^{\pm 1}\}$; and lifts to $\phi_{\ref{infcase:5}a}$  if and only if  two of $\varphi_1(\Frob),\varphi_2(\Fr),\varphi_1(\Fr)\varphi_2(\Fr)$ are in $\{q^{\pm 1}\}$.

\item[(1)]
The Langlands parameter $\phi_{\ref{ssec:PGL3}.1a}$ (resp. $\phi_{\ref{ssec:PGL3}.1a}$) lifts to $\phi_{\ref{infcase:2-long}a}$ (resp. $\phi_{\ref{infcase:2-long}b}$)  if and only if  $\varphi(\Fr)\notin\{- q^{\pm 1/2},q^{\pm 3/2}\}$ ; $\phi_{\ref{infcase:4-D2}a}$ (resp. $\phi_{\ref{infcase:4-D2}b}$)  if and only if  $\varphi(\Fr)\in\{-q^{\pm 1/2}\}$; $\phi_{\ref{infcase:7-reg}a}$ (resp. $\phi_{\ref{infcase:7-reg}b}$)  if and only if  $\varphi(\Fr)\in \{q^{\pm3/2}\}$.

\item[(2)]
The Langlands parameter $\phi_{\ref{ssec:PGL3}.2a}$ (resp. $\phi_{\ref{ssec:PGL3}.2b}$) lifts to $\phi_{\ref{infcase:3}a}$ (resp. $\phi_{\ref{infcase:3}b}$).

\item[(3)] 
The Langlands parameter $\phi_{\ref{ssec:PGL3}.3a}$ (resp. $\phi_{\ref{ssec:PGL3}.3b}$, $\phi_{\ref{ssec:PGL3}.3c}$, $\phi_{\ref{ssec:PGL3}.3d}$) lifts to $\phi_{\ref{infcase:6-A2}a}$ (resp. $\phi_{\ref{infcase:6-A2}b}$, $\phi_{\ref{infcase:6-A2}b}$, $\phi_{\ref{infcase:6-A2}b}$) if and only if  $\varphi(\Fr)\ne 1$; and lifts to $\phi_{\ref{infcase:8-sub}a}$ (resp. $\phi_{\ref{infcase:8-sub}b}$, $\phi_{\ref{infcase:8-sub}c}$, $\phi_{\ref{infcase:8-sub}d}$) if and only if  $\varphi(\Fr)=1$.
\end{enumerate}

\end{enumerate}

If we turn the problem of lifting parameters on its head, our classification of scheme for Langlands parameters is more illuminating: instead of lifting parameters from endoscopic groups, begin with a Langlands parameter for $G_2$ and find all endoscopic triples $(G,s,\xi)$ such that the Langlands parameter factors through $\xi : \Lgroup{G} \to \Lgroup{G_2}$. This information appears in Table~\ref{table:AE}.

\begin{table}[tp]
\caption{Langlands parameters and their associated minimal endoscopic groups. In the column labeled "minimal endoscopic group(s)" we indicate, for each Langlands parameter $\phi$ of $G_2$, the minimal endoscopic groups $G$ such that the parameter factors $\phi$ through $\xi : \Lgroup{G} \to \Lgroup{G_2}$.
The column labeled "Arthur type" gives necessary and sufficient conditions for $\phi$ to be of Arthur type.}
\label{table:AE}
\begin{center}
$
\begin{array}[t]{| c | c | c | c  c | c |} 
\hline
\text{$L$-parameter}  & \text{Minimal} &\text{Arthur} & \multicolumn{ 2 }{ c | }{\text{Component}} & \\ 
 & \text{endoscopic} &\text{type} & \multicolumn{ 2 }{ c | }{\text{groups}} & \\ 
\phi & \text{group(s)} &\text{parameter} & A_\phi & A^\ABV_\phi & \dim(\phi) \\ 
\hline\hline
\phi_{\ref{infcase:0}} 	 & T & \text{$a_1=a_2=0$} & 1 & 1 & 0 \\
\hline\hline
\phi_{\ref{infcase:1-short}a} & T & \text{$a=1/2$} & 1 & 1 & 0\\
\phi_{\ref{infcase:1-short}b} &  \GL_2^{\gamma_2} & \text{$a=1/2$} & 1 & 1 & 1\\
\hline
\phi_{\ref{infcase:2-long}a} & T & \text{$a=1/2$} & 1 & 1 & 0 \\
\phi_{\ref{infcase:2-long}b} & \GL_2^{\gamma_1} & \text{$a=1/2$} & 1 & 1 & 1\\
\hline\hline
\phi_{\ref{infcase:3}.a}  &T & \text{no} & 1 & 1 & 0\\
\phi_{\ref{infcase:3}.b} &\GL_2^{\gamma_1} & \text{no} & 1 & 1 & 2\\
\hline
\phi_{\ref{infcase:4-D2}a}    &T  & \text{yes} & 1 & \langle \theta_2 \rangle & 0\\
\phi_{\ref{infcase:4-D2}b} 	& \GL_2^{\gamma_1} & \text{yes} & 1 & \langle \theta_2 \rangle & 1\\
\phi_{\ref{infcase:4-D2}c} 	& \GL_2^{\gamma_2} & \text{yes} & 1 & \langle \theta_2 \rangle & 1\\
\phi_{\ref{infcase:4-D2}d} 	& \SO_4 & \text{yes} & \langle \theta_2 \rangle & \langle \theta_2 \rangle & 2\\
\hline
\phi_{\ref{infcase:5}a} &T  & \text{no} & 1 & 1 & 0\\
\phi_{\ref{infcase:5}b} & \GL_2^{\gamma_2} & \text{no} & 1 & 1 & 2\\
\hline\hline
\phi_{\ref{infcase:6-A2}a}  & T  & \text{yes} & 1 & \langle \theta_3 \rangle & 0\\
\phi_{\ref{infcase:6-A2}b} 	& \GL_2^{\gamma_1} & \text{no} & 1 & \langle \theta_3 \rangle & 1\\
\phi_{\ref{infcase:6-A2}c} 	& \GL_2^{\gamma_1} & \text{no} & 1 & \langle \theta_3 \rangle & 1\\
\phi_{\ref{infcase:6-A2}d} 	& \PGL_3 & \text{yes} & \langle \theta_3 \rangle & \langle \theta_3 \rangle & 2\\
\hline
\phi_{\ref{infcase:7-reg}a}  &  T & \text{yes} & 1 & 1 & 0\\
\phi_{\ref{infcase:7-reg}b} & \GL_2^{\gamma_1} & \text{no} & 1 & 1 & 1\\
\phi_{\ref{infcase:7-reg}c} 	& \GL_2^{\gamma_2} & \text{no} & 1 & 1 & 1\\
\phi_{\ref{infcase:7-reg}d} 	& G_2 & \text{yes} & 1 & 1 & 2\\
\hline\hline
\phi_{\ref{infcase:8-sub}a}  & T  & \text{yes} & 1 & S_3  & 0\\
\phi_{\ref{infcase:8-sub}b} &  \GL_2^{\gamma_1} & \text{yes} & 1 & \langle \theta_2 \rangle & 2\\
\phi_{\ref{infcase:8-sub}c} 	&  \GL_2^{\gamma_2} & \text{yes} & 1 & \langle \theta_2 \rangle & 3\\
\phi_{\ref{infcase:8-sub}d}	 & \SO_4 , \PGL_3 & \text{yes} & S_3 & S_3 & 4\\
\hline
\end{array}
$
\end{center}
\end{table}%


\begin{table}[htp]
\caption{Summary of possible lifts for members of each family of Langlands parameters. The leftmost column gives endoscopic groups $G$ for $G_2$. 
The column "Endoscopic Langlands parameters" refers to Sections~\ref{ssec:T} through \ref{ssec:SO4}. 
The column  "Arthur-type" indicates those $\xi$-conormal parameters that, when unitary, are parameters of Arthur type; see Definition~\ref{def:xi-conormal} for this notion. 
The lifts appearing in the column "other regular" indicates all other $\xi$-conormal parameters. 
In these last two cases, the information in these two columns refers to the classification of unramified Langlands parameters for $G_2$ appearing in Section~\ref{ssec:infcases}. 
Finally, the column "irregular lifts" indicates all remaining lifts.}
\label{table:lifting parameters}
\begin{center}
$
\begin{array}{ l l | l r |  l}
\text{Endoscopic}&  \text{Endoscopic} &  \multicolumn{ 2 }{ c | }{ \text{$s$-regular lifts} } & \\
\text{group} &  \text{$L$-parameter} &  \text{Arthur} & \text{Other} & \\
 &   & \text{type} & \text{regular} & \text{Irregular lifts} \\
  \hline
\ref{endcase:torus} 
& \phi_{\ref{ssec:T}.0}  & \phi_{\ref{infcase:0}} &&
\phi_{\ref{infcase:1-short}a},
\phi_{\ref{infcase:2-long}a},
\phi_{\ref{infcase:3}a},
\phi_{\ref{infcase:4-D2}a},
\phi_{\ref{infcase:5}a},
\phi_{\ref{infcase:6-A2}a},
\phi_{\ref{infcase:7-reg}a},
\phi_{\ref{infcase:8-sub}a}\\
\hline

\ref{endcase:short} & \phi_{\ref{ssec:GL2}.0} 	 &  
\phi_{\ref{infcase:0}} &&
 \phi_{\ref{infcase:1-short}a},
 \phi_{\ref{infcase:2-long}a},
 \phi_{\ref{infcase:3}a},
 \phi_{\ref{infcase:4-D2}a},
 \phi_{\ref{infcase:5}a},
 \phi_{\ref{infcase:6-A2}a},
 \phi_{\ref{infcase:7-reg}a},
 \phi_{\ref{infcase:8-sub}a}\\

& \phi_{\ref{ssec:GL2}.1a} &  
\phi_{\ref{infcase:1-short}a}&
 \phi_{\ref{infcase:5}a}&
 \phi_{\ref{infcase:4-D2}a},
 \phi_{\ref{infcase:7-reg}a},
 \phi_{\ref{infcase:8-sub}a}\\

&\phi_{\ref{ssec:GL2}.1b}	&  
\phi_{\ref{infcase:1-short}b}&
 \phi_{\ref{infcase:5}b}&
 \phi_{\ref{infcase:4-D2}c},
 \phi_{\ref{infcase:7-reg}c},
 \phi_{\ref{infcase:8-sub}c}\\
\hline

\ref{endcase:long} & \phi_{\ref{ssec:GL2}.0} &  
\phi_{\ref{infcase:0}} &&
 \phi_{\ref{infcase:1-short}a},
 \phi_{\ref{infcase:2-long}a},
 \phi_{\ref{infcase:3}a},
 \phi_{\ref{infcase:4-D2}a},
 \phi_{\ref{infcase:5}a},
 \phi_{\ref{infcase:6-A2}a},
 \phi_{\ref{infcase:7-reg}a},
 \phi_{\ref{infcase:8-sub}a}\\

& \phi_{\ref{ssec:GL2}.1a}&  
 \phi_{\ref{infcase:2-long}a}&
 \phi_{\ref{infcase:3}a}&
 \phi_{\ref{infcase:4-D2}a},
 \phi_{\ref{infcase:6-A2}a},
 \phi_{\ref{infcase:7-reg}a},
 \phi_{\ref{infcase:8-sub}a}\\

&\phi_{\ref{ssec:GL2}.1b}	&  
 \phi_{\ref{infcase:2-long}b}&
 \phi_{\ref{infcase:3}b}&
 \phi_{\ref{infcase:4-D2}b},
 \phi_{\ref{infcase:6-A2}b},
 \phi_{\ref{infcase:7-reg}b},
 \phi_{\ref{infcase:8-sub}b}\\
\hline

\ref{endcase:D2} &\phi_{\ref{ssec:SO4}.0} 	& 
\phi_{\ref{infcase:0}} &&
 \phi_{\ref{infcase:1-short}a},
 \phi_{\ref{infcase:2-long}a},
 \phi_{\ref{infcase:3}a},
 \phi_{\ref{infcase:4-D2}a},
 \phi_{\ref{infcase:5}a},
 \phi_{\ref{infcase:6-A2}a},
 \phi_{\ref{infcase:7-reg}a},
 \phi_{\ref{infcase:8-sub}a}\\

&\phi_{\ref{ssec:SO4}.1a} 	&
\phi_{\ref{infcase:1-short}a} &
\phi_{\ref{infcase:5}a}&
\phi_{\ref{infcase:7-reg}a}\\

&\phi_{\ref{ssec:SO4}.1b}	&  
 \phi_{\ref{infcase:1-short}b} &
 \phi_{\ref{infcase:5}b}&
\phi_{\ref{infcase:7-reg}c}\\

&\phi_{\ref{ssec:SO4}.2a} 	& 
 \phi_{\ref{infcase:2-long}a}&
  \phi_{\ref{infcase:3}a}& 
 \phi_{\ref{infcase:6-A2}a},
\phi_{\ref{infcase:7-reg}a}\\

&\phi_{\ref{ssec:SO4}.2b} 	&  
\phi_{\ref{infcase:2-long}b} &
 \phi_{\ref{infcase:3}b}&
\phi_{\ref{infcase:6-A2}b},
\phi_{\ref{infcase:7-reg}b}\\

&\phi_{\ref{ssec:SO4}.3a} &  
\phi_{\ref{infcase:4-D2}a},
 \phi_{\ref{infcase:8-sub}a} &&\\

&\phi_{\ref{ssec:SO4}.3b} &   
\phi_{\ref{infcase:4-D2}b},
 \phi_{\ref{infcase:8-sub}b} &&\\

&\phi_{\ref{ssec:SO4}.3c} &   
\phi_{\ref{infcase:4-D2}c},
 \phi_{\ref{infcase:8-sub}c} &&\\

&\phi_{\ref{ssec:SO4}.3d} &  
\phi_{\ref{infcase:4-D2}d},
 \phi_{\ref{infcase:8-sub}d} &&\\
\hline

\ref{endcase:A2} & \phi_{\ref{ssec:PGL3}.0} 	& 
 \phi_{\ref{infcase:0}} &&
\phi_{\ref{infcase:1-short}a},
 \phi_{\ref{infcase:5}a}\\

& \phi_{\ref{ssec:PGL3}.1a}& 
\phi_{\ref{infcase:2-long}a}&&
 \phi_{\ref{infcase:4-D2}a},
 \phi_{\ref{infcase:7-reg}a}\\

& \phi_{\ref{ssec:PGL3}.1b}&
\phi_{\ref{infcase:2-long}b}&&
 \phi_{\ref{infcase:4-D2}b},
\phi_{\ref{infcase:7-reg}b} \\

&\phi_{\ref{ssec:PGL3}.2a}&&
 \phi_{\ref{infcase:3}a}&\\

&\phi_{\ref{ssec:PGL3}.2b}&&
 \phi_{\ref{infcase:3}b}&\\

&\phi_{\ref{ssec:PGL3}.3a}&  
\phi_{\ref{infcase:6-A2}a},
 \phi_{\ref{infcase:8-sub}a} &&\\

&\phi_{\ref{ssec:PGL3}.3b}&&	  
\phi_{\ref{infcase:6-A2}b}&
 \phi_{\ref{infcase:8-sub}b}\\

&\phi_{\ref{ssec:PGL3}.3c}&&
 \phi_{\ref{infcase:6-A2}c}&
 \phi_{\ref{infcase:8-sub}b}\\

&\phi_{\ref{ssec:PGL3}.3d}&	
\phi_{\ref{infcase:6-A2}d},
\phi_{\ref{infcase:8-sub}d} &&\\

\end{array}
$
\end{center}
\end{table}

\subsection{Regular conormal vectors and Arthur parameters}\label{ssec:xi-conormal}

Let $(G,s,\xi)$ be an endoscopic triple for $G_2$ and let $\phi : W'_F \to \Lgroup{G}$ be an unramified Langlands parameter.
Let $\lambda$ be the infinitesimal parameter for $\phi$; then $\xi\circ\lambda$ is the infinitesimal parameter for $\xi\circ\phi$.
Now consider the map $\xi : \Lgroup{G} \to \Lgroup{G_2}$ on dual groups and the induced homomorphism $\dualgroup{G} \to \dualgroup{G_2}$, also denoted by $\xi$ below. Pass to tangent spaces $d\xi : \dualgroup{\mathfrak{g}}\to \dualgroup{\mathfrak{g}_2}$ and restricted to Vogan varieties to define the embedding 
$
d\xi\vert_{V_{\lambda}} : V_{\lambda} \to V_{\xi \circ \lambda}
$
such that the subvariety $V_{\lambda}$ is the fixed-point set of the $(V_{\xi \circ \lambda})^s = V_{\lambda}$. This induces
$
d\xi\vert_{T^*(V_{\lambda})} : T^*(V_{\lambda}) \to T^*(V_{\xi \circ \lambda})
$
such that $\left(T^*(V_{\xi \circ \lambda})\right)^s = T^*(V_{\lambda})$; recall that we may view $T^*(V_{\lambda})$ as a subvariety of $\dualgroup{\mathfrak{g}}$. Likewise, this induces
a map of conormal bundles
$
d\xi\vert_{\Lambda_{\lambda}} : \Lambda_{\lambda} \to \Lambda_{\xi \circ \lambda}
$
such that $\Lambda_{\lambda} = \left(\Lambda_{\xi \circ \lambda}\right)^s$; again, we may view $\Lambda_{\lambda}$ as a subvariety of $\dualgroup{\mathfrak{g}}$.
Now let $x_\phi$ (resp.  $x_{\xi\circ\phi}$) be the point in the moduli space $V_{\lambda}$ (resp. $V_{\xi\circ\lambda}$) for $\phi : W'_F \to \Lgroup{G}$ (resp. for $\xi\circ\phi : W'_F \to \Lgroup{G_2}$) and let $C_\phi$ (resp. $C_{\xi\circ\phi}$) be its $H_{\lambda}$-orbit (resp. $H_{\xi\circ\lambda}$-orbit).
Now $d\xi\vert_{\Lambda_{\lambda}}$ restricts to an immersion
$
d\xi\vert_{\Lambda_{\lambda,C_\phi}}: \Lambda_{\lambda,C_\phi} \to \Lambda_{\xi \circ \lambda, C_{\xi \circ \phi}}.
$
Now consider the restriction
\[
\begin{tikzcd}
\Lambda_{\lambda,C_\phi} \arrow{rr}{d\xi\vert_{\Lambda_{\lambda,C_\phi}}} && \Lambda_{\xi \circ \lambda, C_{\xi \circ \phi}}\\
\Lambda_{\lambda,C_\phi}^\text{sreg} \arrow{u}  \arrow{urr}[swap]{d\xi\vert_{\Lambda_{\lambda,C_\phi}^\text{sreg}}} && 
\end{tikzcd}
\]

\begin{definition}\label{def:xi-conormal}
Let $(G,s,\xi)$ be an endoscopic triple for $G_2$ and let $\phi : W'_F \to \Lgroup{G}$ be an unramified Langlands parameter.
Let us say that a Langlands parameter $\phi : W'_F \to \Lgroup{G}$ is $\xi$-conormal, if the image of $d\xi\vert_{\Lambda_{\lambda,C_\phi}^\text{sreg}}$ is contained in $\Lambda_{\xi \circ \lambda, C_{\xi \circ \phi}}^\text{sreg}$.
If $\phi$ is $\xi$-conormal the map
\[
d\xi\vert_{\Lambda_{\lambda,C_\phi}^\text{sreg}} : {\Lambda_{\lambda,C_\phi}^\text{sreg}} \to \Lambda_{\xi \circ \lambda, C_{\xi \circ \phi}}^\text{sreg}
\]
induces a group homomorphism of
equivariant fundamental groups, denoted by 
\[
d\xi_\phi : A^\ABV_\phi \to A^\ABV_{\xi\circ\phi}.
\]
\end{definition}

\begin{proposition}\label{prop:AFPF}
Let $(G,s,\xi)$ be an endoscopic triple for $G_2$ and let $\phi : W'_F \to \Lgroup{G}$ be an unramified Langlands parameter. 
 If $\phi$ is of Arthur type then $\phi$ is $\xi$-conormal.
\end{proposition}

\begin{proof}
Let $\psi : W''_F\to \Lgroup{G}$ be the Arthur parameter such that $\phi(w,x) = \psi(w,x,d_w)$.  
Set $x\ceq d\psi(1,e,1)$ and $y\ceq d\psi(1,1,f)$.
 Then by \cite{CFMMX}*{Prop. 6.1.1} we have $(x,y)$ is a strongly regular conormal vector for $\phi$. 
 Consider $\xi\circ\psi : W''_F \to \Lgroup{G_2}$. Then again by \cite{CFMMX}*{Prop. 6.1.1} 
 we have $d\xi(x,y) = (d(\xi\circ \psi)(1,e,1),d(\xi\circ \psi)(1,1,f))$ is a strongly regular conormal vector for $\xi \circ \phi$ . 
\end{proof}

In Table~\ref{table:lifting parameters} we present a complete list of all endoscopic triples $(G,s,\xi)$ and all unramified Langlands parameters $\phi: W'_F \to \Lgroup{G_2}$ for which $\phi$ is $\xi$-conormal. We now explain this table.
The left-hand column of Table~\ref{table:lifting parameters} lists the endoscopic groups $G$ for $G_2$ as explained in Section~\ref{ssec:endoscopy}. Then, for each such $G$, we list the families of endoscopic Langlands parameters $\phi: W'_F \to \Lgroup{G}$ as they appeared in Sections~\ref{ssec:T} through \ref{ssec:SO4}. 
In Section~\ref{ssec:liftingparameters} we saw that for each $\phi$ in these families, there are multiple Langlands parameters for $G_2(F)$ to which $\phi$ may lift, depending on the properties of $\phi$. Each row of Table~\ref{table:lifting parameters} lists the lifts $\xi\circ\phi$ that arise as $\phi$ ranges through each family, partially ordered left to right along rows by the relative dimension $\dim(\xi\circ\phi)-\dim(\phi)$; note that if this relative dimension is $0$ then $\phi$ is trivially $\xi$-conormal.

\subsection{Statement of the fixed-point formula}

\begin{theorem}\label{thm:VFPF}
Let $(G,s,\xi)$ be an endoscopic triple for $G_2$ and let $\phi : W'_F \to \Lgroup{G}$ be an unramified Langlands parameter for $G$.
If $\phi$ is $\xi$-conormal then
\[
\trace_{s} \left(\NEvs_{\xi\circ\phi}[\dim(\xi\circ\phi)] \mathcal{P}\right) 
=
\trace_{s} \left(\NEvs_{\phi}[\dim(\phi)]\ \res \mathcal{P}\right) 
\]
for every $\mathcal{P}\in \Perv_{H_{\xi\circ\lambda}}(V_{\xi\circ\lambda})$. 
\end{theorem}

The proof of Theorem~\ref{thm:VFPF} will occupy the rest of this Section, so we explain the strategy here. 
We work through the list of endoscopic triples $(G,s,\xi)$ for $G_2$ and, in each case, we recall that classification of Langlands parameters $\phi : W'_F \to \Lgroup{G}$ in terms of infinitesimal parameters in Sections~\ref{ssec:T} through \ref{ssec:SO4}. 
We identify those Langlands parameters $\phi$ that are $\xi$-conormal, expanding on the information presented in Table~\ref{table:lifting parameters}. We then recall the classification of Vogan varieties $H_{\xi\circ\lambda}\times V_{\xi\circ \lambda} \to V_{\xi\circ \lambda}$ in terms of prehomogeneous vector spaces $H\times V\to V$ from Section~\ref{ssec:infcases} and we observe that same list of five prehomogeneous vector spaces capture all instances of the  Vogan varieties $H_{\lambda}\times V_{\lambda} \to V_{\lambda}$. We further show that the embeddings 
$
d\xi\vert_{V_{\lambda}}: V_{\lambda} \to V_{\xi \circ \lambda},
$
explained at the beginning of Section~\ref{ssec:xi-conormal}, all take the form
$
V^s \subseteq V
$
where $V$ is one of these five prehomogeneous vector spaces and $s\in H$ has finite order (hence semisimple). 
We are then able to match $L$-parameters $\phi$ with $H^s$-orbits $C'\subseteq V^s$ and identify those that are $\xi$-conormal with the corresponding property of the $H^s$-orbit $C'$; we refer to these orbits as $V$-conormal. This is the genesis of the notion of $\xi$-conormal, in fact. 
We then prove the theorem for these prehomogeneous vector spaces in the form of Proposition~\ref{prop:PHVFPF}, for all $V$-conormal orbits $C'\subseteq V^s$.

\subsection{Prehomogeneous vector subspaces}\label{ssec:subPHV}

\begin{lemma}\label{lem:subPHV}
Let $H\times V \to V$ be one of the five prehomogenous vector spaces appearing in Proposition~\ref{prop:PHV}; observe that $H$ is reductive in these cases.
For semisimple $s\in H$, set $H^s\ceq Z_{H}(s)$ and $V^s\ceq \{ x\in V \tq s(x) =x\}$.
Then $H^s\times V^s \to V^s$ be is also a prehomogeneous vector space and again one of the five prehomogenous vector spaces appearing in Proposition~\ref{prop:PHV}.
\end{lemma}

\begin{proof}
When $V=V^s$ and $H=H^s$ this is a tautology, whereas $V=V^s$ and $H\neq H^s$ only occurs for \ref{geocase-0}.

Up to the natural notion of equivalence, the possibilities for proper fix-point prehomogenous vector subspaces are the following.
\begin{enumerate}

\item[\ref{geocase-1}]
The prehomogeneous vector space $V= \mathbb{A}^1$ with $H=\GL_1$-action (scalar multiplication) has only one proper fixed-point prehomogeneous vector subspace: 
\begin{enumerate}
\item[(i)] $V^s=\{0\}$ with $H^s=\GL_1$ (this is \ref{geocase-0});
\end{enumerate}
\item[\ref{geocase-2}]
The prehomogeneous vector space $V= \mathbb{A}^2$ with $H=\GL_2$-action given by twisted matrix multiplication $h.x=\det(h)^n\, hx$ also has the following proper fixed-point prehomogeneous vector subspaces: 
\begin{enumerate}
\item[(i)] $V^s=\{0\}$ with $H^s=\GL_2$ or $\GL_1^2$  (this is \ref{geocase-0});
\item[(ii)] $V^s = \{ (x_1,0)\tq x_1\}$ and $H^s=\GL_1^2$ (this is equivalent to \ref{geocase-1});
\end{enumerate}
\item[\ref{geocase-toric}]
The prehomogeneous vector space $V=\mathbb{A}^2$ with $H=\GL_1^2$-action given by $(t_1,t_2).(x_1,x_2) = (t_1x_1,t_1t_2^nx_2)$, for positive integer $n$, the following proper fixed-point prehomogeneous vector subspaces are possible: 
\begin{enumerate}
\item[(i)] $V^s = \{ (0,0)\}$ and $H^s = H$  (this is \ref{geocase-0});
\item[(ii)] $V^s = \{ (x_1,0)\tq x_1\}$ and $H^s = H$  (this is \ref{geocase-1});
\item[(iii)] $V^s = \{ (0,x_2)\tq x_2\}$ and $H^s = H$ (this is \ref{geocase-1});
\end{enumerate}
\item[\ref{geocase-sub}]
The prehomogeneous vector space $V=\mathbb{A}^4$ with $H=\GL_2$ and action $\det^{-1}\otimes\Sym^3$ has the following proper fixed-point prehomogeneous vector subspaces, up to isomorphism: 
\begin{enumerate}
\item[(i)] $V^s = \{(0,0,0,0)\}$ and $H^s = \GL_2$ or $\GL_1^2$ (this is \ref{geocase-0});
\item[(ii)] $V^s = \{(x_1,0,0,0)\tq x_1 \}$ and $H^s = \GL_1^2$ (this is equivalent to \ref{geocase-1});
\item[(iii)] $V^s = \{(0,x_2,0,0)\tq x_2 \}$ and $H^s = \GL_1^2$ (this is equivalent to \ref{geocase-1});
\item[(iv)] $V^s = \{(x_1,0,x_3,0)\tq x_1, x_3 \}$ and $H^s = \GL_1^2$ (this is \ref{geocase-toric} for $n=2$);
\item[(v)] $V^s = \{(x_1,0,0,x_4)\tq x_1, x_4 \}$ and $H^s = \GL_1^2$ (this is \ref{geocase-toric} for $n=3$).
\qedhere
\end{enumerate}
\end{enumerate}
\end{proof}

\subsection{$V$-conormal}\label{ssec:V-conormal}

Let $H\times V\to V$ be one of the prehomogeneous vector spaces appearing in Proposition~\ref{prop:PHV} and let $H^s\times V^s\to V^s$ be one of the subspaces appearing in Lemma~\ref{lem:subPHV}, for semisimple $s\in H$.
Then $T^*(V)^s = T^*(V^s)$, where $H$ acts on $T^*(V)$ by $s\cdot(x,y) \ceq (s\cdot x,s\cdot y)$. Recall that $\Lambda = \{ (x,y)\in T^*(V) \tq [x,y]=0\}$; then $\Lambda^s = \{ (x',y')\in T^*(V^s) \tq [x',y']=0\}$.

Now let $C'\subseteq V^s$ be an $H^s$-orbit.
The $H$-orbit of $C'$, denoted by $H\cdot C'\subseteq V$ and sometimes called the saturation of $C'$, is a single $H$-orbit, denoted below by $C$.
The intersection $C^s$ of $C$ with $V^s$ contains $C'$ but may contain other $H^s$-orbits: we write $C^s = \cup_i C'_i$.
Then, $(\Lambda_{C})^s = \cup_i (\Lambda^s)_{C'_i}$
Passing to the regular part of the conormal varieties we find that it can happen that
$(\Lambda^\text{sreg}_{C})^s$ and $\cup_i (\Lambda^s_{C_i'})^\text{sreg}$ are not equal. Of course, this is essentially the same phenomenon discussed in Section~\ref{ssec:xi-conormal}. Definition~\ref{def:xi-conormal} now takes this form.

\begin{definition}\label{def:V-conormal}
With notation as above, let us say that an $H^s$-orbit $C'\subseteq V^s$ is $V$-conormal if $(\Lambda^s_{C'})^\text{sreg}\subseteq (\Lambda^\text{sreg}_{C})^s$, where the $H$-orbit $C\subseteq V$ is the saturation of $C'$.
In this case, this inclusion induces a map of equivariant fundamental groups $A^\ABV_{C'} \to A^\ABV_{C}$.
\end{definition}

Whenever $V=V^s$ and $H=H^s$ then we trivially have that every orbit is $V$-conormal.
 For those cases under consideration in this paper where $V=V^s$ and $H\neq H^s$ the same remains true. We note that this only occurs for \ref{geocase-0}.

From the cases of prehomogenous vector spaces $V$ and proper subspaces $V^s$ appearing in Lemma~\ref{lem:subPHV}, only the following $H^s$-orbits $C'\subset V^s$ are $V$-conormal:
\begin{enumerate}
\item Case \ref{geocase-2}(ii). 
The closed orbit $C'_0$ and the open orbit $C'_1$ in $V^s$ are $V$-conormal.
We note that the saturation of the closed $H^s$-orbit $C'_0$ in $V^s$  is the closed $H$-orbit $C_0$ in $V$; the saturation of the open $H^s$-orbit $C'_1$ in $V^s$ is the $H$-orbit $C_1$ in $V$, which is open.
\item Case \ref{geocase-sub}(iv).
Every $H^s$-orbit $C'$ in $V^s$ is $V$-conormal.
These orbits are denoted by $C'_0$, $C'_1$, $C'_2$ and $C'_3$, where $C'_0$ is closed and $C'_3$ is open and where $C'_1$ is the $H^s$-orbit of $(1,0,0,0)$ and $C'_2$ is the $H^s$-orbit of $(0,0,1,0)$. 
The saturation of the closed $H^s$-orbit $C'_0$ is the closed $H$-orbit $C_0$, the saturation of $C'_1$ is $C_1$, the saturation of $C'_2$ is $C_2$ and the saturation of the open $H^s$-orbit $C'_3$ is the open $H$-orbit $C_3$. 
\item Case \ref{geocase-sub}(v).
Only two of the four $H^s$-orbits $C'$ in $V^s$ are $V$-conormal.
Let the $H^s$-orbits in $V^s$ be denoted by $C'_0$, $C'_1$, $C'_2$ and $C'_3$, where $C'_0$ is closed and $C'_3$ is open and where $C'_1$ is the $H^s$-orbit of $(1,0,0,0)$ while $C'_2$ is the $H^s$-orbit of $(0,0,0,1)$. 
Then the saturation of the closed $H^s$-orbit $C'_0$ is the closed $H$-orbit $C_0$, the saturation of $C'_1$ is $C_1$, the saturation of $C'_2$ is $C_1$ again, and the saturation of the open $H^s$-orbit $C'_3$ is the open $H$-orbit $C_3$.
Only $C'_0$ and $C'_3$ are $V$-conormal.
\end{enumerate}

\subsection{Restriction}\label{ssec:restriction}

Let $(G,s,\xi)$ be an endscopic triple for $G_2$ over $F$.
Let $\lambda$ be an infinitesimal parameter for $G$.
To prove Theorem~\ref{thm:VFPF} we will calculate the equivariant restriction functor
\[
\res : D_{H_\lambda}(V_{\xi\circ\lambda}) \to D_{H_\lambda}(V_{\lambda})
\]
on simple objects in the abelian category $\Perv_{H}(V_{\xi\circ\lambda})$. As we will see, the restriction $\res\,\mathcal{P}$ of an equivariant perverse sheaf is not necessarily perverse. 
To address this issue, we replace $\Perv_{H}(V_{\lambda})$ with the full subcategory $\Perv_{H}^{\bullet}(V_{\lambda})$ of $D_{H}(V_{\lambda})$ generated by shifts $\mathcal{P}[n]$ of equivariant perverse sheaves. In this section we will see that $\res$ defines
\[
\res : \Perv^{\bullet}_{H_{\xi\circ\lambda}}(V_{\xi\circ\lambda}) \to \Perv^{\bullet}_{H_\lambda}(V_{\lambda})
\]
We use this fact to replace restriction along $d\xi : V_{\lambda} \to V_{\xi\circ\lambda}$ with restriction along $V^s \hookrightarrow V$, where $H\times V\to V$ is one of the five prehomogeneous vector spaces from Proposition~\ref{prop:PHV} and $s\in H$ is semisimple.
The functor
\[
\res : \Perv^{\bullet}_{H}(V) \to \Perv^{\bullet}_{H^s}(V^s)
\]
on simple objects is given in Table~\ref{table:restriction0} for the three prehomogeneous vector spaces appearing semisimple $s\in H$ appearing in Section~\ref{ssec:V-conormal}.

\begin{table}[htp]
\caption{Restrictions of Standard and Simple Perverse Sheaves}
\label{table:restriction0}
\begin{center}
$
\begin{array}{l l l || l l c }
\text{Case} &  \text{Standard} & \text{Restriction} & \text{Perverse} & \text{Restriction} & \\
 &  \text{sheaf} &  & \text{sheaf} &  & \\
\hline\hline
\ref{geocase-2}(ii)
&  \1_{C_0}  &    \1_{C_0'}   & \IC(\1_{C_0}  )  & \IC(\1_{C'_0}) \\
&  \1_{C_1}  &    \1_{C_1'}   & \IC( \1_{C_1} )  & \IC(\1_{C_1'}  )[1] \\
\hline
\ref{geocase-sub}(iv)
&  \1_{C_0}  &   \1_{C_0'}   &  \IC( \1_{C_0} )  & \IC(\1_{C'_0})  & \\
&  \1_{C_1}  &   \1_{C_1'}    &  \IC( \1_{C_1} )  &  \IC( \1_{C'_1} )[1] & \\
&  \1_{C_2}  &   \1_{C_2'}   &  \IC(  \1_{C_2})  &  \mathcal{F}_2[2] \oplus \IC(\1_{C'_0})[1]  & \\
&  \1_{C_3}  &   \1_{C_3'}   &  \IC( \1_{C_3} )  &  \IC(\1_{C_3'}  )[2] & \\
&  \varrho_{C_3} &   \1_{C_3'}  \oplus \rho_{C_3'}   &  \IC(\varrho_{C_3})  &  \mathcal{F}_3[2] \oplus \IC(\1_{C'_0})[2] \oplus \IC( \rho_{C'_3})[2]  & \\
&  \varepsilon_{C_3}  &  \rho_{C_3'}  &  \IC(\varepsilon_{C_3}  )  &  \IC(\rho_{C'_3})[2] & \\
\hline
\ref{geocase-sub}(v)
&  \1_{C_0}  &    \1_{C_0'}   &  \IC( \1_{C_0} )  &  \IC(\1_{C'_0}) &   \\
&  \1_{C_1}  &   \1_{C_1'} \oplus \1_{C_2'}    &  \IC( \1_{C_1} )  &   \mathcal{F}_4[1] &  \\
&  \1_{C_2}  &   0  &  \IC( \1_{C_2} )  &   \mathcal{F}_4[2] \oplus \IC(\1_{C'_0})[1]  &   \\
&  \1_{C_3}  &   \1_{C_3'}   &  \IC(\1_{C_3}  )  &  \IC(\1_{C_3'}  )[2] &  \\
&  \varrho_{C_3} &   \rho_{C'_3} \oplus  \rho^2_{C'_3}   &  \IC(\varrho_{C_3}  )  &  \IC(\rho_{C'_3})[2]\oplus \IC(\rho^2_{C'_3})[2]  \oplus \IC(\1_{C'_0})[2]  &   \\
&  \varepsilon_{C_3} &  \1_{C_3'}   &  \IC( \varepsilon_{C_3} )  &  \mathcal{F}_5[2]  &   \\
\hline
\end{array}
$
\end{center}
\end{table}%

In Table~\ref{table:restriction0} we find four indecomposable perverse sheaves, given here.
\begin{itemize}[leftmargin=15pt]

\item 
On \ref{geocase-toric} (for $n=2$) we define $\mathcal{F}_2 = \1_{C_2\cup C_1\cup C_0}[1]$ which is cohomologically equivalent to a complex 
\[ 
\cdots \rightarrow 0 \rightarrow (\1_{\overline{C_2}} \oplus \1_{\overline{C_1}}) \rightarrow \1_{C_0} 
\]
and hence sits in the exact sequence
\[ 0 \rightarrow \IC(\1_{C_0}  ) \rightarrow \mathcal{F}_2 \rightarrow \IC(\1_{C_1}  ) \oplus \IC(\1_{C_2}  ) \rightarrow 0, \]
which we use to compute the functor $\NEvs$ on these perverse sheaves in Table~\ref{tab:summary-iv}.
\item 
On \ref{geocase-toric}  (for $n=2$) we define $\mathcal{F}_3 =  \1_{C_3\cup C_2}[2] $ which is  cohomologically equivalent to a complex  
\[ \cdots \rightarrow 0 \rightarrow \1_{\overline{C_3}} \rightarrow1_{\overline{C_1}} \rightarrow 0 \]
and hence sits in the exact sequence
\[ 0 \rightarrow \IC(\1_{C_1}  ) \rightarrow \mathcal{F}_3 \rightarrow \IC(\1_{C_3}  ) \rightarrow 0, \]
which we use to compute the functor $\NEvs$ on these perverse sheaves in Table~\ref{tab:summary-iv}.
\item 
On \ref{geocase-toric}  (for $n=2$) we define $\mathcal{F}_4 = \1_{C_2\cup C_1\cup C_0}[1]$ which is cohomologically equivalent to a complex 
\[ 
\cdots \rightarrow 0 \rightarrow (\1_{\overline{C_2}} \oplus \1_{\overline{C_1}}) \rightarrow \1_{C_0} 
\]
and hence sits in the exact sequence
\[ 0 \rightarrow \IC(\1_{C_0}  ) \rightarrow \mathcal{F}_4 \rightarrow \IC(\1_{C_1}  ) \oplus \IC(\1_{C_2}  ) \rightarrow 0, \]
which we use to compute the functor $\NEvs$ on these perverse sheaves in Table~\ref{tab:summary-v}.
\item 
On \ref{geocase-toric}  (for $n=2$) we define   $\mathcal{F}_5 = \1_{C_3}$ which is cohomologically equivalent to a complex  
\[  \1_{\overline{C_3}}[2]  \rightarrow (\1_{\overline{C_2}} \oplus \1_{\overline{C_1}})[1] \rightarrow \1_{C_0}[0] \]
and hence sits in the exact sequence
\[ 0 \rightarrow \mathcal{F}_4   \rightarrow \mathcal{F}_5 \rightarrow  \IC(\1_{C_3}  ) \rightarrow 0, \]
which we use to compute the functor $\NEvs$ on these perverse sheaves in Table~\ref{tab:summary-v}.
\end{itemize}



\subsection{Proof of Theorem~\ref{thm:VFPF}}\label{ssec:FPF}

Let $H\times V\to V$ be one of the prehomogeneous vector spaces appearing in Proposition~\ref{prop:PHV}.
The result is immediate if $V=V^s$ and $H=H^s$ or in the case \ref{geocase-0} so we
 let $H^s\times V^s\to V^s$ be one of the proper prehomogeneous vector spaces appearing in Lemma~\ref{lem:subPHV}.
Let $C'$ be an $H^s$-orbit in $V^s$ and let $C\subseteq V$ be the saturation of $C'$.
In Proposition~\ref{prop:PHVFPF} we show that if $C'$ is $V$-conormal then the outside of the following diagram commutes.
\[
\begin{tikzcd}
D_{H}(V)  \arrow{dd}{res} \arrow{rrr}{\NEvs_{C}[\dim C]} &&& D_{H}(\Lambda^\text{sreg}_{C}) \arrow{d}[swap]{\res}  \arrow{rrd}{\trace_s} && \\
&&& D_{H}((\Lambda^\text{sreg}_{C})^s ) \arrow{d}{\res} && \CC\\
D_{H^s}(V^s) \arrow{rrr}{\NEvs_{C'}[\dim C']} &&& D_{H^s}((\Lambda^s_{C'})^\text{sreg})   \arrow{urr}[swap]{\trace_s} && 
\end{tikzcd}
\]
It is important to note that the square in this diagram does not commute.

\begin{proposition}\label{prop:PHVFPF}
Let $H\times V\to V$ be one of the prehomogeneous vector spaces appearing in Proposition~\ref{prop:PHV} and let $H^s\times V^s\to V^s$ be one of the prehomogeneous vector spaces appearing in Lemma~\ref{lem:subPHV}.
Let $C'$ be an $H^s$-orbit in $V^s$.
If $C'$ is $V$-conormal then
\begin{equation}\label{eqn:FPF}
\trace_{s} \left(\NEvs_{C}[\dim C]\ \mathcal{P}\right)
= \trace_{s} \left(\NEvs_{C'}[\dim C'] \res\, \mathcal{P}\vert_{V^s} \right),
\end{equation}
for all equivariant perverse sheaves $\mathcal{P}$ on $V$, where $C\subseteq V$ is the saturation of $C'$.
\end{proposition}

\begin{proof}

We prove this by verifying Equation~\eqref{eqn:FPF} by calculating the left- and right-hand sides of Equation~\eqref{eqn:FPF} independently, in each of the three cases appearing in Section~\ref{ssec:V-conormal}.
\def\ANDREWSHORTENED{1}
\def\varisone{1}
\ifx\ANDREWSHORTENED\varisone
\begin{enumerate}

\item Case \ref{geocase-2}(ii). 
The main calculations are summarized by the Table \ref{tab:summary-i}. 
The calculations are completed by comparing $\trace_s$ of the sheaves in the final two columns.
In this case these are always $1$.
We note that for each $C_i'$ the saturation is $C_i$. 

\begin{table}[ht]
\caption{Summary of Calculations for Case \ref{geocase-2}(ii)}\label{tab:summary-i}
\begin{tabular}{ccccc}
 $\mathcal{P}$ & $ \res\ \mathcal{P}$ & $C'$ &$\NEvs_{C}[\dim{C}]\ \mathcal{P}$ & $\NEvs_{C'}[\dim{C'}]\ \res\ \mathcal{P}$ \\
\hline\hline
$\IC( \1_{C_0})$ & $\IC( \1_{C_0'})$
       &$ C_0'$ &  $ \1_{\Lambda^\text{sreg}_{C_0}}[0] $     &    $ \1_{\Lambda^\text{sreg}_{C'_0}}[0] $     \\
   & & $C_1'$ &   $  0_{\Lambda^\text{sreg}_{C_1}}[0]$    &    $ 0_{\Lambda^\text{sreg}_{C'_1}}[0]$      \\
\hline
$\IC( \1_{C_1})$ & $\IC( \1_{C_1'})[1]$
       & $C_0'$ &   $\1_{\Lambda^\text{sreg}_{C_0}}[0] $     &     $ \1_{\Lambda^\text{sreg}_{C'_0}}[0]$     \\
   & & $C_1'$ &   $ \1_{\Lambda^\text{sreg}_{C_1}}[2]$    &      $ \1_{\Lambda^\text{sreg}_{C'_1}}[2] $     \\
\end{tabular}
\end{table}

\else
 For brevity, here we include only Case \ref{geocase-sub}(v).
\begin{enumerate}

\fi

\ifx\ANDREWSHORTENED\varisone
\item Case \ref{geocase-sub}(iv).
This is the fixed-point prehomogeneous vector subspace $H^s\times V^s \to V^s$ for $s = \diag(-1,1)\in \GL_2$ and $H=\GL_2$ acting on $V=\mathbb{A}^4$ by $\det^{-1}\otimes\Sym^3$; then $H^s = \GL_1^2$ and $V^s = \mathbb{A}^2$ with action equivalent to \ref{geocase-toric} in the case $n=2$.
Recall that we write $C'_0$ for the trivial $H^s$-orbit, $C'_1$ for the $H^s$-orbit of $(1,0,0,0)\in V^s$, $C'_2$ for the $H^s$-orbit of $(0,0,1,0)\in V^s$ and $C'_3$ for the open $H^s$-orbit in $V^s$.  The equivariant fundamental group of each $H^s$-orbit is trivial except $A_{C'_3} = S_2$; let $\rho$ be the non-trivial quadratic character of this group and let $\rho_{C'_3}$ be the corresponding local system on $C'_3$.

The main calculations are summarized by the Table \ref{tab:summary-iv}. 
The calculations are completed by comparing $\trace_s$ of the sheaves in the final two columns  and accounting for shifts.
Because the image of $s$ in $A^\ABV_{C}$ is order $2$ we have $\trace_s \varrho = 0$, $\trace_s \varepsilon = -1$, $\trace_s \rho = -1$ and $\trace_s \vartheta_2 = -1$. We note that for each $C_i'$ the saturation is $C_i$.

\begin{table}[tb]
\caption{Summary of Calculations for Case \ref{geocase-sub}(iv). In this table we write $\vartheta : \{ 1, \theta_2\} \to \CC^\times$ for the character generally denoted in this paper by $\vartheta_2$, defined by $\vartheta(\theta_2) = \theta_2$. }
\label{tab:summary-iv}
\begin{center}
$
\begin{array}{cccccc} 
  \mathcal{P}  &   \res\ \mathcal{P}  &  C'  & \NEvs_{C}[\dim{C}]\ \mathcal{P}  &  \NEvs_{C'}[\dim{C'}]\ \res\ \mathcal{P}  \\
\hline\hline
 \IC( \1_{C_0})  &  \IC( \1_{C_0'}) 
       &  C_0'  &     \1_{\Lambda^\text{sreg}_{C_0}}[0]       &      \1_{\Lambda^\text{sreg}_{C'_0}}[0]       \\
   & &  C_1'  &      0     &      0       \\
   & &  C_2'  &      0     &      0       \\
   & &  C_3'  &      0     &      0       \\
\hline
 \IC( \1_{C_1})  &  \IC( \1_{C_1'})[1] 
       &  C_0'  &   \varrho_{\Lambda^\text{sreg}_{C_0}}[0]       &       0      \\
   & &  C_1'  &     \1_{\Lambda^\text{sreg}_{C_1}}[2]     &        \1_{\Lambda^\text{sreg}_{C'_1}}[2]       \\
   & &  C_2'  &     0     &      0       \\
   & &  C_3'  &     0     &      0       \\
   \hline
 \IC( \1_{C_2})  &  \mathcal{F}_2[2] \oplus \IC( \1_{C_0'})[1] 
       &  C_0'  &   0       &       \1_{\Lambda^\text{sreg}_{C'_0}}[2]\oplus\1_{\Lambda^\text{sreg}_{C'_0}}[1]      \\
   & &  C_1'  &      \vartheta_{\Lambda^\text{sreg}_{C_1}}[2]     &        \1_{\Lambda^\text{sreg}_{C'_1}}[3]       \\
   & &  C_2'  &      \1_{\Lambda^\text{sreg}_{C_2}}[3]     &      \1_{\Lambda^\text{sreg}_{C'_2}}[3]       \\
   & &  C_3'  &      0     &      0       \\
   \hline
 \IC( \1_{C_3})  &  \IC( \1_{C_3'})[2] 
       &  C_0'  &   0       &       0      \\
   & &  C_1'  &     0     &        0      \\
   & &  C_2'  &      0     &      0       \\
   & &  C_3'  &      \1_{\Lambda^\text{sreg}_{C_3}}[4]     &      \1_{\Lambda^\text{sreg}_{C'_3}}[4]       \\
     \hline
 \IC( \varrho_{C_3})  & ( \mathcal{F}_3\oplus  \IC( \1_{C_0'}) \oplus  \IC(\rho_{C_3'}))[2]  
       &  C_0'  &  0       &       \1_{\Lambda^\text{sreg}_{C'_0}} [2]  \oplus \vartheta_{\Lambda^\text{sreg}_{C'_0}} [2]       \\
   & &  C_1'  &     0     &         \1_{\Lambda^\text{sreg}_{C'_1}}[3] \oplus \vartheta_{\Lambda^\text{sreg}_{C'_1}}[3]       \\
   & &  C_2'  &    \vartheta_{\Lambda^\text{sreg}_{C_2}}[3]     &      \vartheta_{\Lambda^\text{sreg}_{C'_2}}[3]        \\
   & &  C_3'  &       \varrho_{\Lambda^\text{sreg}_{C_3}}[4]     &      \1_{\Lambda^\text{sreg}_{C'_3}}[4] \oplus  \vartheta_{\Lambda^\text{sreg}_{C'_3}}[4]        \\
      \hline
  \IC( \varepsilon_{C_3})  &  \IC( \rho_{C'_3})[2] 
       &  C_0'  &     \varepsilon_{\Lambda^\text{sreg}_{C_0}}[0]       &       \vartheta_{\Lambda^\text{sreg}_{C'_0}}[2]       \\
   & &  C_1'  &     \1_{\Lambda^\text{sreg}_{C_1}}[2]     &        \vartheta_{\Lambda^\text{sreg}_{C'_1}}[3]       \\
   & &  C_2'  &      \vartheta_{\Lambda^\text{sreg}_{C_2}}[3]     &      \vartheta_{\Lambda^\text{sreg}_{C'_2}}[3]        \\
   & &  C_3'  &      \varepsilon_{\Lambda^\text{sreg}_{C_3}}[4]     &     \vartheta_{\Lambda^\text{sreg}_{C'_3}}[4]        \\
\end{array}
$
\end{center}
\end{table}

\fi

\item[(3)] Case \ref{geocase-sub}(v).
The fixed-point prehomogenous space $H^s\times V^s \to V^s$ is equivalent to \ref{geocase-toric} in the case $n=3$.
Without loss of generality, we take $s = \diag(\theta_3^2,\theta_3)$.
Recall that in Case \ref{geocase-sub}(v) there are only two $H^s$-orbits $C'\subset V^s$ that are $V$-conormal: the open orbit $C'_0$ and the closed orbit $C'_3$.

The main calculations are summarized by Table \ref{tab:summary-v}. 
The calculations are completed by comparing $\trace_s$ of the sheaves in the final two columns and accounting for shifts.
Because the image of $s$ in $A^\ABV_{C}$ has order $3$ (or order $1$) we have
$\trace_s \vartheta_2 = 1$, $\trace_s \varepsilon = 1$, $\trace_s \varrho = \theta_3 + \theta_3^2 =  -1$ and
$\trace_s \rho \oplus \rho^2 = \theta_3 + \theta_3^2 =  -1$.
Note that saturation of $C_1'$ and $C_2'$ are both $C_1$, these do not satisfy the hypotheses, and that these cases do not always satisfy (\ref{eqn:FPF}).\qedhere

\begin{table}[tp]
\caption{Summary of Calculations for Case \ref{geocase-sub}(v). In this table we write $\vartheta : \{ 1, \theta_3,\theta_3^2\} \to \CC^\times$ for the character generally denoted in this paper by $\vartheta_3$, defined by $\vartheta(\theta_3) = \theta_3$. }
\label{tab:summary-v}
\begin{center}
\resizebox{\textwidth}{!}{
$
\begin{array}{cccccc}
  \mathcal{P}  &   \res\ \mathcal{P}  &  C'  & \NEvs_{C}[\dim{C}]\ \mathcal{P}  &  \NEvs_{C'}[\dim{C'}]\ \res\ \mathcal{P}  \\
\hline\hline
 \IC( \1_{C_0})  &  \IC( \1_{C_0'}) 
       &  C_0'  &     \1_{\Lambda^\text{sreg}_{C_0}}[0]       &      \1_{\Lambda^\text{sreg}_{C'_0}}[0]       \\
   & &  C_1'  &      0     &      0       \\
   & &  C_2'  &      0     &      0       \\
   & &  C_3'  &      0     &      0       \\
\hline
 \IC( \1_{C_1})  &  \mathcal{F}_4[1]  
       &  C_0'  &    \varrho_{\Lambda^\text{sreg}_{C_0}}[0]           &       \1_{\Lambda^\text{sreg}_{C'_0}}[1]      \\
   & &  C_1'  &      \1_{\Lambda^\text{sreg}_{C_1}}[2]       &        \1_{\Lambda^\text{sreg}_{C'_1}}[2]      \\
   & &  C_2'  &      \1_{\Lambda^\text{sreg}_{C_1}}[2]      &        \1_{\Lambda^\text{sreg}_{C'_2}}[2]        \\
   & &  C_3'  &     0     &      0       \\
   \hline
 \IC( \1_{C_2})  &  \mathcal{F}_4[2] \oplus \IC(\1_{C'_0})[1] 
       &  C_0'  &    0      &       \1_{\Lambda^\text{sreg}_{C'_0}} [2] \oplus \1_{\Lambda^\text{sreg}_{C'_0}}[1]        \\
   & &  C_1'  &      \vartheta_{\Lambda^\text{sreg}_{C_1}}[2]      &        \1_{\Lambda^\text{sreg}_{C'_1}}[3]        \\
   & &  C_2'  &     \vartheta_{\Lambda^\text{sreg}_{C_1}}[2]      &       \1_{\Lambda^\text{sreg}_{C'_1}}[3]       \\
   & &  C_3'  &      0     &      0       \\
   \hline
 \IC( \1_{C_3})  &  \IC( \1_{C_3'})[2] 
       &  C_0'  &   0       &       0      \\
   & &  C_1'   &     0     &        0      \\
   & &  C_2'  &      0     &      0       \\
   & &  C_3'  &      \1_{\Lambda^\text{sreg}_{C_3}}[4]     &      \1_{\Lambda^\text{sreg}_{C'_3}}[4]       \\
     \hline
 \IC( \varrho_{C_3})  & (  \IC(\1_{C'_0})\oplus \IC(\vartheta_{C'_3})\oplus \IC(\vartheta^2_{C'_3}))[2]   
       &  C_0'  &    0      &       \1_{\Lambda^\text{sreg}_{C'_0}}[2] \oplus \vartheta_{\Lambda^\text{sreg}_{C'_3}}[2] \oplus \vartheta^2_{\Lambda^\text{sreg}_{C'_3}}[2]        \\
   & &  C_1'  &    0       &          \vartheta_{\Lambda^\text{sreg}_{C'_3}}[3] \oplus \vartheta^2_{\Lambda^\text{sreg}_{C'_3}}[3]    \\
   & &  C_2'  &    0     &        \vartheta_{\Lambda^\text{sreg}_{C'_3}}[3] \oplus \vartheta^2_{\Lambda^\text{sreg}_{C'_3}}[3]       \\
   & &  C_3'  &       \varrho_{\Lambda^\text{sreg}_{C_3}}[4]     &       \vartheta_{\Lambda^\text{sreg}_{C'_3}}[4] \oplus \vartheta^2_{\Lambda^\text{sreg}_{C'_3}}[4]       \\
      \hline
  \IC( \varepsilon_{C_3})  &  \mathcal{F}_5[2]   
       &  C_0'  &     \varepsilon_{\Lambda^\text{sreg}_{C_0}}[0]     &    \1_{\Lambda^\text{sreg}_{C'_0}}[2]      \\
   & &  C_1'  &       \1_{\Lambda^\text{sreg}_{C_1}}[2]      &         \1_{\Lambda^\text{sreg}_{C'_1}}[3]      \\
   & &  C_2'  &        \1_{\Lambda^\text{sreg}_{C_1}}[2]      &      \1_{\Lambda^\text{sreg}_{C'_2}}[3]       \\
   & &  C_3'  &      \varepsilon_{\Lambda^\text{sreg}_{C_3}}[4]     &       \1_{\Lambda^\text{sreg}_{C'_3}}[4]      \\
\end{array}
$
}
\end{center}
\end{table}

\ifx\ANDREWSHORTENED\varisone
\else
In this case we include the calculation details. 

\begin{enumerate}
\item 
Take $\mathcal{P} =  \IC( \1_{C_0})$ and recall that $\res\ \IC(\1_{C_0}) = \IC(\1_{C'_0})$.
\begin{enumerate}

\item
Take $C'=C'_0$, so $C=C_0$. 
Then the left-hand side of  Equation~\eqref{eqn:FPF} is
\[
\begin{array}{rcl}
&&\hskip-1cm \trace_{s} \left(\NEvs_{C_0}[\dim{C_0}]\ \IC(\1_{C_0})\right) \\
&=& \trace_s \1_{\Lambda^\text{sreg}_{C_0}}[0]\\
&=& 1.
\end{array}
\]
The right-hand side of Equation~\eqref{eqn:FPF} is 
\[
\begin{array}{rcl}
&&\hskip-1cm \trace_{s} \left(\NEvs_{C'_0}[\dim{C'_0}]\ \res\ \IC(\1_{C_0})\right) \\
&=& \trace_{s} \left(\NEvs_{C'_0} [0]\IC(\1_{C'_0})\right) \\
&=& \trace_{s} \1_{\Lambda^\text{sreg}_{C'_0}} \\
&=& 1.
\end{array}
\]
This verifies Equation~\eqref{eqn:FPF} for $\mathcal{P} =  \IC( \1_{C_0})$ and $C'=C'_0$ in Case \ref{geocase-sub}(v).

\item
Take $C'=C'_3$, in which case $C=C_3$.
The left-hand side of  Equation~\eqref{eqn:FPF} is
\[
\begin{array}{rcl}
&&\hskip-1cm \trace_{s} \left(\NEvs_{C_3}[\dim{C_3}]\ \IC(\1_{C_0})\right) \\
&=& \trace_s 0_{\Lambda^\text{sreg}_{C_3}}[4]\\
&=& 0.
\end{array}
\]
The right-hand side of Equation~\eqref{eqn:FPF} is 
\[
\begin{array}{rcl}
&&\hskip-1cm \trace_{s} \left(\NEvs_{C'_3}[\dim{C'_3}]\ \res\ \IC(\1_{C_0})\right) \\
&=& \trace_{s} \left(\NEvs_{C'_3} [1] \IC(\1_{C'_0})\right) \\
&=& \trace_{s} 0_{\Lambda^\text{sreg}_{C'_3}}[1] \\
&=& 0.
\end{array}
\]
This verifies Equation~\eqref{eqn:FPF} for $\mathcal{P} =  \IC( \1_{C_0})$ and $C'=C'_3$ in Case \ref{geocase-sub}(v).
\end{enumerate}
This verifies Equation~\eqref{eqn:FPF} for $\mathcal{P} =  \IC( \1_{C_0})$ in Case \ref{geocase-sub}(iv).

\item 
Take $\mathcal{P} =  \IC( \1_{C_1})$ and recall that $\res\ \IC(\1_{C_1}) = \mathcal{F}_4[1]$.
\begin{enumerate}
\item
Take $C'=C'_0$, so $C=C_0$. 
Then the left-hand side of  Equation~\eqref{eqn:FPF} is
\[
\begin{array}{rcl}
&&\hskip-1cm \trace_{s} \left(\NEvs_{C_0}[\dim{C_0}]\ \IC(\1_{C_1})\right) \\
&=& \trace_s \varrho_{\Lambda^\text{sreg}_{C_0}}[0]\\
&=& -1.
\end{array}
\]
The right-hand side of Equation~\eqref{eqn:FPF} is 
\[
\begin{array}{rcl}
&&\hskip-1cm \trace_{s} \left(\NEvs_{C'_0}[\dim{C'_0}]\ \res\ \IC(\1_{C_1})\right) \\
&=& \trace_{s} \left(\NEvs_{C'_0}[0] \mathcal{F}_4)[1]\right) \\
&=& \trace_{s} \1_{\Lambda^\text{sreg}_{C'_0}}[1] \\
&=& -1.
\end{array}
\]
This verifies Equation~\eqref{eqn:FPF} for $\mathcal{P} =  \IC( \1_{C_1})$ and $C'=C'_0$ in Case \ref{geocase-sub}(v).

\item
Take $C'=C'_3$, in which case $C=C_3$.
The left-hand side of  Equation~\eqref{eqn:FPF} is
\[
\begin{array}{rcl}
&&\hskip-1cm \trace_{s} \left(\NEvs_{C_3}[\dim{C_3}]\ \IC(\1_{C_1})\right) \\
&=& \trace_s 0_{\Lambda^\text{sreg}_{C_2}}[4]\\
&=& 0.
\end{array}
\]
The right-hand side of Equation~\eqref{eqn:FPF} is 
\[
\begin{array}{rcl}
&&\hskip-1cm \trace_{s} \left(\NEvs_{C'_3}[\dim{C'_3}]\ \res\ \IC(\1_{C_1})\right) \\
&=& \trace_{s} \left(\NEvs_{C'_3}[2] \IC(\1_{C'_1})[1]\right) \\
&=& \trace_{s} 0_{\Lambda^\text{sreg}_{C'_3}}[2] \\
&=& 0.
\end{array}
\]
This verifies Equation~\eqref{eqn:FPF} for $\mathcal{P} =  \IC( \1_{C_1})$ and $C'=C'_3$ in Case \ref{geocase-sub}(v).

\end{enumerate}
This verifies Equation~\eqref{eqn:FPF} for $\mathcal{P} =  \IC( \1_{C_1})$ in Case \ref{geocase-sub}(iv).

\item 
Take $\mathcal{P} =  \IC( \1_{C_2})$ and recall that $\res\ \IC(\1_{C_2}) = \mathcal{F}_4[2] \oplus \IC(\1_{C'_0})[1]$.
\begin{enumerate}
\item
Take $C'=C'_0$, so $C=C_0$. 
Then the left-hand side of  Equation~\eqref{eqn:FPF} is
\[
\begin{array}{rcl}
&&\hskip-1cm \trace_{s} \left(\NEvs_{C_0}[\dim{C_0}]\ \IC(\1_{C_2})\right) \\
&=& \trace_s 0_{\Lambda^\text{sreg}_{C_0}}[0]\\
&=& 0.
\end{array}
\]
The right-hand side of Equation~\eqref{eqn:FPF} is 
\[
\begin{array}{rcl}
&&\hskip-1cm \trace_{s} \left(\NEvs_{C'_0}[\dim{C'_0}]\ \res\ \IC(\1_{C_2})\right) \\
&=& \trace_{s} \left(\NEvs_{C'_0}[0] \left(\mathcal{F}_4[2] \oplus \IC(\1_{C'_0})[1]\right)\right) \\
&=& \trace_{s} \left(\NEvs_{C'_0}[0] \mathcal{F}_4[2]\right) +  \trace_{s} \left(\NEvs_{C'_0}[0]\IC(\1_{C_0})[1]\right) \\
&=& \trace_{s} \1_{\Lambda^\text{sreg}_{C'_0}} [2] +  \trace_{s} \1_{\Lambda^\text{sreg}_{C'_0}}[1] \\
&=& 1-1=0.
\end{array}
\]
This verifies Equation~\eqref{eqn:FPF} for $\mathcal{P} =  \IC( \1_{C_2})$ and $C'=C'_0$ in Case \ref{geocase-sub}(v).

\item
Take $C'=C'_3$, in which case $C=C_3$.
The left-hand side of  Equation~\eqref{eqn:FPF} is
\[
\begin{array}{rcl}
&&\hskip-1cm \trace_{s} \left(\NEvs_{C_3}[\dim{C_3}]\ \IC(\1_{C_2})\right) \\
&=& \trace_s 0_{\Lambda^\text{sreg}_{C_3}}[4]\\
&=& 0.
\end{array}
\]
The right-hand side of Equation~\eqref{eqn:FPF} is 
\[
\begin{array}{rcl}
&&\hskip-1cm \trace_{s} \left(\NEvs_{C'_3}[\dim{C'_3}]\ \res\ \IC(\1_{C_2})\right) \\
&=& \trace_{s} \left(\NEvs_{C'_3}[2] \left(\mathcal{F}_4[2] \oplus \IC(\1_{C_0'})[1]\right)\right) \\
&=& \trace_{s} \left(\NEvs_{C'_3}[2] \mathcal{F}_4[2]\right) +\trace_s\left(\NEvs_{C'_3}[2] \IC(\1_{C_0'})[1]\right) \\
&=& \trace_{s} 0_{\Lambda^\text{sreg}_{C'_3}}[4] + \trace_{s} 0_{\Lambda^\text{sreg}_{C'_3}}[3] \\
&=& 0.
\end{array}
\]
This verifies Equation~\eqref{eqn:FPF} for $\mathcal{P} =  \IC( \1_{C_2})$ and $C'=C'_3$ in Case \ref{geocase-sub}(v).

\end{enumerate}
This verifies  Equation~\eqref{eqn:FPF} for $\mathcal{P} =  \IC( \1_{C_2})$ in Case \ref{geocase-sub}(v).

\item 
Take $\mathcal{P} =  \IC( \1_{C_3})$ and recall that $\res\ \IC(\1_{C_3}) = \IC(\1_{C_3'})[2]$.
\begin{enumerate}
\item
Take $C'=C'_0$, so $C=C_0$. 
Then the left-hand side of  Equation~\eqref{eqn:FPF} is
\[
\begin{array}{rcl}
&&\hskip-1cm \trace_{s} \left(\NEvs_{C_0}[\dim{C_0}]\ \IC(\1_{C_3})\right) \\
&=& \trace_s 0_{\Lambda^\text{sreg}_{C_0}}[0]\\
&=& 0.
\end{array}
\]
The right-hand side of Equation~\eqref{eqn:FPF} is 
\[
\begin{array}{rcl}
&&\hskip-1cm \trace_{s} \left(\NEvs_{C'_0}[\dim{C'_0}]\ \res\ \IC(\1_{C_3})\right) \\
&=& \trace_{s} \left(\NEvs_{C'_0}[0] \IC(\1_{C'_3})[2] \right) \\
&=& \trace_{s} 0_{\Lambda^\text{sreg}_{C'_0}} [2] \\
&=& 0.
\end{array}
\]
This verifies Equation~\eqref{eqn:FPF} for $\mathcal{P} =  \IC( \1_{C_3})$ and $C'=C'_0$ in Case \ref{geocase-sub}(v).

\item
Take $C'=C'_3$, in which case $C=C_3$.
The left-hand side of  Equation~\eqref{eqn:FPF} is
\[
\begin{array}{rcl}
&&\hskip-1cm \trace_{s} \left(\NEvs_{C_3}[\dim{C_3}]\ \IC(\1_{C_3})\right) \\
&=& \trace_s \1_{\Lambda^\text{sreg}_{C_3}}[4]\\
&=& 1.
\end{array}
\]
The right-hand side of Equation~\eqref{eqn:FPF} is 
\[
\begin{array}{rcl}
&&\hskip-1cm \trace_{s} \left(\NEvs_{C'_3}[\dim{C'_3}]\ \res\ \IC(\1_{C_3})\right) \\
&=& \trace_{s} \left(\NEvs_{C'_3}[2] \IC(\1_{C'_3})[2]\right) \\
&=& \trace_{s} \1_{\Lambda^\text{sreg}_{C'_3}}[4]  \\
&=& 1.
\end{array}
\]
This verifies Equation~\eqref{eqn:FPF} for $\mathcal{P} =  \IC( \1_{C_3})$ and $C'=C'_3$ in Case \ref{geocase-sub}(v).
\end{enumerate}
This verifies  Equation~\eqref{eqn:FPF} for $\mathcal{P} =  \IC( \1_{C_3})$ in Case \ref{geocase-sub}(v).

\item 
Take $\mathcal{P} =  \IC( \varrho_{C_3})$ and recall that $\res\ \IC(\varrho_{C_3}) = \IC(\rho_{C'_3})[2]\oplus \IC(\rho^2_{C'_3})[2]  \oplus \IC(\1_{C'_0})[2]$.
\begin{enumerate}
\item
Take $C'=C'_0$, so $C=C_0$. 
Then the left-hand side of  Equation~\eqref{eqn:FPF} is
\[
\begin{array}{rcl}
&&\hskip-1cm \trace_{s} \left(\NEvs_{C_0}[\dim{C_0}]\ \IC(\varrho_{C_3})\right) \\
&=& \trace_s 0_{\Lambda^\text{sreg}_{C_0}}[0]\\
&=& 0.
\end{array}
\]
The right-hand side of Equation~\eqref{eqn:FPF} is 
\[
\begin{array}{rcl}
&&\hskip-1cm \trace_{s} \left(\NEvs_{C'_0}[\dim{C'_0}]\ \res\ \IC(\1_{C_3})\right) \\
&=& \trace_{s} \left(\NEvs_{C'_0}[0]\left(\IC(\rho_{C'_3})[2]\oplus \IC(\rho^2_{C'_3})[2]  \oplus \IC(\1_{C'_0})[2]\right) \right) \\
&=&   \trace_{s} \rho_{\Lambda^\text{sreg}_{C'_0}} [2]  + \trace_{s}\rho^2_{\Lambda^\text{sreg}_{C'_0}} [2]  + \trace_{s} \1_{\Lambda^\text{sreg}_{C'_0}} [2] \\
&=& \theta + \theta^2 +1 =0.
\end{array}
\]
Here we used the fact that the image of $s$ in $A^\ABV_{C'_0} = S_3$ is of order $3$.
This verifies Equation~\eqref{eqn:FPF} for $\mathcal{P} =  \IC( \varrho_{C_3})$ and $C'=C'_0$ in Case \ref{geocase-sub}(v).

\item
Take $C'=C'_3$, in which case $C=C_3$.
The left-hand side of  Equation~\eqref{eqn:FPF} is
\[
\begin{array}{rcl}
&&\hskip-1cm \trace_{s} \left(\NEvs_{C_3}[\dim{C_3}]\ \IC(\varrho_{C_3})\right) \\
&=& \trace_s \varrho_{\Lambda^\text{sreg}_{C_3}}[4]\\
&=& -1.
\end{array}
\]
Here we used the fact that the image of $s$ in $A^\ABV_{C_3} = S_3$ is of order $3$ and the trace of the reflection representation of $S_3$ at such elements is $-1$.
The right-hand side of Equation~\eqref{eqn:FPF} is 
\[
\begin{array}{rcl}
&&\hskip-1cm \trace_{s} \left(\NEvs_{C'_3}[\dim{C'_3}]\ \res\ \IC(\varrho_{C_3})\right) \\
&=& \trace_{s} \left(\NEvs_{C'_3}[2] \left(\IC(\rho_{C'_3})[2]\oplus \IC(\rho^2_{C'_3})[2]  \oplus \IC(\1_{C'_0})[2] \right)\right) \\
&=& \trace_{s} \rho_{\Lambda^\text{sreg}_{C'_3}}[4] + \trace_{s} \rho^2_{\Lambda^\text{sreg}_{C'_3}}[4] + \trace_{s} 0_{\Lambda^\text{sreg}_{C'_3}}[4]  \\
&=& \theta+\theta^2 =-1.
\end{array}
\]
Here we used the fact that the image of $s$ in $A^\ABV_{C'_3} = S_3$ is of order $3$.
This verifies Equation~\eqref{eqn:FPF} for $\mathcal{P} =  \IC( \varrho_{C_3})$ and $C'=C'_3$ in Case \ref{geocase-sub}(v).
\end{enumerate}
This verifies  Equation~\eqref{eqn:FPF} for $\mathcal{P} =  \IC( \varrho_{C_3})$ in Case \ref{geocase-sub}(v).

\item 
Take $\mathcal{P} =  \IC( \varepsilon_{C_3})$ and recall that $\res\ \IC(\varepsilon_{C_3}) = \mathcal{F}_5[2]$.
\begin{enumerate}
\item
Take $C'=C'_0$, so $C=C_0$. 
Then the left-hand side of  Equation~\eqref{eqn:FPF} is
\[
\begin{array}{rcl}
&&\hskip-1cm \trace_{s} \left(\NEvs_{C_0}[\dim{C_0}]\ \IC(\varepsilon_{C_3})\right) \\
&=& \trace_s \varepsilon_{\Lambda^\text{sreg}_{C_0}}[0]\\
&=& 1.
\end{array}
\]
Here we used the fact that the image of $s$ in $A^\ABV_{C_0} = S_3$ is of order $3$ and the sign of such elements is $1$.
The right-hand side of Equation~\eqref{eqn:FPF} is 
\[
\begin{array}{rcl}
&&\hskip-1cm \trace_{s} \left(\NEvs_{C'_0}[\dim{C'_0}]\ \res\ \IC(\1_{C_3})\right) \\
&=& \trace_{s} \left(\NEvs_{C'_0}[0]\mathcal{F}_5[2] \right) \\
&=&  \trace_{s} \1_{\Lambda^\text{sreg}_{C'_0}}[2] \\
&=& 1
\end{array}
\]
This verifies Equation~\eqref{eqn:FPF} for $\mathcal{P} =  \IC( \varepsilon_{C_3})$ and $C'=C'_0$ in Case \ref{geocase-sub}(v).

\item
Take $C'=C'_3$, in which case $C=C_3$.
The left-hand side of  Equation~\eqref{eqn:FPF} is
\[
\begin{array}{rcl}
&&\hskip-1cm \trace_{s} \left(\NEvs_{C_3}[\dim{C_3}]\ \IC(\varepsilon_{C_3})\right) \\
&=& \trace_s \varepsilon_{\Lambda^\text{sreg}_{C_3}}[4]\\
&=& 1.
\end{array}
\]
Here we used the fact that the image of $s$ in $A^\ABV_{C_3} = S_3$ is of order $3$.
The right-hand side of Equation~\eqref{eqn:FPF} is 
\[
\begin{array}{rcl}
&&\hskip-1cm \trace_{s} \left(\NEvs_{C'_3}[\dim{C'_3}]\ \res\ \IC(\varepsilon_{C_3})\right) \\
&=& \trace_{s} \left(\NEvs_{C'_3}[2]\mathcal{F}_5[2]\right) \\
&=&  \trace_{s} \1_{\Lambda^\text{sreg}_{C'_3}}[2]  \\
&=& 1.
\end{array}
\]
This verifies Equation~\eqref{eqn:FPF} for $\mathcal{P} =  \IC( \varepsilon_{C_3})$ and $C'=C'_3$ in Case \ref{geocase-sub}(v).
\end{enumerate}
This verifies  Equation~\eqref{eqn:FPF} for $\mathcal{P} =  \IC( \varepsilon_{C_3})$ in Case \ref{geocase-sub}(v).

\end{enumerate}
This concludes the proof of Equation~\eqref{eqn:FPF} in Case \ref{geocase-sub}(v).
\fi
\end{enumerate}

\end{proof}

\subsection{Geometric lifting of stable distributions}\label{ssec:lifting}

From Section~\ref{sec:endoscopy}, recall Definition~\ref{def:Thetaphis-endo}:
\[
\Theta_{\phi,s}^{G^\delta}
\ceq 
e(\delta) 
\sum_{\pi \in \Pi^\ABV_{\phi}(G^\delta(F))} \hskip-0pt  \trace_{s} \left( \NEvs_{\phi}[\dim(\phi)] \mathcal{P}(\pi)[-\dim(\pi)] \right) \  \Theta_{\pi},
\]
where $(G,s,\xi)$ is an endoscopic triple for $G_2$ and $\delta \in Z^1(G,F)$ is a pure inner form of $G$.
Recall also that this distribution can be written in the form
\[
\Theta_{\phi,s}^{G^\delta}
\ceq 
e(\delta) 
\sum_{\pi \in \Pi^\ABV_{\phi}(G^\delta(F))}  (-1)^{\dim(\phi)-\dim(\pi)} \langle s, \pi\rangle\  \Theta_{\pi}.
\]
\begin{remark}\label{rem:relevant}
If the Langlands parameter $\phi : W'_F\to \Lgroup{G}$ is not relevant to $G(F)$, in the sense of \cite{Borel:Automorphic}, then $\Theta_{\phi,s}^{G^\delta}$ may be trivial; for example, 
$
\Theta_{\phi_{\ref{ssec:PGL3}.0}}^{\PGL_3^\delta} =0
$
for non-trivial $\delta$. On the other hand, $\Theta_{\phi,s}^{G^\delta}$ can be non-trivial even when $\phi : W'_F\to \Lgroup{G}$ is not relevant to $G(F)$; for example, although $\phi_{\ref{ssec:PGL3}.3a}$ is not relevant to ${\PGL_3^\delta}$ for non-trivial $\delta$,
$
\Theta_{\phi_{\ref{ssec:PGL3}.3a},s}^{\PGL_3^\delta} = e(\delta) \Theta_{\chi_{\PGL_3^\delta(F)}},
$
which is non-trivial.
\end{remark}

\begin{definition}\label{def:lifting}
Let $(G,s,\xi)$ be an endoscopic triple for $G_2$.
Let $\phi : W'_F \to \Lgroup{G}$ be an unramified Langlands parameter.
We define
\[
\Lift^{}_{(G,s,\xi)} \Theta_{\phi}^{G}
\ceq
\sum_{\pi \in \Pi^\ABV_{\xi\circ\phi}(G_2(F))} \hskip-20pt \trace_{s} \left( \NEvs_{\phi} [\dim(\phi)]\ \res\,  \mathcal{P}(\pi)[-\dim(\pi)] \right)
 \  \Theta_{\pi},
\]
where $\res : D_{Z_\dualgroup{G_2}(\xi\circ\phi)}(V_{\xi\circ\lambda}) \to D_{Z_\dualgroup{G}(\phi)}(V_{\lambda})$ is the equivariant functor defined by $\res\, \mathcal{F} = \mathcal{F}\vert_{V_{\lambda}}$.
If the context make it clear what the endoscopic triple is, we will sometimes use the abbreviated notation $\Lift_{G}^{G_2}$ for $\Lift^{}_{(G,s,\xi)}$.
Let $\delta \in Z^1(F,G)$ be a pure inner form of $G$.
Then $(G,1,\id)$ is an endoscopic triple for $G^\delta$ over $F$ and, extending the definition above,
\[
\Lift_G^{G^\delta} \Theta_{\phi}^{G}
\ceq
e(\delta) 
\mathop{\sum}\limits_{\pi \in \Pi^\ABV_{\xi\circ\phi}(G^\delta(F))}  \trace_{1} \left( \NEvs_{\phi} [\dim(\phi)]\ \res\,  \mathcal{P}(\pi)[-\dim(\pi)] \right)
 \  \Theta_{\pi} .
\]
\end{definition}

\begin{conjecture}\label{conjecture:transfer}
Let $(G,s,\xi)$ be an endoscopic triple for $G_2$ and let $\phi : W'_F \to \Lgroup{G}$ be a unramified Langlands parameter that is $\xi$-conormal.
The Langlands-Shelstad transfer of $\Theta^{G}_{\phi}$ from $G(F)$ to $G_2(F)$ is $\Theta_{\xi\circ \phi,s}^{G_2}$.
Now let $\delta \in Z^1(F,G)$ be a pure inner form for $G$ over $F$ and suppose $\phi : W'_F \to \Lgroup{G}$ is also relevant to $G^\delta(F)$. 
Then the Jacquet-Langlands transfer of $\Theta^{G}_{\phi}$ from $G(F)$ to $G^\delta(F)$ is $e(\delta)  \Theta_{\phi}^{G^\delta}$
\end{conjecture}

\begin{theorem}\label{thm:lifting}
Let $(G,s,\xi)$ be an endoscopic triple for $G_2$.
Let $\phi : W'_F \to \Lgroup{G}$ be a unramified Langlands parameter that is $\xi$-conormal.
Then
\[
\Lift_{(G,s,\xi)}^{G} \Theta^{G}_{\phi} = \Theta^{}_{\xi\circ \phi,s}.
\]
Let $\delta \in Z^1(F,G)$ be a pure inner form of $G$ and suppose also that $\phi$ is relevant to $G^\delta$. Then
\[
\Lift_{G}^{G^\delta} \Theta_{\phi}^{G}
=
e(\delta)  
\Theta_{\phi}^{G^\delta}.
\]
\end{theorem}

\begin{proof}
Theorem~\ref{thm:VFPF}, we have
\begin{equation}\label{eqn:TrEvRes-intro}
\trace_{s} \left(\NEvs_{\phi}[\dim(\phi)]\ \mathcal{P}\right) =  \trace_{s} \left(\NEvs_{\xi\circ\phi}[\dim(\xi\circ\phi)] \res \mathcal{P}\right),
\end{equation}
for every simple object $\mathcal{P}$ in $\Perv_{Z_\dualgroup{G_2}(\xi\circ\lambda)}(V_{\xi\circ\lambda})$. 
Thus,
\[
\begin{array}{rcl r}
\Lift^{}_{(G,s,\xi)} \Theta^{G}_{\phi} 
&:=& 
\mathop{\sum}\limits_{\pi \in \Pi^\ABV_{\xi\circ\phi}(G_2(F))} \hskip-0pt  \trace_{s} \left( \NEvs_{\phi}[\dim(\phi)] \res\, \mathcal{P}(\pi)[-\dim(\pi)] \right)
 \  \Theta_{\pi} & \\
&=& 
\mathop{\sum}\limits_{\pi \in \Pi^\ABV_{\xi\circ\phi}(G_2(F))} \hskip-0pt  \trace_{s} \left( \NEvs_{\xi\circ\phi}[\dim(\xi\circ\phi)] \mathcal{P}(\pi)[-\dim(\pi)] \right)
 \  \Theta_{\pi} &  \\
&=:& 
\Theta^{}_{\xi\circ\phi,s} . & 
\end{array}
\]
Directly from the definitions we see
\[
\begin{array}[b]{rcl}
\Lift_{G}^{G^\delta} \Theta_{\phi}^{G}
&\ceq&
e(\delta) 
\sum_{\pi \in \Pi^\ABV_{\xi\circ\phi}(G^\delta(F))}  \trace_{1} \left( \NEvs_{\phi} [\dim(\phi)]\ \res\,  \mathcal{P}(\pi)[-\dim(\pi)] \right)
 \  \Theta_{\pi} \\
& = &
e(\delta) 
\sum_{\pi \in \Pi^\ABV_{\xi\circ\phi}(G^\delta(F))} \trace_{1} \left( \NEvs_{\phi} [\dim(\phi)]\ \mathcal{P}(\pi)[-\dim(\pi)] \right)
 \  \Theta_{\pi} \\
&=&
e(\delta)  
\Theta_{\phi}^{G^\delta}.
\end{array}\qedhere
\]
\end{proof}




\begin{definition}\label{def:s-conormal}
Let $\phi: W'_F \to \Lgroup{G_2}$ be an unramified Langlands parameter for $G_2(F)$.
If $s\in \mathcal{S}^\ABV_\phi$ then $\phi = \xi\circ \phi^s$ for an endoscopic triple $(G,s,\xi)$ and a Langlands parameter $\phi^s : W'_F \to \Lgroup{G}$.
If $\phi$ has the property that $\phi^s$ is $\xi$-conormal, we say that $\phi$ is $s$-conormal. 
If $\phi$ is of Arthur type then $\phi$ is $s$-conormal for every $s\in \mathcal{S}^\ABV_\phi$.
\end{definition}

\begin{theorem}\label{thm:EC}
Let $\phi: W'_F \to \Lgroup{G_2}$ be an unramified Langlands parameter that is $s$-conormal for every $s\in \mathcal{S}^\ABV_\phi$. 
If $\pi$ is any unipotent representation of $G_2(F)$  then the distribution character $\Theta_\pi$ may be expressed as a linear combination of the distributions $\Lift^{G_2}_{(G,s,\xi)}\Theta^{G}_{\phi^s}$, letting $\phi$ range over Langlands parameters with the same infinitesimal parameter as $\pi$ and letting $s$ range over $\mathcal{S}^\ABV_\phi $.
And if $\Pi^\ABV_\phi(G_2(F)) \to \widehat{A^\ABV_\phi}$ is a bijection then,
\[
\Theta_{\pi} 
= \sum_{(G,s,\xi)} (-1)^{\dim(\phi^s)-\dim(\pi)} \frac{\overline{\langle s,\pi\rangle}}{\abs{Z_{A_\phi}(s)}} \ \Lift^{}_{(G,s,\xi)} \Theta^{G}_{\phi^s},
\qquad\qquad  \forall \pi \in \Pi^\ABV_\phi(G_2(F)),
\]
where the sum is taken over equivalence classes of endoscopic triples $(G,s,\xi)$ with $s \in \mathcal{S}^\ABV_\phi$ and where we identify $s$ with its image under $\mathcal{S}^\ABV_\phi \to A^\ABV_\phi$ in the calculation of $Z_{A_\phi}(s)$.
\end{theorem}

\begin{proof} 
Distribution characters for equivalence classes of irreducible admissible representations of $G_2(F)$ are linearly independent by \cite{BZ}. For each unramified Langlands parameter $\phi$ for $G_2(F)$, the number of equivalence classes of endoscopic triples for $(G,s,\xi)$ for $G_2(F)$ with $s\in  \mathcal{S}^\ABV_\phi$ is equal to the number of irreducible representations of $A^\ABV_\phi$. This makes it possible to determine the characters $\langle s,\pi\rangle $ appearing in the sum $\Theta_{\phi,s} =\sum_{\pi} (-1)^{\dim(\phi)-\dim(\pi)} \langle s,\pi\rangle\ \Theta_\pi$. The rest of the theorem is a direct consequence of \ref{G2:basis}.
\end{proof}

\end{document}